\documentclass{amsproc}
\usepackage[margin=1.5in]{geometry}
\usepackage{amssymb, comment}
\usepackage{tikz, mathtools}
\usetikzlibrary{arrows}
\usetikzlibrary{decorations.pathreplacing}

\newtheorem{theorem}{Theorem}[section]
\newtheorem{lemma}[theorem]{Lemma}
\newtheorem{proposition}[theorem]{Proposition}

\theoremstyle{definition}
\newtheorem{definition}[theorem]{Definition}
\newtheorem{conjecture}[theorem]{Conjecture}

\newtheorem{example}[theorem]{Example}
\newtheorem{nc}[theorem]{Notational Conventions}

\theoremstyle{remark}
\newtheorem{remark}[theorem]{Remark}

\numberwithin{equation}{section}



\newcommand{\R}{{\mathbb R}}

\newcommand{\Z}{{\mathbb Z}}

\newcommand{\bbZ}{{\mathbb Z}}

\newcommand{\rank}{\operatorname{rank}}

\newcommand{\Ker}{\operatorname{Ker}}
\newcommand{\diag}{\operatorname{diag}}

\newcommand{\Sym}{{\rm S}}

\newcommand{\sgn}{{\rm sgn}}
\newcommand{\M}{{\rm Mat}}

\newcommand{\In}{{\rm In}}
\newcommand{\Ic}{{\rm Ic}}

\newcommand{\nodeDist}{2}
\newcommand{\labelDist}{0.1}
\newcommand{\indeg}{{\rm indeg}}
\newcommand{\outdeg}{{\rm outdeg}}
\newcommand{\supp}{{\rm supp}}
\newcommand{\nullity}{{\rm null}}
\newcommand{\gr}{{\rm gr}}
\newcommand{\mon}{{\rm mon}}
\newcommand{\src}{{\rm src}}
\newcommand{\tar}{{\rm tar}}
\newcommand{\uni}{{\rm Eul}}
\newcommand{\interval}{\longleftrightarrow}

\newcommand\imCMsym[4][\mathord]{%
  \DeclareFontFamily{U} {#2}{}
  \DeclareFontShape{U}{#2}{m}{n}{
   <-6> #25
    <6-7> #26
    <7-8> #27
    <8-9> #28
    <9-10> #29
    <10-12> #210
    <12-> #212}{}
  \DeclareSymbolFont{CM#2} {U} {#2}{m}{n}
  \DeclareMathSymbol{#4}{#1}{CM#2}{#3}
}

\imCMsym[\mathop]{cmsy}{113}{\CMamalg}
\imCMsym[\mathop]{cmex}{96}{\CMcoprod}

\title{A graph-theoretic approach to a conjecture of Dixon and Pressman}

\author[Matthew Brassil and Zinovy Reichstein]{Matthew Brassil 
and Zinovy Reichstein}

\address
{Department of Mathematics \\
University of British Columbia \\
Vancouver
\\
CANADA}

\email{mbrassil@math.ubc.ca, reichst@math.ubc.ca}

\thanks
{Matthew Brassil was partially supported by a Graduate Research Fellowship from the University of British Columbia.
Zinovy Reichstein was partially supported by
National Sciences and Engineering Research Council of
Canada Discovery grant 253424-2017.}


\begin{document}

\begin{abstract} Given $n \times n$ matrices, $A_1, \dots, A_k$, consider the linear operator \[ L(A_1,\dots,A_k) \colon \M_n \to \M_n \] 
given by
$L(A_1,\dots,A_k)(A_{k+1})=\sum_{\sigma\in S_{k+1}}\sgn(\sigma)A_{\sigma(1)}A_{\sigma(2)}\cdots A_{\sigma(k+1)}$. 
The Amitsur-Levitzki theorem asserts that $L(A_1, \ldots, A_k)$ is identically $0$ for every $k \geqslant 2n-1$.
Dixon and Pressman conjectured that if $k$ is an even number between $2$ and $2n - 2$, then
the kernel of $L(A_1, \ldots, A_k)$ is of dimension $k$ for
$A_1,\dots,A_k\in \M_n(\R)$ in general position. We prove this conjecture using graph-theoretic techniques.
\end{abstract}

\maketitle

\section{Introduction}
Recall that the standard polynomial $[A_1,\dots,A_m]$ in $m$ variables is defined as
\[ [A_1,\dots,A_m] = \sum_{\sigma\in S_m} \sgn(\sigma)A_{\sigma(1)}\cdots A_{\sigma(m)}\ .\]
The celebrated theorem of Amitsur and Levitzki~\cite{al} asserts that $[A_1,\dots,A_m] = 0$ 
for any $m \geqslant 2n$ and any $n \times n$-matrices $A_1, \dots, A_m \in \M_n(\Lambda)$ over a commutative ring $\Lambda$.
The original proof in~\cite{al} is quite involved. Simpler proofs have since been given by Swan~\cite{swan, swan2}, Razmyslov~\cite{razmyslov}, Rosset~\cite{rosset} and, most recently, Procesi~\cite{procesi-al}.

Let $F$ be a field. For a $k$-tuple of matrices $(A_1,\dots,A_k)$ with $A_i\in \M_n(F)$, let 
\[ L(A_1,\dots,A_k) \colon \M_n(F) \to \M_n(F) \]
be the linear operator
given by
\[L(A_1,\dots,A_k)(X):=[A_1,\dots,A_k,X]\ .\]
Dixon and Pressman investigated the kernel of this operator in \cite{dp}. When $k=1$, the kernel of $L(A_1)$ is the centralizer of $A_1$. When $k \geqslant  2n-1$, $L(A_1,\dots,A_k)$ is identically zero, by the Amitsur-Levitzki theorem. 

\begin{conjecture}[Dixon, Pressman~\cite{dp}] \label{conj.main}
Suppose that $2\leqslant k\leqslant 2n-2$. Then for $A_1,\dots,A_k\in \M_n(\R)$ in general position, the nullity $d$ of $L(A_1,\dots,A_k)$ 
is given by 

\smallskip
(i) $d=k$, if $k$ is even,

\smallskip
(ii) $d=k+1$, if $k$ is odd and $n$ is even, and 

\smallskip
(iii) $d=k+2$, if both $k$ and $n$ are odd.
\end{conjecture}

Here, as usual, $\R$ is the field of real numbers and the nullity of a linear transformation is the dimension of its kernel.
Dixon and Pressman showed that $d \geqslant k$, $k+1$ and $k+2$ in cases (i), (ii), and (iii), respectively, and verified computationally
that equality holds for small values of $n$ and $k$. 
Note that one may view Conjecture~\ref{conj.main} and the Amitsur-Levitsky theorem as pointing in opposite directions:
Conjecture~\ref{conj.main} gives an upper bound on the generic nullity of $L(A_1, \ldots, A_k)$ for $2 \leqslant k \leqslant 2n-2$, 
whereas the Amitsur-Levitsky theorem gives a lower bound for $k \geqslant 2n -1$. 

The purpose of this paper is to prove Conjecture~\ref{conj.main} in case (i). Our main result is the following.

\begin{theorem}\label{main_thm}
Let $k = 2r$ be a positive even integer and $F$ be an infinite field whose characteristic does not divide $2(2r+1)r!$. Assume that $n > r$. Then for $A_1,A_2,\dots,A_k\in \M_n(F)$ in general position, the nullity of $L(A_1,A_2,\dots,A_k)$ is $k$.
\end{theorem}

Our proof will rely on graph-theoretic techniques. To motivate it, let us briefly recall Swan's proof of the Amitsur-Levitsky theorem. 
Since the standard polynomial $[A_1, \ldots, A_m]$ is multi-linear in $A_1, \ldots, A_m$, it suffices to show that
$[E_{a_1 b_1}, \ldots, E_{a_m b_m}] = 0$ for any choice of $a_1, b_1, \ldots, a_m, b_m \in \{ 1, \ldots, n \}$, as long as $m \geqslant 2n$.
Here $E_{a b}$ denotes the elementary matrix with $1$ in the $(a, b)$-position and $0$s elsewhere. As we expand $[E_{a_1 b_1}, \ldots, E_{a_m b_m}]$,
the term $\sgn(\sigma) E_{a_{\sigma(1)} b_{\sigma(1)}} \ldots E_{a_{\sigma(m)} b_{\sigma(m)}}$ contributes $\sgn(\sigma) E_{a_{\sigma(1)} b_{\sigma(m)}}$
to the sum if 
\begin{equation} \label{e.path} b_{\sigma(1)} = a_{\sigma(2)}, \; \; b_{\sigma(2)} = a_{\sigma(3)}, \; \ldots \; , b_{\sigma(m-1)} = a_{\sigma(m)}, 
\end{equation}
and $0$ otherwise. Conditions~\eqref{e.path} can be conveniently rephrased in graph-theoretic terms. Let $G$ be the directed
graph with $n$ vertices, $1, \ldots, n$ and $m$ edges, $e_1 = (a_{1}, b_1), \ldots, e_m = (a_m, b_m)$. Then conditions~\eqref{e.path} hold if and only if
$e_{\sigma(1)}, \ldots, e_{\sigma(m)}$ form an Eulerian path on $G$. We will say that this Eulerian path is even if $\sigma$ is an even permutation and
odd otherwise. This way the Amitsur-Levitsky theorem reduces to the following graph-theoretic assertion. 

\smallskip
\begin{theorem} {\rm (}Swan~\cite{swan, swan2}{\rm )} \label{thm.swan}
Let $G$ be a directed graph with $n$ vertices and $m$ edges. Let $a$ and $b$ be two of the vertices (not necessarily distinct). If $m \geqslant 2n$,
then the number of even Eulerian paths from $a$ to $b$ equals the number of odd Eulerian paths from $a$ to $b$.
\end{theorem}

\smallskip
If one were to use a similar approach to prove Conjecture~\ref{conj.main}, one would set
\begin{equation} \label{e.general-matrix} A_{\ell} = \sum_{a, b = 1}^n  x_{ab}^{(\ell)} E_{a b} , \end{equation}
for $\ell = 1, \ldots, k$. Here $x_{ab}^{(\ell)}$ are $k n^2$
independent variables. Each entry of the $n^2 \times n^2$ matrix of $L = L(A_1, \ldots, A_k)$ 
is then a multilinear polynomial of degree $k$ in the groups of variables, $\{x_{ab}^{(1)} \}, \ldots, \{ x_{ab}^{(k)} \}$.
(Here we identify the linear transformation $L(A_1, \ldots, A_k) \colon \M_n \to \M_n$
with its matrix in the standard basis $E_{ab}$ of $\M_n$.)
The coefficient of the monomial $x_{a_1 b_1}^{(1)} \cdot \ldots \cdot x_{a_k b_k}^{(k)}$ in a given position 
in $L$ can again be computed as 
the signed sum of Eulerian paths on a certain graph. However, for $k \leqslant 2n-2$, these signed sums will no longer 
be identically $0$. To prove Conjecture~\ref{conj.main}(i) in this way, one would need to assemble these coefficients into 
the $n^2 \times n^2$ matrix $L$ with polynomial entries, then show that the nullity of $L$ over the field $F(x_{ab}^{(\ell)})$ 
is $k$ (or equivalently, is $\leqslant k$). We are not able to carry out the computations directly in this setting; 
the matrix $L$ is too complicated.  To prove Theorem~\ref{main_thm} we will modify this approach in the following ways.

\smallskip
(1) We will specialize the matrices $A_{\ell}$ by setting some of the variables $x_{ab}^{(\ell)}$ equal to $0$. 
For the purpose of showing that $\nullity(L) \leqslant k$, this is sufficient. 
In fact, we will set $n^2-n$ entries of each $A_l$ equal to $0$; the other $n$ entries will remain independent variables.

\smallskip
(2) We will choose $A_1, \ldots, A_k$ so that  $L(A_1, \dots, A_k)$ decomposes as a direct sum 
\[ L(A_1, \dots, A_k) = L_0 \oplus L_1 \oplus \ldots \oplus L_{n-1}, \]
where each $L_i$ is represented by an $n \times n$ matrix. This simplifies our analysis of $L$ and 
reduces the problem to showing that $\nullity(L_0) + \nullity(L_1) + \ldots + \nullity(L_{n-1}) \leqslant k$.
The specific matrices we will use are described in Sections~\ref{sect.first-reductions} and~\ref{specialization}.

\smallskip
(3) To get a better handle on the nullities of $L_0, \ldots, L_{n-1}$, we will replace each $L_j$ by its
``matrix of initial coefficients" $\Ic(L_j)$ with respect to a certain lexicographic monomial order on 
the variables $x_{\ell,\alpha}$; see Section~\ref{sect.initial}. This will further simplify the computations in two ways.
First, the entries of $\Ic(L_j)$ will be integers, rather than polynomials. These integers will be obtained by
counting Eulerian paths on certain graphs, as in Swan's argument. Secondly, 
passing from $L_j$ to $\Ic(L_j)$ will allow us to focus only on the (rather special) 
graphs corresponding to leading monomials. 

\smallskip
We will classify these ``maximal graphs" in Sections~\ref{sect.maximal} 
and~\ref{sect.maximal_2} and complete the proof of Theorem~\ref{main_thm} in Sections~\ref{section_j_nonzero} and~\ref{section_j_0}.
The last part of the proof will rely on the computations of signed counts of Eulerian paths 
in Section~\ref{prelims}. The overall structure of the paper is shown in the flowchart below.
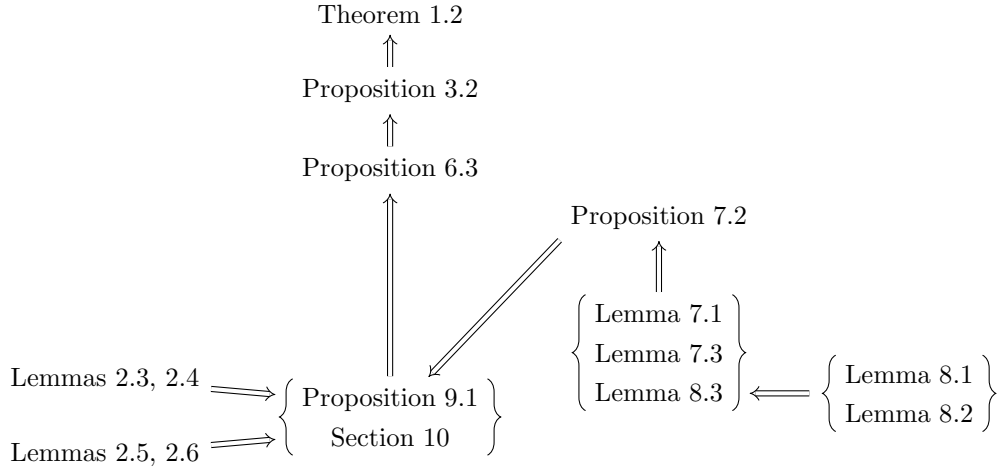
\begin{equation*}
\renewcommand{\nodeDist}{0.5}
\begin{tikzpicture}[baseline,vertex/.style={anchor=base, circle, fill,minimum size=4pt, inner sep=0pt, outer sep=0pt},auto,
                               edge/.style={->,>=latex, shorten > = 5pt,shorten < = 5pt},
position/.style args={#1:#2 from #3}{
    at=(#3.(#1-90)), anchor=#1+90, shift=(#1-90:#2)
}]

  \node (t12) {Theorem~\ref{main_thm}};
  \node[position=0:{\nodeDist} from t12] (p32) {Proposition~\ref{prop.red2}};
 \node[position=0:{\nodeDist} from p32] (p63) {Proposition~\ref{Ic_L_j_ker}};
  \node[position=80:{\nodeDist*2} from p63] (p72) {Proposition~\ref{max_graph_decomp}};
  \node[align=center, position=0:{\nodeDist*1.5} from p72] (l73) {Lemma~\ref{max_subgraphs}\\[3] Lemma~\ref{max_connecting_path}\\[3] Lemma~\ref{subgraph_connected}};
  \node[align=center, position=0:{\nodeDist*5} from p63] (p91) {Proposition~\ref{Ic_L_j_ker_j_neq_0}\\[3] Section~\ref{section_j_0}};

  \draw [decorate,decoration={brace,amplitude=5,mirror}] (p91.north west) node[below=3] (bnw) {} -- (p91.south west) node[above=3] (bsw) {};
  \draw [decorate,decoration={brace,amplitude=5}] (p91.north east) -- (p91.south east);
  \draw [decorate,decoration={brace,amplitude=5,mirror}] (l73.north west) -- (l73.south west);
  \draw [decorate,decoration={brace,amplitude=5}] (l73.north east) -- (l73.south east) node[above=3] (bse) {};

  \node[align=center, position=90:{\nodeDist*2.5} from bse] (l81) {Lemma~\ref{reachable_H}\\[3] Lemma~\ref{max_t}};

  \draw [decorate,decoration={brace,amplitude=5,mirror}] (l81.north west) -- (l81.south west);
  \draw [decorate,decoration={brace,amplitude=5}] (l81.north east) --  (l81.south east);

  \draw[-implies,shorten >=1, shorten <=1, double equal sign distance] (p32) to (t12);
  \draw[-implies,shorten >=1, shorten <=1, double equal sign distance] (p63) to (p32);
  \draw[-implies,shorten >=1, shorten <=1, double equal sign distance] (p91) to (p63);
  \draw[-implies,shorten >=1, shorten <=1, double equal sign distance] (p72.south west) to (p91);
  \draw[-implies,shorten >=1, shorten <=1, double equal sign distance] (l73) to (p72);
  \draw[-implies,shorten >=3, shorten <=10, double equal sign distance] (l81) to (bse);

  \node[position=-95:{\nodeDist*2} from bnw] (l24) {Lemmas~\ref{diag_Eulerian_sum_1},~\ref{diag_Eulerian_sum_2}};
  \node[position=-85:{\nodeDist*2} from bsw] (l26) {Lemmas~\ref{Eulerian_sum_1},~\ref{Eulerian_sum_2}};
  
  \draw[-implies,shorten >=3, shorten <=1, double equal sign distance] (l24) to (bnw);
  \draw[-implies,shorten >=3, shorten <=1, double equal sign distance] (l26) to (bsw);
\end{tikzpicture}
\end{equation*}

\section{Preliminaries on graphs and Eulerian paths}
\label{prelims}

Throughout this paper our graphs will all be directed with labeled edges and vertices. 
An Eulerian path in a graph $\Gamma$ is a path which visits every edge exactly once.  We will denote by $\uni_{a}(\Gamma)$ the set of Eulerian paths on $\Gamma$ which begin at a vertex $a$. It is easy to see that any two paths in $\uni_a(\Gamma)$ terminate at the same vertex.

For an edge $\underset{a}\bullet\underset{e}\rightarrow\underset{b}\bullet$ appearing in a graph $\Gamma$ 
we define $\src_\Gamma(e)=a$ and $\tar_\Gamma(e)=b$ to be the source and target vertices of the edge $e$ respectively. 
We define the outdegree $\outdeg_{\Gamma}(v)$ to be the number of edges in $\Gamma$ whose source vertex is $v$, and the indegree $\indeg_{\Gamma}(v)$ 
to be the number of edges in $\Gamma$ whose target vertex is $v$. When the graph $\Gamma$ is clear from the context we will abbreviate these terms 
as $\src(e)$, $\tar(e)$, $\outdeg(v)$, and $\indeg(v)$.

An Eulerian path beginning and ending at the same vertex is known as an Eulerian circuit. 
The following fundamental theorem, due to Euler, is usually stated in terms of Eulerian circuits. 
In the sequel we will need a variant in terms of Eulerian paths. 


\begin{theorem} \label{eulerian_unic} 
Let $a$, $b$ be vertices of $\Gamma$, not necessarily distinct.
There exists an Eulerian path from $a$ to $b$ on $\Gamma$, if and only if $\Gamma$ is connected and 
\begin{equation}\label{eulerian_degrees}
\begin{split}
\outdeg_{\Gamma}(v)&=\indeg_{\Gamma}(v)\ , \text{ for all $v\neq a, b$}\ ,\\
\outdeg_{\Gamma}(a)&=\indeg_{\Gamma}(a)+1\ , \text{ if $a\neq b$}\ ,\\
\outdeg_{\Gamma}({b})&=\indeg_{\Gamma}({b})-1\ , \text{ if $a\neq b$}\ ,\\
\outdeg_{\Gamma}(a)&=\indeg_{\Gamma}(a)\ , \text{ if $a= b$}\ .
\end{split}
\end{equation}
\end{theorem}

\begin{proof} If $a = b$, then this is the usual form of Euler's Theorem; see, e.g.,~\cite[Theorems 12, 13]{bollobas}. 

If $a \neq b$, let $\Gamma'$ be the graph obtained from $\Gamma$ by adding an edge from $b$ to $a$.  Eulerian paths $w$ from $a$ to $b$ on $\Gamma$ 
are in bijective correspondence with Eulerian circuits $w'$ on $\Gamma'$. Indeed, given $w$, we obtain $w'$ by appending  $e$ at the end. 
Conversely, given $w'$, after cyclically permuting the edges, we may assume that $e$ is the last edge in $w'$. Now $w$ is obtained from $w'$ by removing $e$. 

Finally, observe that conditions~\eqref{eulerian_degrees} are equivalent to $\outdeg_{\Gamma'}(v) = \indeg_{\Gamma'}(v)$ for every vertex $v$. Thus
Theorem~\ref{eulerian_unic} reduces to Euler's theorem for $\Gamma'$.
\end{proof}


Given a labeling of the edges $e_1,e_2,\dots, e_m$ in $\Gamma$, we define the signature $\sgn(w)$ of an Eulerian path $w = (e_{\sigma(1)}, \ldots, e_{\sigma(m)})$ to be the signature of the permutation $\sigma \in \Sym_m$. Note that changing the initial labeling $e_1, e_2, \ldots, e_m$ either 
leaves every $\sgn(w)$ unchanged or multiplies $\sgn(w)$ by $-1$ for every Eulerian path $w$. We will be particularly interested in
the signed sum $\sum_{w \in \uni_a(\Gamma)} \, \sgn(w)$; this sum is well-defined (i.e., is independent of the labeling of the edges) up to a factor of $-1$.

We say that a graph has no repeated edges if there are no distinct edges which share both source and target vertices. 
The following lemma is remarked upon by Swan; ~\cite[page~369]{swan}.

\begin{lemma}\label{repeated_edges}
Let $\Gamma$ be a graph with a repeated edge. Then $\sum_{w\in \uni_{a}(\Gamma)}\sgn(w)=0$ for any vertex $a$ of $\Gamma$.
\end{lemma}

\begin{proof}
Let $e_1$ and $e_2$ be a pair of repeated edges. Let us partition the Eulerian paths in $\uni_{a}(\Gamma)$ into two groups, $\uni_1$ and $\uni_2$, as follows:
$w \in \uni_1$ if $e_1$ occurs before $e_2$ in $w$ and $w \in \uni_2$ if $e_2$ occurs before $e_1$. 
Given an Eulerian path $w$ on $\Gamma$, we can form a new Eulerian path $w'$ by interchanging $e_1$ and $e_2$. This way we obtain a bijective correspondence 
between $\uni_1$ and $\uni_2$. Since we have performed a transposition to get from $w$ to $w'$, $\sgn(w')=-\sgn(w)$. 
This shows that  
\[ \sum_{w\in \uni_{a}(\Gamma)}\sgn(w)= \sum_{w\in \uni_1}\sgn(w) +  \sum_{w' \in \uni_2} \sgn(w') = 0 , \] as desired.
\end{proof}

The remainder of this section will be devoted to computing $\sum_{w\in \uni_{a}(\Gamma)}\sgn(w)$ for several families of graphs which will arise in the sequel.

\begin{lemma} \label{diag_Eulerian_sum_1}
If
\renewcommand{\nodeDist}{2}
\renewcommand{\labelDist}{0.1}
\[\Gamma\ =\ \begin{tikzpicture}[baseline,vertex/.style={anchor=base, circle, fill,minimum size=4pt, inner sep=0pt, outer sep=0pt},auto,
                               edge/.style={->,>=latex, shorten > = 5pt,shorten < = 5pt},
position/.style args={#1:#2 from #3}{
    at=(#3.#1), anchor=#1+180, shift=(#1:#2)
}]

  \node[vertex] (v0) {};
  \node[vertex,position = 110:{\nodeDist} from v0] (v1L) {};
  \node[vertex,position = 150:{\nodeDist} from v0] (v2L) {};
  \node[vertex,position = 190:{\nodeDist} from v0] (v3L) {};
  \node[vertex,position = 250:{\nodeDist} from v0] (v4L) {};

  \draw (v0) node[right] {$P_0$};
  \draw (v1L) node[above] {$P_{1}$};
  \draw (v2L) node[left] {$P_{2}$};
  \draw (v3L) node[left] {$P_{3}$};
  \draw (v4L) node[below] {$P_{\alpha}$};

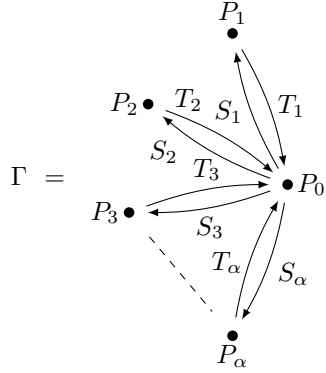
  \draw[edge] (v0) to [bend left=10] node[left] {$S_{1}$} (v1L);
  \draw[edge] (v1L) to [bend left=10] node[right] {$T_{1}$} (v0);

  \draw[edge] (v0) to [bend left=10] node[left = 10] {$S_{2}$} (v2L);
  \draw[edge] (v2L) to [bend left=10] node[above = 14, left = 4] {$T_{2}$} (v0);

  \draw[edge] (v0) to [bend left=10] node[below] {$S_{3}$} (v3L);
  \draw[edge] (v3L) to [bend left=10] node[above] {$T_{3}$} (v0);

  \draw[dashed,shorten > = 10pt, shorten < = 10pt, bend right=10] (v3L) -- (v4L);

  \draw[edge] (v0) to [bend left=10] node[below = 4, right] {$S_{\alpha}$} (v4L);
  \draw[edge] (v4L) to [bend left=10] node[below=2, left] {$T_{\alpha}$} (v0);
\end{tikzpicture}
\]
then
$ \displaystyle \sum_{w\in \uni_{P_a}(\Gamma)}\sgn(w)=\begin{cases}
\pm \alpha!\ ,\ &\text{ if $a=0$, or}\\
\pm (\alpha-1)!\ &\text{ otherwise.}
\end{cases}
$
\end{lemma}

\begin{proof}
First assume $a = 0$. There are $\alpha!$ Eulerian paths from $P_0$ on $\Gamma$, determined by the order in which each of the vertices $P_1, \ldots, P_{\alpha}$ are visited. Each is of the form
\[w_\tau=(S_{\tau(1)},T_{\tau(1)},S_{\tau(2)},T_{\tau(2)},\dots,S_{\tau(\alpha)},T_{\tau(\alpha)})\]
for $\tau\in \Sym_{\alpha}$.
It thus suffices to show that these $\alpha!$ Eulerian paths all have the same signature. Indeed,
the edges of $w_\tau$ come in groups of $2$, being $(S_1,T_1), \dots, (S_\alpha,T_\alpha)$. Interchanging any two of these groups 
results in an even permutation of the edges. Thus 
\[\sum_{w\in \uni_{P_0}(\Gamma)}\sgn(w)=\sum_{\tau\in\Sym_{\alpha}}\sgn(w_\tau)=\pm \alpha!\ . \]
Now assume that $a\neq 0$. In this case the Eulerian paths on $\Gamma$ from $P_a$ are precisely those of the form
$w=(T_a,w',S_a)$,
where $w'$ is an Eulerian path from $P_0$ on $\Gamma\setminus\{S_a,T_a\}$. As we showed above, there $(\alpha - 1)!$
possibilities for $w'$, and they all have the same signature; hence,
\[  \displaystyle \sum_{w\in \uni_{P_a}(\Gamma)}\sgn(w)= \pm (\alpha-1)!. \]
\end{proof}

\begin{lemma}\label{diag_Eulerian_sum_2}
Let $\alpha\geqslant 2$ and 
\renewcommand{\nodeDist}{2}
\renewcommand{\labelDist}{0.1}
\[\Gamma=
\begin{tikzpicture}[baseline,vertex/.style={anchor=base, circle, fill,minimum size=4pt, inner sep=0pt, outer sep=0pt},auto,
                               edge/.style={->,>=latex, shorten > = 5pt,shorten < = 5pt},
position/.style args={#1:#2 from #3}{
    at=(#3.#1), anchor=#1+180, shift=(#1:#2)
}]

  \node[vertex] (v0) {};
  \node[vertex,position = 140:{\nodeDist} from v0] (v1) {};
  \node[vertex,position = 40:{\nodeDist} from v0] (v2) {};
  \node[vertex,position = 190:{\nodeDist} from v0] (v3) {};
  \node[vertex,position = 240:{\nodeDist} from v0] (v4) {};
  \node[vertex,position = 320:{\nodeDist} from v0] (v5) {};

  \draw (v0) node[above=5] {$P_0$};
  \draw (v1) node[left] {$P_{1}$};
  \draw (v2) node[right] {$P_{2}$};
  \draw (v3) node[left] {$P_{3}$};
  \draw (v4) node[below] {$P_{4}$};
  \draw (v5) node[below] {$P_{\alpha}$};

  \draw[edge] (v0) to [bend left=10] node[below=5, left=-3] {$S_{1}$} (v1);
  \draw[edge] (v1) to [bend left=10] node[above=2, right] {$T_{1}$} (v0);

  \draw[edge] (v1) to [bend left=20] node[above] {$e$} (v2);

  \draw[edge] (v0) to [bend left=10] node[above=2, left] {$S_{2}$} (v2);
  \draw[edge] (v2) to [bend left=10] node[below=5, right] {$T_{2}$} (v0);

  \draw[edge] (v0) to [bend left=10] node[below=10, left=5] {$S_{3}$} (v3);
  \draw[edge] (v3) to [bend left=10] node[above=3, left=5] {$T_{3}$} (v0);

  \draw[edge] (v0) to [bend left=10] node[below = 4, right] {$S_{4}$} (v4);
  \draw[edge] (v4) to [bend left=10] node[below=2, left=3] {$T_{4}$} (v0);

  \draw[dashed, shorten > = 14pt, shorten < = 14pt, bend right=10] (v4) -- (v5);

  \draw[edge] (v0) to [bend left=10] node[right] {$S_{\alpha}$} (v5);
  \draw[edge] (v5) to [bend left=10] node[below=2] {$T_{\alpha}$} (v0);
\end{tikzpicture}\ .
\]
Then $\displaystyle \sum_{w\in \uni_{P_1}(\Gamma)}\sgn(w)=\pm(\alpha-1)!$.
\end{lemma}

\begin{proof}
Let us subdivide the Eulerian paths from $P_1$ on $\Gamma$ into three groups, $\Lambda_1$, $\Lambda_2$ and $\Lambda_3$, depending on whether the edge $e$ occurs at the beginning, the end or the middle of the path. 

It follows from Lemma~\ref{diag_Eulerian_sum_1}, that there are $(\alpha-1)!$ paths in $\Lambda_1$, all having the same signature,
and there are $(\alpha-1)!$ paths in $\Lambda_2$, all having the same signature. Moreover,
the signature of a path from $\Lambda_1$ is the same as the signature of a path from $\Lambda_2$. 
This can be seen directly by comparing the signatures of, say $(e,T_2,S_1,T_1,\tau,S_2) \in \Lambda_1$ and
$(T_1,\tau,S_2,T_2,S_1,e) \in \Lambda_2$, 
where $\tau=(S_3,T_3,\dots,S_{\alpha},T_{\alpha})$ is a path from $P_0$ to $P_0$. A simple calculation shows that
\[ \sgn(e,T_2,S_1,T_1,\tau,S_2)=\sgn(T_1,\tau,S_2,T_2,S_1,e), \]
as claimed. We now turn our attention to $\Lambda_3$. Any path in $\Lambda_3$ begins with edge $T_1$ and ends with edge $S_2$. 
These paths are determined by the order in which the $\alpha - 1$ circuits $(S_1, e, T_2)$, $(S_3, T_3)$, $\ldots$, $(S_{\alpha}, T_{\alpha})$
from $P_0$ are traversed in the subgraph
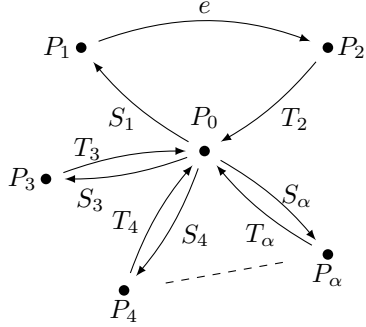
\renewcommand{\nodeDist}{2}
\[
\begin{tikzpicture}[baseline,vertex/.style={anchor=base, circle, fill,minimum size=4pt, inner sep=0pt, outer sep=0pt},auto,
                               edge/.style={->,>=latex, shorten > = 5pt,shorten < = 5pt},
position/.style args={#1:#2 from #3}{
    at=(#3.#1), anchor=#1+180, shift=(#1:#2)
}]

  \node[vertex] (v0) {};
  \node[vertex,position = 140:{\nodeDist} from v0] (v1) {};
  \node[vertex,position = 40:{\nodeDist} from v0] (v2) {};
  \node[vertex,position = 190:{\nodeDist} from v0] (v3) {};
  \node[vertex,position = 240:{\nodeDist} from v0] (v4) {};
  \node[vertex,position = 320:{\nodeDist} from v0] (v5) {};

  \draw (v0) node[above=5] {$P_0$};
  \draw (v1) node[left] {$P_{1}$};
  \draw (v2) node[right] {$P_{2}$};
  \draw (v3) node[left] {$P_{3}$};
  \draw (v4) node[below] {$P_{4}$};
  \draw (v5) node[below] {$P_{\alpha}$};

  \draw[edge] (v0) to [bend left=10] node[below=5, left=-3] {$S_{1}$} (v1);

  \draw[edge] (v1) to [bend left=20] node[above] {$e$} (v2);

  \draw[edge] (v2) to [bend left=10] node[below=5, right] {$T_{2}$} (v0);

  \draw[edge] (v0) to [bend left=10] node[below=10, left=5] {$S_{3}$} (v3);
  \draw[edge] (v3) to [bend left=10] node[above=3, left=5] {$T_{3}$} (v0);

  \draw[edge] (v0) to [bend left=10] node[below = 4, right] {$S_{4}$} (v4);
  \draw[edge] (v4) to [bend left=10] node[below=2, left=3] {$T_{4}$} (v0);

  \draw[dashed, shorten > = 14pt, shorten < = 14pt, bend right=10] (v4) -- (v5);

  \draw[edge] (v0) to [bend left=10] node[right] {$S_{\alpha}$} (v5);
  \draw[edge] (v5) to [bend left=10] node[below=2] {$T_{\alpha}$} (v0);
\end{tikzpicture}\ \text{.}
\]
Since interchanging any two of these circuits results in an even permutation of the edges, all paths in $\Lambda_3$ have the same signature. 
Moreover, a path in $\Lambda_3$ has the opposite signature to the paths in $\Lambda_1$ and $\Lambda_2$, as illustrated by
\[\sgn(T_1,S_2,T_2,\tau,S_1,e)= -\sgn(T_1,S_1,e,T_2,\tau,S_2)\ ,\]
where $\tau=(T_3,S_3,\dots,T_{\alpha},S_{\alpha})$.
Thus 
\begin{align*}
\sum_{w\in \uni_{P_1}(\Gamma)} \sgn(w) & =  
\pm \big( \sum_{w_1 \in \Lambda_1} \sgn(w_1) + \sum_{w_2 \in \Lambda_2} \sgn(w_2) + \sum_{w_3 \in \Lambda_3} \sgn(w_3) \big) \\
 & = 
\pm \big( 2(\alpha-1)!-(\alpha-1)! \big) = \pm (\alpha-1)!. 
\end{align*}
\end{proof}

\begin{lemma}\label{Eulerian_sum_1}
Let $a,b\in\{0,1,2,\dots,\alpha\}$ and 
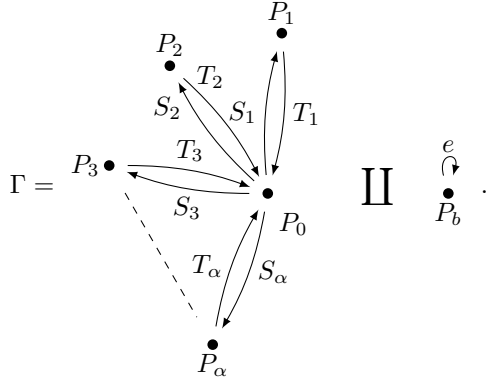
\renewcommand{\nodeDist}{2}
\renewcommand{\labelDist}{0.1}
\[\Gamma=
\begin{tikzpicture}[baseline,vertex/.style={anchor=base, circle, fill,minimum size=4pt, inner sep=0pt, outer sep=0pt},auto,
                               edge/.style={->,>=latex, shorten > = 5pt,shorten < = 5pt},
position/.style args={#1:#2 from #3}{
    at=(#3.#1), anchor=#1+180, shift=(#1:#2)
}]

  \node[vertex] (v0) {};
  \node[vertex,position = 85:{\nodeDist} from v0] (v1L) {};
  \node[vertex,position = 127.5:{\nodeDist} from v0] (v2L) {};
  \node[vertex,position = 170:{\nodeDist} from v0] (v3L) {};
  \node[vertex,position = 250:{\nodeDist} from v0] (v4L) {};

  \draw (v0) node[position=310:{\labelDist} from v0] {$P_0$};
  \draw (v1L) node[above] {$P_{1}$};
  \draw (v2L) node[above] {$P_{2}$};
  \draw (v3L) node[left] {$P_{3}$};
  \draw (v4L) node[below] {$P_{\alpha}$};

  \draw[edge] (v0) to [bend left=10] node[left] {$S_{1}$} (v1L);
  \draw[edge] (v1L) to [bend left=10] node[right] {$T_{1}$} (v0);

  \draw[edge] (v0) to [bend left=10] node[above = 10, left = 8] {$S_{2}$} (v2L);
  \draw[edge] (v2L) to [bend left=10] node[above = 18, left = -3] {$T_{2}$} (v0);

  \draw[edge] (v0) to [bend left=10] node[below] {$S_{3}$} (v3L);
  \draw[edge] (v3L) to [bend left=10] node[above] {$T_{3}$} (v0);

  \draw[dashed,shorten > = 10pt, shorten < = 10pt, bend right=10] (v3L) -- (v4L);

  \draw[edge] (v0) to [bend left=10] node[right] {$S_{\alpha}$} (v4L);
  \draw[edge] (v4L) to [bend left=10] node[left] {$T_{\alpha}$} (v0);
\end{tikzpicture}
\quad\CMcoprod\quad
\begin{tikzpicture}[baseline,vertex/.style={anchor=base, circle, fill,minimum size=4pt, inner sep=0pt, outer sep=0pt},auto,
                               edge/.style={->,>=latex, shorten > = 5pt,shorten < = 5pt},
position/.style args={#1:#2 from #3}{
    at=(#3.#1), anchor=#1+180, shift=(#1:#2)
}]

  \node[vertex] (vb) {};
  \draw (vb) node[below] {$P_{b}$};
  \draw[edge] (vb) to [loop above] node[above] {$e$} (vb);
\end{tikzpicture}\ .
\]
Then
\begin{equation}\label{eq_Eulerian_sum_1}
\sum_{w\in \uni_{P_a}(\Gamma)}\sgn(w)=\begin{cases}
\pm (\alpha+1)!\ ,\ &\text{ if $a = b = 0$,}\\
\pm 2(\alpha-1)!\ ,\ &\text{ if $a=b \neq 0$,}\\
\pm \alpha!\ ,\ &\text{ if $a=0$ and $b\neq 0$, or $a\neq 0$ and $b=0$,}\\
\pm (\alpha-1)!\ ,\ &\text{in all other cases,}
\end{cases}
\end{equation}
where for fixed $b$, the sum $\sum_{w\in \uni_{P_a}(\Gamma)}\sgn(w)$ is either positive for all $a$, or negative for all $a$.
\end{lemma}

\begin{proof}
Any Eulerian path on $\Gamma$ is a cyclic permutation of either
\[(e,S_{\tau(1)},T_{\tau(1)},S_{\tau(2)},T_{\tau(2)},\dots,S_{\tau(\alpha)},T_{\tau(\alpha)})\]
if $b=0$, or
\[(S_{\tau(1)},e,T_{\tau(1)},S_{\tau(2)},T_{\tau(2)},\dots,S_{\tau(\alpha)},T_{\tau(\alpha)})\]
if $b\neq 0$, for some permutation $\tau\in \Sym_{\alpha}$.  
Interchanging any two blocks of the form $(S_{i_1},T_{i_1})$ and $(S_{i_2},T_{i_2})$ induces 
an even permutation of the edges. Cyclically permuting a path of length $2 \alpha+1$ also induces an even permutation of the edges. 
Hence every Eulerian path has the same fixed signature when $b=0$, and every Eulerian path has the same (opposite) fixed signature when $b\neq 0$.

If $a = b = 0$, the Eulerian paths from $P_0$ are determined by the order in which the 
$\alpha + 1$ circuits $e$, $(S_1, T_1)$, $\ldots$, $(S_{\alpha}, T_{\alpha})$ are traversed. Thus there are $(\alpha + 1)!$ Eulerian paths in this case.

If $a = b \neq 0$, then $e$ occurs either at the beginning or the end of each Eulerian path. Lemma~\ref{diag_Eulerian_sum_1}
tells us that there are $(\alpha-1)!$ Eulerian paths starting with $e$ and $(\alpha-1)!$ Eulerian paths ending with $e$.
Thus the total number of Eulerian paths from $P_a$ in this case is $2 (\alpha-1)!$.

If $a = 0$ and $b \neq 0$, then Eulerian paths from $P_a$ are in bijective correspondence with permutations of the $\alpha$ circuits, 
$(S_b, e, T_b)$ and $(S_i, T_i)$, where $i \neq b$. 

If $a \neq 0$ and $b = 0$, then every path starts with $T_a$ and ends with $S_a$, and the Eulerian paths from $P_a$ are in bijective correspondence 
with permutations of the remaining $\alpha$ circuits $e$ and $(S_i, T_i)$, where $i \neq a$.

Finally, if $a$, $b$ are distinct and non-zero, then again every path starts with $T_a$ and ends with $S_a$, so the count is the same as above, 
except that instead we only have $\alpha - 1$ circuits, $(S_b, e, T_b)$ and $(S_i, T_i)$, $i \neq a, b$.
\end{proof}

\begin{lemma}\label{Eulerian_sum_2}
If
\[\Gamma=
\begin{tikzpicture}[baseline,vertex/.style={anchor=base, circle, fill,minimum size=4pt, inner sep=0pt, outer sep=0pt},auto,
                               edge/.style={->,>=latex, shorten > = 5pt,shorten < = 5pt},
position/.style args={#1:#2 from #3}{
    at=(#3.#1), anchor=#1+180, shift=(#1:#2)
}]

  \node[vertex] (v0) {};
  \node[vertex,position = 0:{\nodeDist*2/3} from v0] (va1) {};
  \node[position = 0:{\nodeDist} from va1] (va2) {};
  \node[vertex,position = 0:{\nodeDist*2/3} from va2] (va3) {};
  \node[vertex,position = 110:{\nodeDist} from v0] (v1L) {};
  \node[vertex,position = 150:{\nodeDist} from v0] (v2L) {};
  \node[vertex,position = 190:{\nodeDist} from v0] (v3L) {};
  \node[vertex,position = 250:{\nodeDist} from v0] (v4L) {};

  \draw (v0) node[below=12, right=-1] {$P_0$};
  \draw (va3) node[below=2] {$P_{\alpha+\beta}$};
  \draw (v1L) node[above] {$P_{1}$};
  \draw (v2L) node[left] {$P_{2}$};
  \draw (v3L) node[left] {$P_{3}$};
  \draw (v4L) node[below] {$P_{\alpha}$};

  \draw[edge] (va3) to [loop above] node[above] {$e$} (va3);

  \draw[edge] (v0) to [bend left=10] node[above=8, right=-3] {$S_{\alpha+1}$} (va1);
  \draw[edge] (va1) to [bend left=10] node[below=9, right=-3] {$T_{\alpha+1}$} (v0);

  \draw[dashed,shorten > = 10pt, shorten < = 10pt, bend right=10] (va1) -- (va2);

  \draw[edge] (va2) to [bend left=10] node[above=8, left=-6] {$S_{\alpha+\beta}$} (va3);
  \draw[edge] (va3) to [bend left=10] node[below=9, left=-6] {$T_{\alpha+\beta}$} (va2);

  \draw[edge] (v0) to [bend left=10] node[left] {$S_{1}$} (v1L);
  \draw[edge] (v1L) to [bend left=10] node[right] {$T_{1}$} (v0);

  \draw[edge] (v0) to [bend left=10] node[left = 10] {$S_{2}$} (v2L);
  \draw[edge] (v2L) to [bend left=10] node[above = 14, left = 4] {$T_{2}$} (v0);

  \draw[edge] (v0) to [bend left=10] node[below] {$S_{3}$} (v3L);
  \draw[edge] (v3L) to [bend left=10] node[above] {$T_{3}$} (v0);

  \draw[dashed,shorten > = 10pt, shorten < = 10pt, bend right=10] (v3L) -- (v4L);

  \draw[edge] (v0) to [bend left=10] node[below = 4, right] {$S_{\alpha}$} (v4L);
  \draw[edge] (v4L) to [bend left=10] node[below=2, left] {$T_{\alpha}$} (v0);
\end{tikzpicture}
\]
then
$\displaystyle \sum_{w\in \uni_{P_{\alpha+\beta}}(\Gamma)}\sgn(w)=\pm 2\alpha ! $ .
\end{lemma}

\begin{proof}
Any Eulerian path from $P_{\alpha+\beta}$ on $\Gamma$ is either of the form
\begin{align*}
&(e,T_{\alpha+\beta},\dots,T_{\alpha+1},w,S_{\alpha+1},\dots,S_{\alpha+\beta})\ ,\quad \text{ or}
&(T_{\alpha+\beta},\dots,T_{\alpha+1},w,S_{\alpha+1},\dots,S_{\alpha+\beta},e)\ ,
\end{align*}
where $w$ is an Eulerian path from $P_0$ on the subgraph of $\Gamma$ consisting of edges $S_1,T_1,S_2,T_2$, $\dots,S_\alpha,T_\alpha$. These two Eulerian paths on $\Gamma$ for fixed $w$ are related by the even permutation cyclically permuting the $2\alpha+2\beta+1$ edges, and all $\alpha!$ Eulerian paths from $P_0$ on the subgraph have the same signature, by Lemma \ref{Eulerian_sum_1}. Thus the signed count of Eulerian paths on $\Gamma$ from $P_{\alpha+\beta}$ is $\pm 2\alpha!$.
\end{proof}

\section{Proof of Theorem \ref{main_thm}: First reductions}\label{first_reductions}
\label{sect.first-reductions}

Fix integers $k, n \geqslant 1$. 
Recall from the Introduction that given a $k$-tuple of $n \times n$ matrices $A_1, \ldots, A_k$, we defined
the linear transformation
$L(A_1, \ldots, A_k) \colon \M_n \to \M_n$ by $L(A_1, \ldots, A_k)(A_{k+1}) = [A_1, \ldots, A_k, A_{k+1}]$.
We will identify $L(A_1, \ldots, A_k)$ with its matrix relative to the standard basis $ \{ E_{a b} \, | \, a, b = 1, \ldots, n \}$ of elementary matrices in $\M_n$.
Here $E_{a b}$ is the elementary matrix with $1$ in the $(a, b)$-position and $0$s elsewhere; we will sometimes write
$E_{a, b}$ in place of $E_{a b}$.

Let $W_{\nullity \, > \, i} \subset (\M_n)^k$ be the locus of $k$-tuples $(A_1, \dots, A_k) \in \M_n$ such that 
\[ \text{$\nullity(L(A_1, \ldots, A_k))> i$ or equivalently, $\rank(L(A_1, \ldots, A_k)) < n^2 - i$.} \]
Clearly 
$\emptyset = W_{\nullity \, >  \, n^2} \subseteq W_{\nullity \,  > \, n^2 -1} \subseteq \ldots \subseteq W_{\nullity \, > \, 0} \subseteq W_{\nullity \, > \,  -1} = (\M_n)^k.$

\begin{lemma} \label{lem.zariski} (a) $W_{\nullity \, > \, i}$ is Zariski closed in $(\M_n)^k$ for every integer $i$.

\smallskip
(b) Assume $k \leqslant n^2$. Then $W_{\nullity \, > \, k-1} = (\M_n)^k$ and hence, 
\[ \nullity(L(A_1, \ldots, A_k)) \geqslant k \] for any
$A_1, \ldots, A_k \in \M_n$.

\smallskip
(c) Assume that $k$ is even and $2 \leqslant n \leqslant 2n - 2$.
In order to prove Theorem~\ref{main_thm} it suffices to show that there exists a field $K$ containing $F$
and $k$ matrices $A_1, \ldots, A_k \in \M_n(K)$ such that \[ \nullity(L(A_1, \dots, A_k)) \leqslant k . \]
\end{lemma}

\begin{proof} (a) The entries of the $n^2 \times n^2$ matrix 
$L(A_1, \ldots, A_k) \colon \M_n \to \M_n$ are polynomials in the entries of $A_1, \dots, A_k$. By definition $W_{\nullity \, > \, i} \subset (\M_n)^k$ 
is the common zero locus of the determinants of the $(n^2 -i) \times (n^2 -i)$-minors of this matrix. 
These determinants are again polynomials in the entries of $A_1, \dots, A_k$, and part (a) follows.

(b) Since the standard polynomial $[X_1,\dots, X_{k+1}]$ is alternating in $X_1, \dots, X_{k+1}$, we have
\[ \text{$L(A_1, \dots, A_k)(A_j) = 0$ for every $j = 1, \ldots, k$.} \] In other words, the kernel of $L(A_1,A_2,\dots,A_k)$ 
contains the span of $A_1, \dots, A_k$. If $A_1, \ldots, A_k$ are linearly independent, this shows that $\nullity(L(A_1, \ldots, A_k)) \geqslant k$.
In other words, $W_{\nullity \, > \, k-1}$  contains the dense open subvariety of $\M_n^k$ consisting of linearly independent $k$-tuples of $n \times n$ matrices.
Since $W_{\nullity \, > \, k-1}$ is Zariski closed by part (a), we conclude that $W_{\nullity \, > \, k-1} = \M_n^k$.

(c) Note that under the assumptions of Theorem~\ref{main_thm}, $k = 2r < 2n$. Hence, $k \leqslant n^2$, and part (b) 
applies. In view of part (b), Theorem~\ref{main_thm} is equivalent to the assertion that $W_{\nullity \, > \, k} \neq (\M_n)^k$. 
By part (a), $W_{\nullity \, > \, k}$ is Zariski closed in $\M_n^k$.
To prove that $W_{\nullity \, > \, k} \neq (\M_n)^k$, it suffices to show that the complement 
$(\M_n)^k \setminus W_{\nullity \, > \, k}$ has a $K$-point for some field $K$ containing $F$. 
In other words, it suffices to show that there exist matrices $A_1, \ldots, A_k \in \M_n(K)$ such that $\nullity(L(A_1, \dots, A_k)) \leqslant k$.
\end{proof}

Our proof of Theorem~\ref{main_thm} will be based on Lemma~\ref{lem.zariski}(c). Note that it is not a priori clear how to choose
the matrices $A_1, \ldots, A_k$. Informally speaking, if they are chosen to be very general (e.g., 
if their entries are independent variables over $F$),
it becomes difficult to compute $L(A_1, \ldots, A_k)$ explicitly 
enough to determine its nullity. On the other hand, in multiple examples where we chose 
the $k$-tuple $A_1, \ldots, A_k$ in various special positions, the nullity of $L(A_1, \ldots, A_k)$
turned out to be higher than $k$ (and usually $\longrightarrow \infty$ with $n$)\footnote{As an extreme example of this phenomenon,
$L(A_1, \ldots, A_k) = 0$ and hence has nullity $n^2$, if $k > 2$ and $A_1, \ldots, A_k$ are required to commute pairwise.}.
The remainder of this section will be devoted to defining a field $K$ containing $F$ and a $k$-tuple $A_1, \ldots, A_k \in \M_n(K)$ 
that will, in retrospect, turn out to be ``just right": ``special enough" to make $\nullity (L(A_1, \ldots, A_k))$ computable,
yet ``general enough", to ensure that 
\begin{equation} \label{e.nullity}
\nullity ( L(A_1, \ldots, A_k) ) \leqslant k.
\end{equation} 
The special property of $A_1, \ldots, A_k$ that will facilitate subsequent computations is that the $n^2 \times n^2$
matrix $L(A_1,\dots,A_k)$ naturally decomposes as a direct sum of $n\times n$ matrices. On the other hand, the inequality~\eqref{e.nullity} 
will not be obvious at this stage; its proof will take up much of the remainder of this paper. Note also that the $k$-tuple
$A_1, \ldots, A_k$ we will define in this section is really a family of $k$-tuples that depends 
on the integer parameters $s_1, \ldots, s_k$. These integer parameters will remain unspecified until 
Section~\ref{specialization}.

   From now on we will set $K= F(x_{\ell, \alpha})$, where $x_{\ell, \alpha}$ are independent variables, 
as $\ell$ ranges from $1$ to $k$ and $\alpha$ ranges over $\bbZ/ n \bbZ$. For notational convenience, we will
label rows and columns of $n \times n$ matrices by $0, 1, \ldots, n-1$ and view these labels as integers modulo $n$.
Let
\begin{equation} \label{e.D}
D_1 = \diag(x_{1, 0}, x_{1, 1}, \ldots, x_{1, n-1}), \ldots,
D_k = \diag(x_{k, 0}, x_{k, 1}, \ldots, x_{k, n-1})
\end{equation}
be a $k$-tuple of diagonal matrices in $\M_n(K)$.
We will study $L(A_1, \ldots, A_k)$ for 
\begin{equation} \label{e.A_i} 
\text{$A_{1} = D_{1} \cdot C^{s_1}$, $A_{2} = D_{2} \cdot C^{s_2}$, \ldots, $A_{k} = D_{k} \cdot C^{s_k}$} 
\end{equation}
in $\M_n(K)$. Here $C$ denotes the cyclic permutation matrix 
\begin{equation} \label{e.C}  C= \sum_{i \in \mathbb Z/n \mathbb Z} E_{i, i+1} = \left(\begin{array}{ccccc}
0&1&0&\cdots&0\\
0&0&1&\cdots&0\\
\vdots&\vdots&\vdots&\ddots&\vdots\\
0&0&0&\cdots&1\\
1&0&0&\cdots&0
\end{array}\right)\ . \end{equation} 
and the exponents $s_1, \dots, s_k$ are integers, to be specified later. Note that $C^n = I$, where $I$ denotes the $n \times n$ identity matrix.
Moreover, 
\[ \text{$CE_{ii}=E_{i-1, i-1}C$ and $E_{i, i + j} = E_{ii} C^j$ for every $i, j \in \mathbb Z / n \mathbb Z$.} \]
Let $V_j$ be the n-dimensional vector space spanned by the matrices
\begin{equation} \label{e.basis}
E_{0, j} = E_{0, 0} C^j, \; E_{1, j + 1} = E_{1, 1} C^j, \,  \ldots \, , E_{n - 1, n - 1 + j} = E_{n-1, n-1} C^j .
\end{equation}
Equivalently, $V_j$ is the space of matrices all of whose non-zero entries 
are concentrated on the main diagonal, shifted up by $j$ units,
i.e., in positions $(i, i + j)$, where $j$ is fixed and $i$ ranges over $\mathbb Z/ n \mathbb Z$. 

Now observe that $\M_n=V_0\oplus V_1\oplus\cdots\oplus V_{n-1}$. Moreover, every term in
\[L(A_1,\dots,A_k)(E_{i i}C^j)=[D_1C^{s_1},\dots,D_kC^{s_k},E_{ii}C^j]\]
is of the form ${\tilde D}C^{j + s}$ for some diagonal matrix ${\tilde D}$. Here $s = s_1 + \ldots + s_k$.
This matrix lies in $V_{j+s}$, where $j+s$ is viewed modulo $n$. In other words, the linear transformation
$L = L(A_1, \ldots, A_k) \colon \M_n(K) \to \M_n(K)$ naturally decomposes as a direct sum of $n$ linear maps 
$L_j \colon V_j \to V_{j + s}$, where $j$ ranges over $0,1,2,\ldots, n-1$ and $\dim(V_j) = n$ for each $j$.
Consequently, $\Ker(L)$ decomposes as a direct sum
$\Ker(L_0) \oplus \ldots \oplus \Ker(L_{n-1})$. In summary, we have reduced Theorem~\ref{main_thm} to the following.

\begin{proposition}
\label{prop.red2}
Assume that $k = 2r$ is even, $n > r$, the base field $F$ is infinite of characteristic not dividing $2(2r+1)r!$, and
the matrices $A_1, \ldots, A_k$ are as in~\eqref{e.A_i}. Let $L_j \colon V_j \to V_{j + s_1 + \ldots + s_k}$ 
be the restriction of $L(A_1, \ldots, A_k)$ to $V_j$. Then for some choice of the exponents $s_1, \ldots, s_k \in \bbZ$, 
\[ \nullity(L_0)+\dots+\nullity(L_{n-1})\leqslant k. \]
\end{proposition}

The remainder of this paper will be devoted to proving Proposition~\ref{prop.red2}.

\section{A graph-theoretic description of $L_j$}
\label{overall_strategy}

Throughout this section we fix positive integers $k$ and $n$ as well as 
$s_1,\dots,s_k$ and $j$ in $\Z/n\Z$. We will continue using the following notations:
$D_1, \ldots, D_k$ will be generic diagonal matrices, as in~\eqref{e.D}, and
\[A_\ell= D_{\ell} C^{s_{\ell}} = \sum_{i=0}^{n-1}x_{\ell, \alpha}E_{\alpha, \alpha+s_\ell}\]
will be as in~\eqref{e.A_i} for $\ell = 1, \ldots, k$. 
As we saw in the previous section, for this choice of $A_1, \ldots, A_k$, 
the linear transformation $L(A_1, \dots, A_k)$ decomposes as $L_0 \oplus L_1 \oplus \ldots \oplus L_{n-1}$.
We will identify the linear transformation $L_j \colon V_j\rightarrow V_{j+s}$, where $s=\sum_{\ell=1}^k s_{\ell}$, 
with the $n \times n$ matrix representing it in the bases~\eqref{e.basis} of $V_j$ and $V_{j+s}$.

In order to find the matrix of $L_j$, we will want to calculate the generalized commutator 
\[ L(A_1, \ldots, A_k)(A_{k+1}) = [A_1,\dots,A_k, A_{k+1}] \]
as $A_{k+1}$ ranges over the basis $E_{\alpha, \alpha + j} = E_{\alpha \alpha} C^j$ of $V_j$. Expanding the generalized commutator $[A_1, \ldots, A_{k+1}]$, 
we see that each entry is a multilinear polynomial in the groups of variables $\{ x_{1, \alpha} \}, \{ x_{2, \alpha} \}, \ldots, \{ x_{k, \alpha} \}$, i.e.,
a linear combination of monomials of the form
$x_{1,\alpha_1}x_{2,\alpha_2}\cdots x_{k,\alpha_{k}}$ with integer coefficients. 
We will now give a graph-theoretic description of the coefficients of these monomials.

\begin{nc} \label{nc.graph}
(a) For the rest of this paper by a graph we will mean a directed graph with $n$ vertices labeled $P_v$ for $v\in \Z/n\Z$ and
at most $k+1$ edges, labeled $e_{\ell}$ with $\ell$ from $\{ 1,2,\dots, k+1 \}$. 
For $\ell = 1, \ldots, k$, the labeled edge $e_{\ell}$ will be of the form
\[ \underset {P_{\alpha_{\ell}}} \bullet \underset {e_{\ell}} \longrightarrow \underset {P_{\alpha_{\ell} +s_{\ell}}} \bullet \]
and the labeled edge $e_{k+1}$ will be of the form
$ \underset {P_{\alpha_{k+1}}} \bullet \underset {e_{k+1}} \longrightarrow \underset {P_{\alpha_{k+1} + j}} \bullet$.
Each edge $e_{\ell}$ will appear in a given graph at most once.

\smallskip
(b) We will say that a graph $\Gamma$ is a disjoint union of $\Gamma'$ and $\Gamma''$ and write 
\[\Gamma = \Gamma' \ \CMamalg\ \Gamma '\]
if the edge set of $\Gamma$ is the disjoint union of the edge sets of $\Gamma'$ and $\Gamma''$. Here by ``disjoint" we mean that $e_{\ell}$ 
cannot be an edge in both $\Gamma'$ and $\Gamma''$ for any $\ell = 1, \ldots, k+1$. Note that the vertices of $\Gamma$, $\Gamma'$ and $\Gamma''$ are
assumed to be $P_v$, $v \in \mathbb Z/ n \mathbb Z$, as in (a).

\smallskip
(c) Let $G$ be a graph with vertex set $\{ P_v \, | \, v \in \Z/ n \Z \}$ whose edge set is a subset of $\{ e_1, \ldots, e_k \}$.
We define the graph $G_b$ to be
\[G_b=G\ \CMamalg\ \underset {P_b} \bullet \underset {e_{k+1}} \longrightarrow \underset {P_{b+j}} \bullet\ . \]
In other words, $G_b$ is the graph obtained from $G$ by adding one extra edge $e_{k+1}$ having 
source vertex $P_b$ and target vertex $P_{b+j}$.
\end{nc}

Let $m$ be a monomial of the form
\[m=x_{1,\alpha_1}x_{2,\alpha_2}\cdots x_{k,\alpha_k}\]
where $\alpha_1,\dots,\alpha_k$ are elements of $\Z/n\Z$. We define $\gr(m)$ to be the graph with $k$ edges 
\begin{equation}\label{gr_m}
\underset {P_{\alpha_1}} \bullet \underset {e_{1}} \longrightarrow \underset {P_{\alpha_1+s_1}} \bullet
\quad
\underset {P_{\alpha_2}} \bullet \underset {e_{2}} \longrightarrow \underset {P_{\alpha_2+s_{2}}} \bullet
\quad\cdots\quad
\underset {P_{\alpha_k}} \bullet \underset {e_{k}} \longrightarrow \underset {P_{\alpha_k + s_{k}}} \bullet\ .
\end{equation}
Conversely, for a graph $G$ with vertices of the form $P_v$, $v \in \bbZ/ n \bbZ$ 
whose edges have labels $e_{1},e_{2},\dots,e_{k}$, we define the monomial $\mon(G)$ to be
\begin{equation}\label{mon_G}
\mon(G)=x_{1,\src_G(e_{1})}x_{2,\src_G(e_{2})}\cdots x_{k,\src_G(e_{k})}\ .
\end{equation}
Note that our definitions of $\mon(G)$ and $\gr(m)$ are inverse to each other: $\mon(\gr(m))=m$ and $\gr(\mon(G))=G$ for any monomial $m$ and graph $G$ of our required form. We will use this correspondence between graphs and monomials to study the entries of $L_j$ in a graph-theoretic manner.
Graphs of the form $\gr(m)$ correspond to monomials $m$ which may appear 
in the matrix $L_j$. The addition of the extra edge $e_{k+1}$ in $\gr(m)_b$ will help us keep track of the coefficient of $m$ 
in the $b^{\rm th}$ column of $L_j$, as is explained in the lemma below.

\begin{lemma}\label{matrix_from_graph}
Fix $\alpha_1,\dots,\alpha_k$ and $b$ to be integers defined modulo $n$. Define $m$ to be the monomial
\[m=x_{1,\alpha_1}x_{2,\alpha_2}\cdots x_{k,\alpha_k}\]
and $G=\gr(m)$.
The coefficient of $m$ appearing in the $(a,b)^{\rm th}$ entry of $L_j$ is
\[\sum_{w\in \uni_{P_a}(G_b)}\sgn(w)\ .\]
\end{lemma}

Here, as usual, $\uni_{P_a}(G_b)$ denotes the set of Eulerian paths on $G_b$ originating at $P_a$.

\begin{proof}
Recall that $V_j$ is spanned by the ordered basis $\{E_{0, 0} C^j,E_{1, 1,} C^j,\dots,E_{n-1, n-1} C^j\}$. The matrix $L_j$ is the restriction of $L$ to $V_j$, mapping $V_j$ to $V_{j+s}$ where $s=s_1+s_2+\dots+s_k$. Thus the $(a,b)^{\rm th}$ entry of $L_j$ is the coefficient of $E_{a, a}C^{j+s}$ appearing in 
$L_j(E_{b, b}C^j)=L(E_{b, b}C^j)$. This is the $(a,a+j+s)^{\rm th}$ entry of $L(E_{b, b}C^j)$.

Write $\alpha_{k+1}=b$ and $A_{k+1}=E_{b, b}C^j$. To calculate the coefficients of $x_{1,\alpha_1}\cdots x_{k,\alpha_k}$ appearing in $A_{\sigma(1)}\cdots A_{\sigma(k+1)}$ we set all indeterminants $x_{\ell, \alpha}$, other than $x_{1,\alpha_1},\dots,x_{k,\alpha_k}$, to $0$. Thus the coefficients of $x_{1,\alpha_1}\cdots x_{k,\alpha_k}$ appearing in $A_{\sigma(1)}\cdots A_{\sigma(k+1)}$ are the same as the coefficients of $x_{1,\alpha_1}\cdots x_{k,\alpha_k}$ appearing in $B_{\sigma(1)}\cdots B_{\sigma(k+1)}$, where for $\ell = 1, 2, \ldots, k$,
\[ B_\ell= \begin{pmatrix} 0 & \hdots & 0 & \hdots & 0 \\
\vdots & \ddots & \vdots & \vdots  & \vdots \\
0 & \hdots & x_{\ell, \alpha_\ell} & \hdots & 0 \\
\vdots & \vdots & \vdots & \ddots & \hdots  \\
0 & \hdots & 0 & \hdots & 0 
\end{pmatrix} \cdot C^{s_i} = x_{\ell, \alpha_\ell} E_{\alpha_\ell,\alpha_\ell + s_\ell} \]
and $B_{k+1} = A_{k+1} = E_{b, b + j}$. 
Recall that by the definition of $\gr(m)$, \[ \text{$(\alpha_\ell, \alpha_\ell + s_\ell) = (\src(e_\ell), \tar(e_\ell))$ for $\ell = 1, \ldots, k$, and
$(b, b+j) = (\src(e_{k+1}), \tar(e_{k+1}))$.} \]
As in Swan's paper~\cite{swan}, $L_j(B_{k+1}) = [B_1, \ldots, B_{k+1}]$
can be described by counting Eulerian paths. Using the product rule
\[ E_{p, q}E_{p', q'} =\begin{cases}
E_{p, q'} & \text{if $q = p'$, and}\\
0 & \text{otherwise}\ ,
\end{cases}\]
we see that 
\[ B_{\sigma(1)} \cdot \ldots \cdot B_{\sigma(k+1)} = m E_{\src(e_{\sigma(1)}),\tar(e_{\sigma(1)})}\cdots E_{\src(e_{\sigma(k+1)}),\tar(e_{\sigma(k+1)})} \]
has an $m = x_{1,\alpha_1},\dots,x_{k,\alpha_k}$ in the $(\src(e_{\sigma(1)}),\tar(e_{\sigma(k+1)}))^{\rm th}$ entry, if and only if $\tar(e_{\sigma(\ell)})=\src(e_{\sigma(\ell+1)})$ for each $\ell=1,2,\dots,k$. This is precisely the requirement that $(e_{\sigma(1)},\dots,e_{\sigma(k+1)})$ forms an Eulerian path from $P_{\src(e_{\sigma(1)})}$ to $P_{\tar(e_{\sigma(k+1)})}$ on $G_b$. All other entries of this product are zero.

When $(e_{\sigma(1)},\dots,e_{\sigma(k+1)})$ does form an Eulerian path on $G_b$, this path terminates at $\tar(e_{\sigma(k+1)})=\src(e_{\sigma(1)})+s_1+s_2+\dots+s_k+j=src(e_{\sigma(1)})+j+s$, as $s_1,\dots,s_k$ and $j$ are the differences between the source and target vertices of each edge in the path. Thus the $(a,a+j+s)^{\rm th}$ entry of $B_{\sigma(1)}\cdots B_{\sigma(k+1)}$ is $m = x_{1,\alpha_1}\cdots x_{k,\alpha_k}$ if and only if $\src(e_{\sigma(1)})=P_a$ and $(e_{\sigma(1)},\dots,e_{\sigma(k+1)})$ is an Eulerian path on $G_b$, that is, if and only if $(e_{\sigma(1)},\dots,e_{\sigma(k+1)})$ is an Eulerian path from $P_a$ on $G_b$.

Summing over all permutations $\sigma\in S_{k+1}$ 
we obtain that the coefficient of $x_{1,\alpha_1},\dots,x_{k,\alpha_k}$ appearing in the $(a,a+j+s)^{\rm th}$ entry of 
$L(E_{b, b}C^j)=[A_1,\dots,A_{k+1}]$ is $\sum_{w\in \uni_{P_a}(G_b)}\sgn(w)$, and therefore the $(a,b)^{\rm th}$ entry of $L_j$ is $\sum_{w\in \uni_{P_a}(G_b)}\sgn(w)$.
\end{proof}

With this lemma we can study $L_j$ by considering only graphs of the form $\gr(m)_b$ which admit Eulerian paths. We define a set of graphs which could possibly give rise to a nonzero coefficient in the $a^{\rm th}$ row of $L_j$.

\begin{definition} \label{U_gen_def}
Fix $a,j$ and $s_1,s_2,\dots,s_k$ to be elements of $\Z/n\Z$ and let $I\subseteq\{1,2,\dots,k\}$. Define $U(a,j,I)$ 
to be the set of directed graphs $G$, as in Notational Conventions~\ref{nc.graph}(a), satisfying the following additional conditions.

\smallskip
(i) The edges of $G$ are precisely $e_\ell$ for $\ell\in I$.

\smallskip
(ii) Each edge $e_\ell$ for $\ell\in I$ is of the form $\underset {P_{\alpha_\ell}} \bullet \underset {e_\ell} \longrightarrow \underset {P_{\alpha_\ell+s_\ell}} \bullet$ for some $\alpha_\ell$.

\smallskip
(iii) $G$ has no repeated edges.

\smallskip
(iv) There is some $b \in \Z/n\Z$ such that $G_b$ has an Eulerian path from $P_a$ in $G_b$.

\smallskip
\noindent
We abbreviate
\[U(a,j)=U\big(a,j,\{1,2,\dots,k\}\big) . \]
and
\[U_{nz}(a,j)=\left\{G\in U(a,j)\ |\ \text{There exists $b\in\Z/n\Z$ such that $\sum_{w\in \uni_{P_a}(G_b)}\sgn(w)\neq 0$}\right\}\ .\]

\end{definition}

\begin{remark}
We will be primarily interested in the case where $I = \{ 1, \ldots, k \}$. We allow $I$ to be a proper subset of $\{ 1, 2, \ldots, k \}$ to facilitate induction arguments later on.

In the case where $I = \{ 1, 2, \ldots, k \}$, conditions (i) and (ii) are equivalent to the requirement 
that $G$ is of the form $\gr(m)$ for some monomial $m = x_{1, \alpha_1} x_{2, \alpha_2} \ldots x_{k, \alpha_k}$. The reason for conditions (iii) and (iv) 
is that if they fail, then $\sum_{w\in \uni_{P_a}(G_b)}\sgn(w) = 0$ for each $b \in \bbZ/ n \bbZ$; see Lemma~\ref{repeated_edges}.
Thus by Lemma~\ref{matrix_from_graph} the monomial $m$ never appears in the $a$th row of $L_j$, 
and the graph $G$ does not contribute anything to the $a^{\rm th}$ row of $L_j$. 

For the same reason we are only really interested in graphs from $U_{nz}(a, j)$. 
However, it is not always transparent which graphs lie in $U_{nz}(a, j)$, so as a preliminary step, 
it will be convenient for us to work with all graphs from $U(a, j)$.
\end{remark}

Here is a brief example illustrating Definition~\ref{U_gen_def}.

\begin{example}\label{L_1_eg}
Let $k=2$, $n=3$, $j=1$ and $s_1=1$, $s_2=-1$. For fixed $a$, $U(a,j)$ consists of $4$ graphs, being

\renewcommand{\nodeDist}{1}
\[G^1=
\begin{tikzpicture}[baseline,vertex/.style={anchor=base, circle, fill,minimum size=4pt, inner sep=0pt, outer sep=0pt},auto,
                               edge/.style={->,>=latex, shorten > = 5pt,shorten < = 5pt},
position/.style args={#1:#2 from #3}{
    at=(#3.#1), anchor=#1+180, shift=(#1:#2)
}]
  \node[vertex] (v0) {};
  \node[vertex,position = 0:{\nodeDist} from v0] (v1) {};
  \node[vertex,position = 0:{\nodeDist} from v1] (v2) {};

  \draw (v0) node[below] {$P_a$};
  \draw (v1) node[below] {$P_{a+1}$};
  \draw (v2) node[below] {$P_{a+2}$};

  \draw[edge] (v0) to node[above] {$e_1$} (v1);
  \draw[edge] (v2) to [bend left=20] node[below] {$e_2$} (v1);
\end{tikzpicture}\quad \text{ , }\quad 
G^2=
\begin{tikzpicture}[baseline,vertex/.style={anchor=base, circle, fill,minimum size=4pt, inner sep=0pt, outer sep=0pt},auto,
                               edge/.style={->,>=latex, shorten > = 5pt,shorten < = 5pt},
position/.style args={#1:#2 from #3}{
    at=(#3.#1), anchor=#1+180, shift=(#1:#2)
}]
  \node[vertex] (v0) {};
  \node[vertex,position = 0:{\nodeDist} from v0] (v1) {};
  \node[vertex,position = 0:{\nodeDist} from v1] (v2) {};

  \draw (v0) node[below] {$P_a$};
  \draw (v1) node[below] {$P_{a+1}$};
  \draw (v2) node[below] {$P_{a+2}$};

  \draw[edge] (v1) to [bend left=20] node[above] {$e_1$} (v2);
  \draw[edge] (v2) to [bend left=20] node[below] {$e_2$} (v1);
\end{tikzpicture},
\]
\[
G^3=
\begin{tikzpicture}[baseline,vertex/.style={circle, fill,minimum size=4pt, inner sep=0pt, outer sep=0pt},auto,
                               edge/.style={->,>=latex, shorten > = 5pt,shorten < = 5pt},
position/.style args={#1:#2 from #3}{
    at=(#3.#1), anchor=#1+180, shift=(#1:#2)
}]
  \node[vertex] (v0) {};
  \node[vertex,position = 0:{\nodeDist} from v0] (v1) {};
  \node[vertex,position = 0:{\nodeDist} from v1] (v2) {};

  \draw (v0) node[below] {$P_{a-1}$};
  \draw (v1) node[below] {$P_{a}$};
  \draw (v2) node[below] {$P_{a+1}$};

  \draw[edge] (v0) to [bend left=20] node[above] {$e_1$} (v1);
  \draw[edge] (v1) to [bend left=20] node[below] {$e_2$} (v0);
\end{tikzpicture}\quad \text{ and  }\quad 
G^4=
\begin{tikzpicture}[baseline,vertex/.style={circle, fill,minimum size=4pt, inner sep=0pt, outer sep=0pt},auto,
                               edge/.style={->,>=latex, shorten > = 5pt,shorten < = 5pt},
position/.style args={#1:#2 from #3}{
    at=(#3.#1), anchor=#1+180, shift=(#1:#2)
}]
  \node[vertex] (v0) {};
  \node[vertex,position = 0:{\nodeDist} from v0] (v1) {};
  \node[vertex,position = 0:{\nodeDist} from v1] (v2) {};

  \draw (v0) node[below] {$P_{a-1}$};
  \draw (v1) node[below] {$P_{a}$};
  \draw (v2) node[below] {$P_{a+1}$};

  \draw[edge] (v1) to [bend left=20] node[below] {$e_2$} (v0);
  \draw[edge] (v1) to node[above] {$e_1$} (v2);
\end{tikzpicture}.
\]
The only graphs of the form $G_b$ which admit Eulerian paths from $P_a$, for $G\in U(a,j)$, are
\[G^1_{a+1}=
\begin{tikzpicture}[baseline,vertex/.style={anchor=base, circle, fill,minimum size=4pt, inner sep=0pt, outer sep=0pt},auto,
                               edge/.style={->,>=latex, shorten > = 5pt,shorten < = 5pt},
position/.style args={#1:#2 from #3}{
    at=(#3.#1), anchor=#1+180, shift=(#1:#2)
}]
  \node[vertex] (v0) {};
  \node[vertex,position = 0:{\nodeDist} from v0] (v1) {};
  \node[vertex,position = 0:{\nodeDist} from v1] (v2) {};

  \draw (v0) node[below] {$P_a$};
  \draw (v1) node[below] {$P_{a+1}$};
  \draw (v2) node[below] {$P_{a+2}$};

  \draw[edge] (v0) to node[above] {$e_1$} (v1);
  \draw[edge] (v1) to [bend left=20] node[above] {$e_3$} (v2);
  \draw[edge] (v2) to [bend left=20] node[below] {$e_2$} (v1);
\end{tikzpicture}\quad \text{ , }\quad G^2_a=
\begin{tikzpicture}[baseline,vertex/.style={anchor=base, circle, fill,minimum size=4pt, inner sep=0pt, outer sep=0pt},auto,
                               edge/.style={->,>=latex, shorten > = 5pt,shorten < = 5pt},
position/.style args={#1:#2 from #3}{
    at=(#3.#1), anchor=#1+180, shift=(#1:#2)
}]
  \node[vertex] (v0) {};
  \node[vertex,position = 0:{\nodeDist} from v0] (v1) {};
  \node[vertex,position = 0:{\nodeDist} from v1] (v2) {};

  \draw (v0) node[below] {$P_a$};
  \draw (v1) node[below] {$P_{a+1}$};
  \draw (v2) node[below] {$P_{a+2}$};

  \draw[edge] (v0) to node[above] {$e_3$} (v1);
  \draw[edge] (v1) to [bend left=20] node[above] {$e_1$} (v2);
  \draw[edge] (v2) to [bend left=20] node[below] {$e_2$} (v1);
\end{tikzpicture},
\]
\[G^3_a=
\begin{tikzpicture}[baseline,vertex/.style={circle, fill,minimum size=4pt, inner sep=0pt, outer sep=0pt},auto,
                               edge/.style={->,>=latex, shorten > = 5pt,shorten < = 5pt},
position/.style args={#1:#2 from #3}{
    at=(#3.#1), anchor=#1+180, shift=(#1:#2)
}]
  \node[vertex] (v0) {};
  \node[vertex,position = 0:{\nodeDist} from v0] (v1) {};
  \node[vertex,position = 0:{\nodeDist} from v1] (v2) {};

  \draw (v0) node[below] {$P_{a-1}$};
  \draw (v1) node[below] {$P_{a}$};
  \draw (v2) node[below] {$P_{a+1}$};

  \draw[edge] (v0) to [bend left=20] node[above] {$e_1$} (v1);
  \draw[edge] (v1) to [bend left=20] node[below] {$e_2$} (v0);
  \draw[edge] (v1) to node[above] {$e_3$} (v2);
\end{tikzpicture}\quad \text{ and  }\quad G^4_{a-1}
\begin{tikzpicture}[baseline,vertex/.style={circle, fill,minimum size=4pt, inner sep=0pt, outer sep=0pt},auto,
                               edge/.style={->,>=latex, shorten > = 5pt,shorten < = 5pt},
position/.style args={#1:#2 from #3}{
    at=(#3.#1), anchor=#1+180, shift=(#1:#2)
}]
  \node[vertex] (v0) {};
  \node[vertex,position = 0:{\nodeDist} from v0] (v1) {};
  \node[vertex,position = 0:{\nodeDist} from v1] (v2) {};

  \draw (v0) node[below] {$P_{a-1}$};
  \draw (v1) node[below] {$P_{a}$};
  \draw (v2) node[below] {$P_{a+1}$};

  \draw[edge] (v0) to [bend left=20] node[above] {$e_3$} (v1);
  \draw[edge] (v1) to [bend left=20] node[below] {$e_2$} (v0);
  \draw[edge] (v1) to node[above] {$e_1$} (v2);
\end{tikzpicture}.
\]

Each of the above graphs has a unique Eulerian path from $P_a$ and so $U_{nz}(a,j)=U(a,j)$. For fixed $a\in\{0,1,2\}$, by Lemma \ref{matrix_from_graph}, the only monomials appearing in the $a^{\rm th}$ row of $L_1$ are the following.

\begin{enumerate}
\item $x_{1,a}x_{2,a+2}$ appears with coefficient $-1$ in the $(a,a+1)^{\rm th}$ entry, corresponding to path $(e_1,e_3,e_2)$.
\item $x_{1,a+1}x_{2,a+2}$ appears with coefficient $1$ in the $(a,a)^{\rm th}$ entry, corresponding to path $(e_3,e_1,e_2)$.
\item $x_{1,a-1}x_{2,a}$ appears with coefficient $-1$ in the $(a,a)^{\rm th}$ entry, corresponding to path $(e_2,e_1,e_3)$.
\item $x_{1,a}x_{2,a}$ appears with coefficient $1$ in the $(a,a-1)^{\rm th}$ entry, corresponding to path $(e_2,e_3,e_1)$.
\end{enumerate}

We obtain
\[L_1=\left(\begin{array}{ccc}
x_{1,1}x_{2,2}-x_{1,2}x_{2,0} & -x_{1,0}x_{2,2} & x_{1,0}x_{2,0}\\
x_{1,1}x_{2,1} & x_{1,2}x_{2,0}-x_{1,0}x_{2,1} & -x_{1,1}x_{2,0}\\
-x_{1,2}x_{2,1} & x_{1,2}x_{2,2} & x_{1,0}x_{2,1}-x_{1,1}x_{2,2}\\
\end{array}\right)\ .\]
\end{example}

\section{The matrix of initial coefficients}
\label{sect.initial}

In order to prove Proposition~\ref{prop.red2} (and thus Theorem \ref{main_thm}), we need to bound the nullities
of $L_0, L_1, \ldots, L_{n-1}$ from above. We will not be able to work with the matrices $L_j$ directly; their entries 
are too complicated (recall that these entries are polynomials in the variables $x_{\ell, \alpha}$).
Our approach will be to consider the matrices $\Ic(L_j)$ arising from the initial coefficients of $L_j$ 
with respect to a suitably defined lexicographic monomial order on the variables $x_{\ell,\alpha}$. 
The matrix $\Ic(L_j)$ will turn out to be more manageable than $L_j$ and as we shall soon see, its nullity will give us 
an upper bound on the nullity of $L_j$.

\begin{definition}\label{monom_order_def}
Let $R=F[x_1,\dots,x_t]$ be a polynomial ring.

(a) We define a lexicographic order $\succ$ on $R$ to be a total order on monomials from $R$ induced by an order on the variables $x_1,\dots,x_t$.

(b) Let $f\in R$ and write $f=\sum_{i\in I} c_i m_i$ for nonzero coefficients $c_i\in F$ and distinct monomials $m_i\in R$. If $m$ is the maximal monomial from 
$\{m_i \, | \,  i\in I\}$ with respect to $\succ$, then we define the initial monomial of $f$ to be $\In(f)=m$. 
The leading coefficient of $f$ is defined to be the coefficient of $\In(f)$ in $f$.

(c) We define the matrix $\Ic(M) \in \M_n(F)$ of initial coefficients of an $n \times n$ matrix $M \in \M_n(R)$ as follows. 
For each $a$ let $m_a$ be the largest monomial occurring in the $a^{\rm th}$ row of $M$, that is, 
$m_a=\max(\In(M_{a,0}),\dots,\In(M_{a,n-1}))$. Then the entry $\Ic(M)_{a, b}$ of $\Ic(M)$ 
in position $(a,b)$ is defined to be the coefficient of $m_a$ in $M_{a,b}$.
\end{definition}

\begin{lemma}\label{Ic_det} Let $R$ be a polynomial ring as above, 
and $M$ be an $n \times n$ matrix with coefficients in $R$. Then $\nullity(M) \leqslant \nullity(\Ic(M))$.
\end{lemma}

\begin{proof} Suppose $r = \rank(\Ic(M))$. Then there exists a non-singular $r \times r$ submatrix $N_0$ of $\Ic(M)$.
Let us say that $N_0$ is obtained from $\Ic(M)$ by removing rows $a_1, \ldots, a_{n-r}$ and columns $b_1, \dots, b_{n-r}$. 
Let $N$ be the $r \times r$ submatrix of $M$ obtained by removing the same rows $a_1, \ldots, a_{n-r}$ and 
columns $b_1, \dots, b_{n-r}$. Clearly, $\det(N_0)$ is the leading coefficient of $\det(N)$. Hence,
$\det(N) \neq 0$ and consequently, $\rank(M) \geqslant r = \rank(\Ic(M))$. Equivalently, $\nullity(M) \leqslant \nullity(\Ic(M)) = n- r$.
\end{proof}

Recall that a graph $G\in U(a,j)$ determines a monomial $\mon(G)$ by the source vertices of its edges, as in (\ref{mon_G}).
To determine the $a^{\rm th}$ row of $\Ic(L_j)$ we need only consider the maximal graph from $U_{nz}(a,j)$, where we define $G_1\succ G_2$ if $\mon(G_1)\succ \mon(G_2)$. 

\begin{lemma}\label{Ic_entries}
Assume $U_{nz}(a,j)$ is nonempty, and let $G$ be the maximal graph from $U_{nz}(a,j)$. Then the $(a,b)^{\rm th}$ entry of $\Ic(L_j)$ is
\[(\Ic(L_j))_{a,b}=\sum_{w\in \uni_{P_a}(G_b)}\sgn(w)\ .\]
\end{lemma}

\begin{proof}
Let $m$ be the maximal monomial appearing in the $a^{\rm th}$ row of $L_j$ with nonzero coefficient. Then $\gr(m)$ is the maximal graph in $U_{nz}(a,j)$. Thus $\mon(G)=m$, and the $(a, b)^{\rm th}$-entry of $\Ic(L_j)$ is the coefficient of $\mon(G)$ appearing in the $(a, b)^{\rm th}$-entry of $L_j$. 
By Lemma \ref{matrix_from_graph}, this coefficient is
\[(\Ic(L_j))_{a,b}=\sum_{w\in \uni_{P_a}(G_b)}\sgn(w)\ .\]
\end{proof}

\begin{example}\label{Ic_eg}
Let $k=2$, $n=3$, $j=1$ and $s_1=1$, $s_2=-1$, as in Example \ref{L_1_eg}. Define an order on graphs from $U(a,j)$ to be the lexicographic order induced by the order on edges

\renewcommand{\nodeDist}{1}
\[
\begin{tikzpicture}[baseline,vertex/.style={anchor=base, circle, fill,minimum size=4pt, inner sep=0pt, outer sep=0pt},auto,
                               edge/.style={->,>=latex, shorten > = 5pt,shorten < = 5pt},
position/.style args={#1:#2 from #3}{
    at=(#3.#1), anchor=#1+180, shift=(#1:#2)
}]
  \node[vertex] (v0) at (-0.9,0.1) {};
  \node[vertex] (v1) at (0,0) {};

  \draw (v0) node[below] {$P_{2}$};
  \draw (v1) node[below] {$P_0$};

  \draw[edge] (v0) to node[below] (e) {$e_1$} (v1);
  \node [below=0.4cm, align=flush center] at (e) {$x_{1,2}$};
\end{tikzpicture}\enskip \succ \enskip 
\begin{tikzpicture}[baseline,vertex/.style={anchor=base, circle, fill,minimum size=4pt, inner sep=0pt, outer sep=0pt},auto,
                               edge/.style={->,>=latex, shorten > = 5pt,shorten < = 5pt},
position/.style args={#1:#2 from #3}{
    at=(#3.#1), anchor=#1+180, shift=(#1:#2)
}]
  \node[vertex] (v0) at (-0.9,0.1) {};
  \node[vertex] (v1) at (0,0) {};

  \draw (v0) node[below] {$P_{2}$};
  \draw (v1) node[below] {$P_0$};

  \draw[edge] (v1) to node[below] (e) {$e_2$} (v0);
  \node [below=0.4cm, align=flush center] at (e) {$x_{2,0}$};
\end{tikzpicture}\enskip \succ \enskip 
\begin{tikzpicture}[baseline,vertex/.style={anchor=base, circle, fill,minimum size=4pt, inner sep=0pt, outer sep=0pt},auto,
                               edge/.style={->,>=latex, shorten > = 5pt,shorten < = 5pt},
position/.style args={#1:#2 from #3}{
    at=(#3.#1), anchor=#1+180, shift=(#1:#2)
}]
  \node[vertex] (v0) at (0,0) {};
  \node[vertex] (v1) at (0.9,0.1) {};

  \draw (v0) node[below] {$P_0$};
  \draw (v1) node[below] {$P_1$};

  \draw[edge] (v0) to node[below] (e) {$e_1$} (v1);
  \node [below=0.4cm, align=flush center] at (e) {$x_{1,0}$};
\end{tikzpicture}\enskip \succ \enskip 
\begin{tikzpicture}[baseline,vertex/.style={anchor=base, circle, fill,minimum size=4pt, inner sep=0pt, outer sep=0pt},auto,
                               edge/.style={->,>=latex, shorten > = 5pt,shorten < = 5pt},
position/.style args={#1:#2 from #3}{
    at=(#3.#1), anchor=#1+180, shift=(#1:#2)
}]
  \node[vertex] (v0) at (0,0) {};
  \node[vertex] (v1) at (0.9,0.1) {};

  \draw (v0) node[below] {$P_0$};
  \draw (v1) node[below] {$P_1$};

  \draw[edge] (v1) to node[below] (e) {$e_2$} (v0);
  \node [below=0.4cm, align=flush center] at (e) {$x_{2,1}$};
\end{tikzpicture}\enskip \succ \enskip 
\begin{tikzpicture}[baseline,vertex/.style={anchor=base, circle, fill,minimum size=4pt, inner sep=0pt, outer sep=0pt},auto,
                               edge/.style={->,>=latex, shorten > = 5pt,shorten < = 5pt},
position/.style args={#1:#2 from #3}{
    at=(#3.#1), anchor=#1+180, shift=(#1:#2)
}]
  \node[vertex] (v0) at (-0.5,0.1) {};
  \node[vertex] (v1) at (0.5,0.1) {};

  \draw (v0) node[below] {$P_{2}$};
  \draw (v1) node[below] {$P_1$};

  \draw[edge] (v1) to node[below] (e) {$e_1$} (v0);
  \node [below=0.4cm, align=flush center] at (e) {$x_{1,1}$};
\end{tikzpicture}\enskip \succ \enskip 
\begin{tikzpicture}[baseline,vertex/.style={anchor=base, circle, fill,minimum size=4pt, inner sep=0pt, outer sep=0pt},auto,
                               edge/.style={->,>=latex, shorten > = 5pt,shorten < = 5pt},
position/.style args={#1:#2 from #3}{
    at=(#3.#1), anchor=#1+180, shift=(#1:#2)
}]
  \node[vertex] (v0) at (-0.5,0.1) {};
  \node[vertex] (v1) at (0.5,0.1) {};

  \draw (v0) node[below] {$P_{2}$};
  \draw (v1) node[below] {$P_1$};

  \draw[edge] (v0) to node[below] (e) {$e_2$} (v1);
  \node [below=0.4cm, align=flush center] at (e) {$x_{2,2}$};
\end{tikzpicture}\ .
\]
The correspondence between variables $x_{\ell,\alpha}$ and pairs $(e_\ell,P_\alpha)$ of labeled edges with source vertices determines a monomial order. The graph order defined here corresponds to the lexicographic monomial order on the polynomial ring
$ F[x_{1,0},x_{1,1},x_{1,2},x_{2,0},x_{2,1},x_{2,2}]$
induced by $x_{1,2}\succ x_{2,0}\succ x_{1,0}\succ x_{2,1}\succ x_{1,1}\succ x_{2,2}$. Our maximal graph from $U(a,1)$ is then
\[G=
\begin{tikzpicture}[baseline,vertex/.style={anchor=base, circle, fill,minimum size=4pt, inner sep=0pt, outer sep=0pt},auto,
                               edge/.style={->,>=latex, shorten > = 5pt,shorten < = 5pt},
position/.style args={#1:#2 from #3}{
    at=(#3.#1), anchor=#1+180, shift=(#1:#2)
}]
  \node[vertex] (vL) at (-1.5,0.2) {};
  \node[vertex] (v0) at (0,0) {};
  \node[vertex] (vR) at (1.5,0.2) {};

  \draw (vL) node[below] {$P_{2}$};
  \draw (v0) node[below] {$P_0$};
  \draw (vR) node[below] {$P_1$};

  \draw[edge] (vL) to [bend left=20] node[above] {$e_1$} (v0);
  \draw[edge] (v0) to [bend left=20] node[below] {$e_2$} (vL);
\end{tikzpicture}
\]
for any choice of $a=0,1,2$. This graph $G$ does lie in $U_{nz}(a,1)$ when $a=0,1$, however when $a=2$, the only graph of the form $G_b$ admitting Eulerian paths is $G_2$. But $G_2$ has a repeated edge from $P_2$ to $P_0$. By Lemma \ref{repeated_edges} we have $\sum_{w\in \uni_{P_2}(G_2)}\sgn(w)=0$ and hence $G\notin U_{nz}(2,1)$. We consider the next largest graph from $U(2,1)$ and see that the maximal graphs from $U_{nz}(a,1)$ are
\[
\begin{tikzpicture}[baseline,vertex/.style={anchor=base, circle, fill,minimum size=4pt, inner sep=0pt, outer sep=0pt},auto,
                               edge/.style={->,>=latex, shorten > = 5pt,shorten < = 5pt},
position/.style args={#1:#2 from #3}{
    at=(#3.#1), anchor=#1+180, shift=(#1:#2)
}]
  \node[vertex] (vL) at (-1.5,0.2) {};
  \node[vertex] (v0) at (0,0) {};
  \node[vertex] (vR) at (1.5,0.2) {};

  \draw (vL) node[below] {$P_{2}$};
  \draw (v0) node[below] {$P_0$};
  \draw (vR) node[below] {$P_1$};

  \draw[edge] (vL) to [bend left=20] node[above] {$e_1$} (v0);
  \draw[edge] (v0) to [bend left=20] node[below] {$e_2$} (vL);

  \node [below=0.6cm, align=flush center] at (v0) {$a=0$};
\end{tikzpicture}\quad,\quad
\begin{tikzpicture}[baseline,vertex/.style={anchor=base, circle, fill,minimum size=4pt, inner sep=0pt, outer sep=0pt},auto,
                               edge/.style={->,>=latex, shorten > = 5pt,shorten < = 5pt},
position/.style args={#1:#2 from #3}{
    at=(#3.#1), anchor=#1+180, shift=(#1:#2)
}]
  \node[vertex] (vL) at (-1.5,0.2) {};
  \node[vertex] (v0) at (0,-0.1) {};
  \node[vertex] (vR) at (1.5,0.2) {};

  \draw (vL) node[below] {$P_{2}$};
  \draw (v0) node[below] {$P_0$};
  \draw (vR) node[below] {$P_1$};

  \draw[edge] (vL) to [bend left=20] node[above] {$e_1$} (v0);
  \draw[edge] (v0) to [bend left=20] node[below] {$e_2$} (vL);

  \node [below=0.6cm, align=flush center] at (v0) {$a=1$};
\end{tikzpicture}\quad\text{and}\quad
\begin{tikzpicture}[baseline,vertex/.style={anchor=base, circle, fill,minimum size=4pt, inner sep=0pt, outer sep=0pt},auto,
                               edge/.style={->,>=latex, shorten > = 5pt,shorten < = 5pt},
position/.style args={#1:#2 from #3}{
    at=(#3.#1), anchor=#1+180, shift=(#1:#2)
}]
  \node[vertex] (vL) at (-1.5,0.2) {};
  \node[vertex] (v0) at (0,0) {};
  \node[vertex] (vR) at (1.5,0.2) {};

  \draw (vL) node[below] {$P_{2}$};
  \draw (v0) node[below] {$P_0$};
  \draw (vR) node[below] {$P_1$};

  \draw[edge] (vL) to node[below] {$e_1$} (v0);
  \draw[edge] (vR) to node[below] {$e_2$} (v0);

  \node [below=0.6cm, align=flush center] at (v0) {$a=2$};
\end{tikzpicture}\ .
\]
For each of the above $3$ graphs, there is precisely one placement of the edge $e_3$ such that the graph $G_b=G\ \CMamalg\ \underset {P_b} \bullet \underset {e_3} \rightarrow \underset {P_{b+1}} \bullet$ admits an Eulerian path from $P_a$. These are
\[
\begin{tikzpicture}[baseline,vertex/.style={anchor=base, circle, fill,minimum size=4pt, inner sep=0pt, outer sep=0pt},auto,
                               edge/.style={->,>=latex, shorten > = 5pt,shorten < = 5pt},
position/.style args={#1:#2 from #3}{
    at=(#3.#1), anchor=#1+180, shift=(#1:#2)
}]
  \node[vertex] (vL) at (-1.5,0.2) {};
  \node[vertex] (v0) at (0,0) {};
  \node[vertex] (vR) at (1.5,0.2) {};

  \draw (vL) node[below] {$P_{2}$};
  \draw (v0) node[below] {$P_0$};
  \draw (vR) node[below] {$P_1$};

  \draw[edge] (vL) to [bend left=20] node[above] {$e_1$} (v0);
  \draw[edge] (v0) to [bend left=20] node[below] {$e_2$} (vL);
  \draw[edge] (v0) to node[below] {$e_3$} (vR);

  \node [below=0.6cm, align=flush center] at (v0) {$a=0, b=0$};
\end{tikzpicture}\quad,\quad
\begin{tikzpicture}[baseline,vertex/.style={anchor=base, circle, fill,minimum size=4pt, inner sep=0pt, outer sep=0pt},auto,
                               edge/.style={->,>=latex, shorten > = 5pt,shorten < = 5pt},
position/.style args={#1:#2 from #3}{
    at=(#3.#1), anchor=#1+180, shift=(#1:#2)
}]
  \node[vertex] (vL) at (-1.5,0.2) {};
  \node[vertex] (v0) at (0,-0.1) {};
  \node[vertex] (vR) at (1.5,0.2) {};

  \draw (vL) node[below] {$P_{2}$};
  \draw (v0) node[below] {$P_0$};
  \draw (vR) node[below] {$P_1$};

  \draw[edge] (vL) to [bend left=20] node[above = 2, right = 1] {$e_1$} (v0);
  \draw[edge] (v0) to [bend left=20] node[below] {$e_2$} (vL);
  \draw[edge] (vR) to [bend right=25] node[above] {$e_3$} (vL);

  \node [below=0.6cm, align=flush center] at (v0) {$a=1, b=1$};
\end{tikzpicture}\quad\text{and}\quad
\begin{tikzpicture}[baseline,vertex/.style={anchor=base, circle, fill,minimum size=4pt, inner sep=0pt, outer sep=0pt},auto,
                               edge/.style={->,>=latex, shorten > = 5pt,shorten < = 5pt},
position/.style args={#1:#2 from #3}{
    at=(#3.#1), anchor=#1+180, shift=(#1:#2)
}]
  \node[vertex] (vL) at (-1.5,0.2) {};
  \node[vertex] (v0) at (0,0) {};
  \node[vertex] (vR) at (1.5,0.2) {};

  \draw (vL) node[below] {$P_{2}$};
  \draw (v0) node[below] {$P_0$};
  \draw (vR) node[below] {$P_1$};

  \draw[edge] (vL) to node[below] {$e_1$} (v0);
  \draw[edge] (v0) to [bend left=20] node[above] {$e_3$} (vR);
  \draw[edge] (vR) to [bend left=20] node[below] {$e_2$} (v0);

  \node [below=0.6cm, align=flush center] at (v0) {$a=2, b=0$};
\end{tikzpicture}\ .
\]
By Lemma \ref{Ic_entries} we obtain
\[\Ic(L_1)=\left(\begin{array}{ccc}
-1&0&0\\
0&1&0\\
-1&0&0
\end{array}\right)\]
from which we see that $\nullity(\Ic(L_1))=1$ and hence $\nullity(L_1)\leqslant 1$ by Lemma \ref{Ic_det}.

In our description of $L_1$ in Example \ref{L_1_eg}, the largest monomials in each row are $x_{1,2}x_{2,0}$, $x_{1,2}x_{2,0}$ and $x_{1,2}x_{2,1}$ respectively. Lemma \ref{Ic_entries} allows us to extract the coefficients of these monomials by considering the associated order on graphs from $U(a,j)$.
\end{example}

\section{Specialization of the lexicographic order and the exponents $s_i$}\label{specialization}

 From now on we will assume that $k = 2r$ is even and $1 \leqslant r \leqslant n-1$.
We will now choose our exponents $s_1,\dots,s_k$. These exponents will be fixed for 
the remainder of the proof of Theorem~\ref{main_thm}. 
The matrices $A_\ell=D_\ell C^{s_\ell}$, $\ell = 1, \ldots, k$ and the linear transformations 
$L(A_1,\dots,A_k)$ and $L_j$, $j = 0, 1, \ldots, n-1$ defined in Section~\ref{first_reductions}, will also be fixed.
We define
\begin{align*}
s_i&=\lceil \frac i 2 \rceil, \text{ for $i=1,2,\dots,r$, and}\\
s_{r+i}&=s_{-i}=-s_i=-\lceil \frac i 2 \rceil, \text{ for $i=1,2,\dots,r$}\ .
\end{align*}
In other words,
\begin{align*}
(s_1,s_2,s_3,\dots,s_k)&=(s_1,s_2,s_3,\dots,s_r,s_{-1},s_{-2},\dots,s_{-r})\\
&=\left(\lceil \frac {1} {2} \rceil,\lceil \frac {2} {2} \rceil,\lceil \frac {3} {2} \rceil,\dots,\lceil \frac {r} {2} \rceil,-\lceil \frac {1} {2} \rceil,-\lceil \frac {2} {2} \rceil,\dots,-\lceil \frac {r} {2} \rceil\right)\\
&=(1,1,2,\dots,\lceil\frac r 2 \rceil,-1,-1,\dots,-\lceil \frac r 2 \rceil)
\end{align*}
and our matrices $A_1, \ldots, A_k$ specialize to
\begin{gather*} A_1=D_1C,\quad A_2=D_2C, \quad A_3=D_3C^2, \quad \ldots \quad , \quad  
A_r=D_rC^{\lceil \frac r 2\rceil}, \\
A_{r+1}=D_{r+1}C^{-1}, \quad  A_{r+2} = D_{r + 2} C^{-1}, \quad \quad \dots \quad ,   \quad A_k=D_kC^{-\lceil \frac r 2 \rceil}\ .\end{gather*}

\begin{nc} In the sequel
$[a\interval b]$ will denote an `interval' in $\Z/n\Z$, i.e., the set of successive integers $a,a+1,\dots$ in $\Z/n\Z$ up to the first integer congruent to $b$ modulo $n$.
\end{nc}

\begin{definition}\label{graph_order_def}
For $v\in \Z/n\Z$ we let $|v|$ denote the absolute value of the representative of $v$ 
in $[-\lceil \frac{n-1}{2} \rceil\interval \lfloor \frac{n-1}{2} \rfloor]$. For two vertices $P_{v_1}, P_{v_2}$, $v_1,v_2\in \Z/n\Z$ we define
\[
P_{v_1}\succ P_{v_2}, \text{ if $|v_1|<|v_2|$, or if $v_1=-v_2$ and $v_2\in[1\interval \lfloor \frac{n-1}{2} \rfloor]$.}
\]
Our order on vertices becomes
\begin{equation}\label{vertex_order}
P_0\succ P_{-1}\succ P_1\succ P_{-2}\succ P_2\succ\dots\ .
\end{equation}

This defines a lexicographic order on unordered pairs of vertices. If $\max(P_{v_1},P_{v_2})\succ \max(P_{v_1'},P_{v_2'})$, or $\max(P_{v_1},P_{v_2})= \max(P_{v_1'},P_{v_2'})$ and $\min(P_{v_1},P_{v_2})\succ \min(P_{v_1'},P_{v_2'})$, then we define $(P_{v_1},P_{v_2})\succ (P_{v_1'},P_{v_2'})$.

Next we define an order on pairs of labeled edges with their source vertices. For $\ell_1,\ell_2\in\{1,2,\dots,k\}$ and $v_1,v_2\in \Z/n\Z$ we define
\begin{equation}\label{edge_order}
\begin{split}
(e_{\ell_1},P_{v_1})\succ (e_{\ell_2},P_{v_2}),\quad &\text{ if }\ (P_{v_1}, P_{v_1+s_{\ell_1}})\succ (P_{v_2}, P_{v_2+s_{\ell_2}})\ ,\\
(e_{\ell_1},P_{v_1})\succ (e_{\ell_2},P_{v_2}),\quad  &\text{ if }\  (P_{v_1}, P_{v_1+s_{\ell_1}})= (P_{v_2}, P_{v_2+s_{\ell_2}})\ , \text{ and } \ell_1<\ell_2.
\end{split}
\end{equation}
The equality in the second line of this definition is of unordered pairs. Note that for $n>r$, the equality $(P_{v_1}, P_{v_1+s_{\ell_1}})= (P_{v_2}, P_{v_2+s_{\ell_2}})$ can only hold if $s_{\ell_1}= \pm s_{\ell_2}$.



This order on pairs of labeled edges and vertices determines an order on graphs from $U(a,j,I)$. A graph $G\in U(a,j,I)$ is determined by the pairs $(e_\ell,\src_G(e_\ell))$ for $\ell\in I$. Our order on $U(a,j,I)$ is the lexicographic order induced by the order on these pairs in (\ref{edge_order}).
\end{definition}

In a graph of the form $G=G(m)$, an edge $e_\ell$ will satisfy $\tar_G(e_\ell)=\src_G(e_\ell)+s_\ell$. Thus when comparing edges $e\in G$ and $e'\in G'$, the order given in (\ref{edge_order}) first compares the unordered pairs $(\src_G(e),\tar_G(e))$ and $(\src_{G'}(e'),\tar_{G'}(e'))$, with ties broken by comparing edge labels.

Our global pictures of the $n$ vertices used in the graphs of $U(a,j)$, when $n$ is even and odd respectively, are
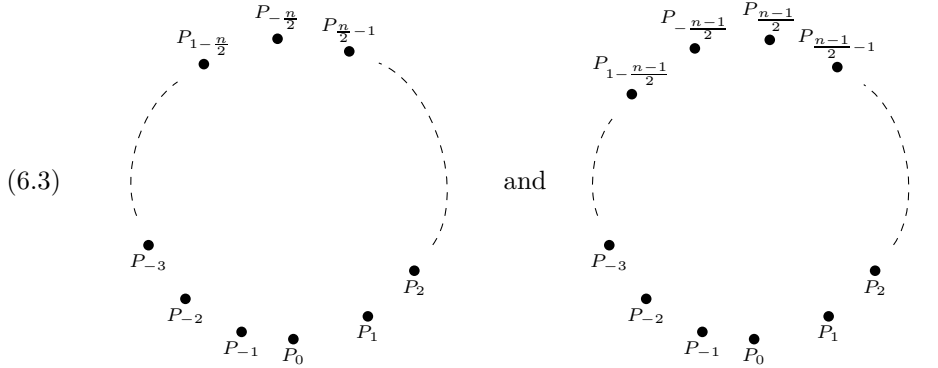
\begin{equation}\label{global_pic}
\renewcommand{\nodeDist}{2}
\begin{tikzpicture}[baseline,vertex/.style={anchor=base, circle, fill,minimum size=4pt, inner sep=0pt, outer sep=0pt},auto,
                               edge/.style={->,>=latex, shorten > = 5pt,shorten < = 5pt},
position/.style args={#1:#2 from #3}{
    at=(#3.#1), anchor=#1+180, shift=(#1-90:#2)
}]

  \coordinate (vc) {};
  \node[vertex,position = 0:{\nodeDist} from vc] (v0) {};
  \node[vertex,position = -20:{\nodeDist} from vc] (v1L) {};
  \node[vertex,position = -45:{\nodeDist} from vc] (v2L) {};
  \node[vertex,position = -70:{\nodeDist} from vc] (v3L) {};
  \node[vertex,position = 30:{\nodeDist} from vc] (v1R) {};
  \node[vertex,position = 55:{\nodeDist} from vc] (v2R) {};

  \node[vertex,position = -178:{\nodeDist} from vc] (vn1L) {};
  \node[vertex,position = -148:{\nodeDist} from vc] (vn2L) {};
  \node[vertex,position = 154:{\nodeDist} from vc] (vn1R) {};

  \draw (v0) node[below] {\scriptsize $P_{0}$};
  \draw (v1L) node[below] {\scriptsize $P_{-1}$};
  \draw (v2L) node[below] {\scriptsize $P_{-2}$};
  \draw (v3L) node[below] {\scriptsize $P_{-3}$};
  \draw (v1R) node[below] {\scriptsize $P_{1}$};
  \draw (v2R) node[below] {\scriptsize $P_{2}$};

  \draw (vn1L) node[above] {\scriptsize $P_{-\frac n 2}$};
  \draw (vn2L) node[above] {\scriptsize $P_{1-\frac n 2}$};
  \draw (vn1R) node[above] {\scriptsize $P_{\frac n 2 - 1}$};

  \draw[dashed, shorten >=10pt, shorten < = 10 pt] (v3L) to[bend left = 39] (vn2L);
  \draw[dashed, shorten >=10pt, shorten < = 10 pt] (v2R) to[bend right=52] (vn1R);
\end{tikzpicture}\quad\text{ and }\quad
\begin{tikzpicture}[baseline,vertex/.style={anchor=base, circle, fill,minimum size=4pt, inner sep=0pt, outer sep=0pt},auto,
                               edge/.style={->,>=latex, shorten > = 5pt,shorten < = 5pt},
position/.style args={#1:#2 from #3}{
    at=(#3.#1), anchor=#1+180, shift=(#1-90:#2)
}]

  \coordinate (vc) {};
  \node[vertex,position = 0:{\nodeDist} from vc] (v0) {};
  \node[vertex,position = -20:{\nodeDist} from vc] (v1L) {};
  \node[vertex,position = -45:{\nodeDist} from vc] (v2L) {};
  \node[vertex,position = -70:{\nodeDist} from vc] (v3L) {};
  \node[vertex,position = 30:{\nodeDist} from vc] (v1R) {};
  \node[vertex,position = 55:{\nodeDist} from vc] (v2R) {};

  \node[vertex,position = -161:{\nodeDist} from vc] (vn1L) {};
  \node[vertex,position = -131:{\nodeDist} from vc] (vn2L) {};
  \node[vertex,position = 170:{\nodeDist} from vc] (vn1R) {};
  \node[vertex,position = 142:{\nodeDist} from vc] (vn2R) {};

  \draw (v0) node[below] {\scriptsize $P_{0}$};
  \draw (v1L) node[below] {\scriptsize $P_{-1}$};
  \draw (v2L) node[below] {\scriptsize $P_{-2}$};
  \draw (v3L) node[below] {\scriptsize $P_{-3}$};
  \draw (v1R) node[below] {\scriptsize $P_{1}$};
  \draw (v2R) node[below] {\scriptsize $P_{2}$};

  \draw (vn1L) node[above] {\scriptsize $P_{-\frac {n-1} 2}$};
  \draw (vn2L) node[above] {\scriptsize $P_{1-\frac {n-1} 2}$};
  \draw (vn1R) node[above] {\scriptsize $P_{\frac {n-1} 2}$};
  \draw (vn2R) node[above] {\scriptsize $P_{\frac {n-1} 2 - 1}$};

  \draw[dashed, shorten >=10pt, shorten < = 10 pt] (v3L) to[bend left = 30] (vn2L);
  \draw[dashed, shorten >=10pt, shorten < = 10 pt] (v2R) to[bend right=46] (vn2R);
\end{tikzpicture}\enskip .
\end{equation}

In these pictures, the lower vertices have larger weight, corresponding to our order on vertices $P_0\succ P_{-1}\succ P_1\succ P_{-2}\succ\dots$. The largest edges will be those incident on $P_0$ and, more generally, the lower an edge appears in the above picture, the larger weight it has.

The graphs $G\in U(a,j)$ are defined so that there exists an Eulerian path from $P_a$ on $G_b$ for some $b$. Conversely, a graph in $U(a,j)$ can be defined by a path $(e_{\sigma(1)},e_{\sigma(2)},\dots,e_{\sigma(k+1)})$ on the $n$ vertices. In the largest graph from $U(a,j)$, this path will reach the lowest possible vertex in (\ref{global_pic}).

Intuitively, if $P_a$ is to the left of $P_0$ in (\ref{global_pic}), we would suspect an Eulerian path on $G_b$ from $P_a$ on a maximal graph $G\in U(a,j)$ to traverse down the left side of the picture as low as possible. If such a path can reach $P_0$, then the remaining edges will be filled in to be incident to $P_0$. This intuition will be formalized in the next section.

The order given in Example \ref{Ic_eg} is equivalent to the order defined above when $k=2$ and $n=3$. The unordered pairs of vertices satisfy
\[(P_0,P_{-1})\succ(P_0,P_1)\succ (P_{-1},P_1)\ ,\]
and if the two edges share both source and target vertices we break ties by declaring $e_1\succ e_2$.

\smallskip
Our proof of Proposition~\ref{prop.red2} will be based on the following.

\begin{proposition}\label{Ic_L_j_ker}
Assume that $n>r$ and that the base field $F$ is infinite of characteristic not dividing $r!$.

\smallskip
(a) If $j \neq 0$ then \[ \nullity (\Ic(L_j))\leqslant \delta_j, \]
where
\[ \delta_j = \begin{cases} \text{$2$, if both $j$ and $-j$ lie in $[\lceil r/2 \rceil \interval \lfloor r/2 \rfloor]$, } \\
\text{$1$, if exactly one of $j, -j$ lies in $[\lceil r/2 \rceil \interval \lfloor r/2 \rfloor]$,} \\
\text{$0$, if neither $j$ nor $-j$ lie in $[\lceil r/2 \rceil \interval \lfloor r/2 \rfloor]$.}
\end{cases}
\]

\smallskip
(b) Assume further that the characteristic of $F$ does not divide $2(2r+1)r!$. Then the $n \times n$ matrix $\Ic(L_0)$ is non-singular.
\end{proposition}

Here $\Ic(L_j)$ is the matrix of initial coefficients of $L_j$ with respect to the order on graphs described in Definition \ref{graph_order_def}.

To see that Proposition~\ref{Ic_L_j_ker} implies Proposition~\ref{prop.red2} (and thus Theorem~\ref{main_thm}),
assume for a moment that Proposition~\ref{Ic_L_j_ker} is established. Then by Lemma~\ref{Ic_det}, 
\[ \nullity(L_0) + \nullity(L_1) + \ldots \nullity(L_{n-1}) \leqslant
\nullity(\Ic(L_0)) + \nullity(\Ic(L_1)) + \ldots \nullity(\Ic(L_{n-1})) \leqslant \sum_{j = 1}^{n-1} \delta_j. \]
Each $0 \neq a \in [\lceil r/2 \rceil \interval \lfloor r/2 \rfloor]$ contributes exactly $2$ to the sum
$\sum_{j = 1}^{n-1} \delta_j$, one when $j = a$ and one when $-j = a$. (Note that $a$ contributes $2$ to this sum even if $a = -a$ in $\Z/ n \Z$.)
Since there are exactly $r$ non-zero elements $a$ in the interval $[\lceil r/2 \rceil \interval \lfloor r/2 \rfloor]$, we conclude that
$\sum_{j = 1}^{n-1} \delta_j = 2r = k$. Substituting $k$ for $\sum_{j = 1}^{n-1} \delta_j$ into the above inequality, we obtain 
\[ \nullity(L_0) + \nullity(L_1) + \ldots \nullity(L_{n-1}) \leqslant k, \]
and Proposition~\ref{prop.red2} follows.
\qed

\section{Maximal graphs} \label{sect.maximal}

As a first step towards proving Proposition~\ref{Ic_L_j_ker}, we will now describe the maximal graph of $U(a,j)$ for fixed $a$ and $j$
under the ordering defined in Section~\ref{specialization}. This is a purely graph-theoretic problem. The answer is given by 
Proposition~\ref{max_graph_decomp}, whose proof will be completed in the next section.

In the definition of $U(a,j)$ our $k$ edges were given labels $e_1,\dots,e_k$. 
In the sequence it will be convenient for us to use the following alternative labels with negative indices:
\[e_{-1}=e_{r+1},e_{-2}=e_{r+2},\dots,e_{-r}=e_k\ .\]
The graphs of $U(a,j)$ will then be defined by the source vertices $v_i$ of the $2r$ edges of the form
\[\underset {P_{v_i}} \bullet \underset {e_i} \longrightarrow \underset {P_{v_i+s_i}} \bullet\]
as $i$ ranges over $\{\pm 1,\pm 2,\dots,\pm r\}$, where $s_i=\lceil \frac i 2 \rceil$ and $s_{-i}=-\lceil \frac i 2 \rceil$ for $i=1,2,\dots,r$.

For each $t\geqslant 0$ let $H_t$ be the graph consisting of the $t$ pairs of edges of the form
\[
\begin{tikzpicture}[baseline,vertex/.style={anchor=base, circle, fill,minimum size=4pt, inner sep=0pt, outer sep=0pt},auto,
                               edge/.style={->,>=latex, shorten > = 5pt,shorten < = 5pt},
position/.style args={#1:#2 from #3}{
    at=(#3.#1), anchor=#1+180, shift=(#1:#2)
}]
  \node[vertex] (v1) {};
  \node[vertex,position = 0:{\nodeDist} from v1] (v2) {};

  \draw (v1) node[below = 13] {$P_{i}$};
  \draw (v2) node[below = 13] {$P_{0}$};

  \draw[edge] (v1) to [bend left=10] node[above] {$e_{1-2i}$} (v2);
  \draw[edge] (v2) to [bend left=15] node[below] {$e_{2i-1}$} (v1);
\end{tikzpicture}
\quad\text{if $i<0$,}\ \ \text{and}\quad
\begin{tikzpicture}[baseline,vertex/.style={anchor=base, circle, fill,minimum size=4pt, inner sep=0pt, outer sep=0pt},auto,
                               edge/.style={->,>=latex, shorten > = 5pt,shorten < = 5pt},
position/.style args={#1:#2 from #3}{
    at=(#3.#1), anchor=#1+180, shift=(#1:#2)
}]
  \node[vertex] (v1) {};
  \node[vertex,position = 0:{\nodeDist} from v1] (v2) {};

  \draw (v1) node[below = 13] {$P_{0}$};
  \draw (v2) node[below = 13] {$P_{i}$};

  \draw[edge] (v1) to [bend right=15] node[below] {$e_{2i}$} (v2);
  \draw[edge] (v2) to [bend right=10] node[above] {$e_{-2i}$} (v1);
\end{tikzpicture}
\quad\text{if $i>0$,}
\]
as $i$ ranges over $[ - \lceil t/2 \rceil \interval \lfloor t/2 \rfloor ]\setminus \{0\}$.
When $t=0$, $H_0$ is the empty graph with no edges, and for each $t$, $H_{t+1}$ is the disjoint union of $H_t$ 
with a pair of edges connecting $P_0$ to the next largest vertex. For example, $H_5$ is the graph
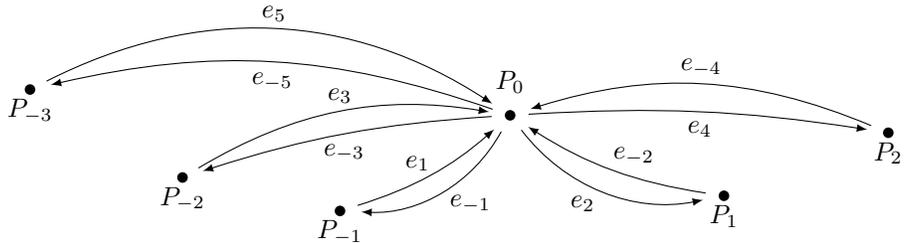
\[
\begin{tikzpicture}[baseline, vertex/.style={anchor = base, circle, fill,minimum size=4pt, inner sep=0pt, outer sep=0pt},auto,
                               edge/.style={->,>=latex, shorten > = 6pt,shorten < = 5pt},
position/.style args={#1:#2 from #3}{
    at=(#3.#1), anchor=#1+180, shift=(#1:#2)
}]

  \node[vertex] (v0) {};
  \coordinate[position=270:{\nodeDist*0.7} from v0]  (z0) {};
  \node[vertex,position = 175:{\nodeDist*1.1} from z0] (v1L) {};
  \node[vertex,position = 168:{\nodeDist} from v1L] (v2L) {};
  \node[vertex,position = 150:{\nodeDist*1.1} from v2L] (v3L) {};
  \node[vertex,position = 8:{\nodeDist*1.4} from z0] (v1R) {};
  \node[vertex,position = 21:{\nodeDist*1.1} from v1R] (v2R) {};

  \draw (v0) node[above=5] {$P_0$};
  \draw (v1L) node[below] {$P_{-1}$};
  \draw (v2L) node[below] {$P_{-2}$};
  \draw (v3L) node[below] {$P_{-3}$};
  \draw (v1R) node[below] {$P_{1}$};
  \draw (v2R) node[below] {$P_{2}$};

  \draw[edge] (v0) to [bend left=32] node[below=5, right=0] {$e_{-1}$} (v1L);
  \draw[edge] (v1L) to [bend right=12] node[above=4, left=-3] {$e_1$} (v0);

  \draw[edge] (v0) to [bend right=7] node[below] {$e_{-3}$} (v2L);
  \draw[edge] (v2L) to [bend left=20] node[above] {$e_{3}$} (v0);


  \draw[edge] (v0) to [bend right=16] node[below] {$e_{-5}$} (v3L);
  \draw[edge] (v3L) to [bend left=30] node[above] {$e_{5}$} (v0);
  

  \draw[edge] (v0) to [bend right=32] node[below=5, left=0] {$e_2$} (v1R);
  \draw[edge] (v1R) to [bend left=12] node[above=5, right=-3] {$e_{-2}$} (v0);


  \draw[edge] (v0) to [bend left=6] node[below] {$e_{4}$} (v2R);
  \draw[edge] (v2R) to [bend right=20] node[above] {$e_{-4}$} (v0);

\end{tikzpicture}\ .
\]
The following lemma shows that the maximal graph from $U(a,j)$ will contain $H_t$ for the largest possible $t$.

\begin{lemma}\label{max_subgraphs}
Let $G,G'\in U(a,j)$. If there exists $t$ such that $G \supset H_t$ and $G'\not\supset H_t$, then $G \succ G'$.
\end{lemma}

\begin{proof}
Our order on edges follows the lexicographic order on pairs of vertices, with
\[(P_0,P_{-1})\succ (P_0,P_1)\succ (P_0,P_{-2})\succ (P_0,P_2)\succ \dots\]
being the maximal pairs.
There are at most $2$ edges connecting any pair of vertices in a graph with no repeated edges, and if two edges connect the same pair of vertices, 
ties are broken using the following order on the edge labels:
\[ \text{$e_{2\alpha+1}\succ e_{2\alpha+2}\succ e_{-2\alpha-1}\succ e_{-2\alpha-2}$ for any $\alpha\geqslant 0$.} \]
We conclude that the largest $2t$ edges that could possibly appear in a graph $G \in U(a, j)$
are those in $H_t$, and the lemma follows.
\end{proof}

Our next goal is to prove the following proposition, which describes the maximal graph of $U(a,j)$.

\begin{proposition}\label{max_graph_decomp}
Let $t$ be the largest integer such that there exists $G\in U(a,j)$ containing $H_t$.  Let $P_{a'}=\max(P_a,P_{a+j})$.

\smallskip
(a) If $a'\in[-\lceil \frac{n-1}{2} \rceil \interval 0]$, then define $\hat G$ to be the graph

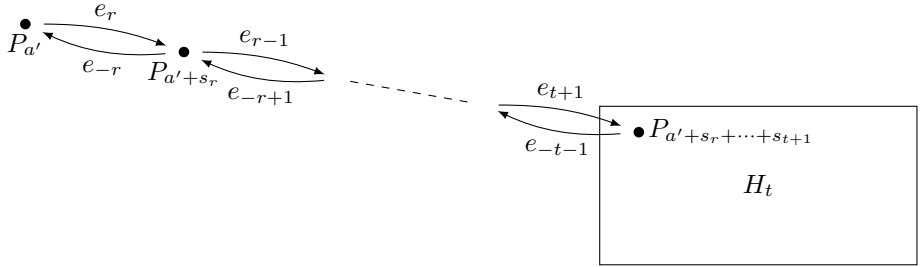
\renewcommand{\nodeDist}{2}
\[
\begin{tikzpicture}[vertex/.style={circle, fill,minimum size=4pt, inner sep=0pt, outer sep=0pt},auto,
                               edge/.style={->,>=latex, shorten > = 5pt,shorten < = 5pt},
position/.style args={#1:#2 from #3}{
    at=(#3.#1), anchor=#1+180, shift=(#1:#2)
}]
  \node[vertex] (v1) {};
  \node[vertex,position = -10:{\nodeDist} from v1] (v2) {};
  \coordinate[position = -10:{\nodeDist} from v2] (v3);
  \coordinate[position = -10:{\nodeDist} from v3] (v4) {};
  \node[vertex,position = -10:{\nodeDist} from v4] (v5) {};
  \node[draw, anchor=north west,rectangle, minimum height=60pt, minimum width=120pt,xshift=-15pt, yshift=10pt] (Hm) at (v5) {$H_t$};

  \draw (v1) node[below] {$P_{a'}$};
  \draw (v2) node[below] {$P_{a'+s_{r}}$};
  \draw (v5) node[right] {$P_{a'+s_{r}+\dots+s_{t+1}}$};

  \draw[edge] (v1) to [bend left=10] node[above] {$e_{r}$} (v2);
  \draw[edge] (v2) to [bend left=15] node[below] {$e_{-r}$} (v1);

  \draw[edge] (v2) to [bend left=10] node[above] {$e_{r-1}$} (v3);
  \draw[edge] (v3) to [bend left=15] node[below] {$e_{-r+1}$} (v2);

  \draw[dashed,shorten > = 5pt, shorten < = 5pt] (v3) -- (v4);

  \draw[edge] (v4) to [bend left=10] node[above] {$e_{t+1}$} (v5);
  \draw[edge] (v5) to [bend left=15] node[below] {$e_{-t-1}$} (v4);
\end{tikzpicture}\enskip .
\]
In other words, $\hat G=H_t\ \CMamalg\ \hat G'$ where
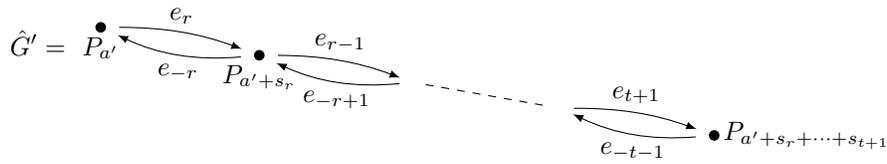
\renewcommand{\nodeDist}{2}
\[\hat G'=
\begin{tikzpicture}[baseline,vertex/.style={anchor=base, circle, fill,minimum size=4pt, inner sep=0pt, outer sep=0pt},auto,
                               edge/.style={->,>=latex, shorten > = 5pt,shorten < = 5pt},
position/.style args={#1:#2 from #3}{
    at=(#3.#1), anchor=#1+180, shift=(#1:#2)
}]
  \node[vertex] (v2) {};
  \node[vertex,position = 170:{\nodeDist} from v2] (v1) {};
  \coordinate[position = -10:{\nodeDist} from v2] (v3);
  \coordinate[position = -10:{\nodeDist} from v3] (v4) {};
  \node[vertex,position = -10:{\nodeDist} from v4] (v5) {};

  \draw (v1) node[below] {$P_{a'}$};
  \draw (v2) node[below] {$P_{a'+s_{r}}$};
  \draw (v5) node[right] {$P_{a'+s_{r}+\dots+s_{t+1}}$};

  \draw[edge] (v1) to [bend left=10] node[above] {$e_{r}$} (v2);
  \draw[edge] (v2) to [bend left=15] node[below] {$e_{-r}$} (v1);

  \draw[edge] (v2) to [bend left=10] node[above] {$e_{r-1}$} (v3);
  \draw[edge] (v3) to [bend left=15] node[below] {$e_{-r+1}$} (v2);

  \draw[dashed,shorten > = 5pt, shorten < = 5pt] (v3) -- (v4);

  \draw[edge] (v4) to [bend left=10] node[above] {$e_{t+1}$} (v5);
  \draw[edge] (v5) to [bend left=15] node[below] {$e_{-t-1}$} (v4);
\end{tikzpicture}\enskip .
\]

(b) If $a'\in[1\interval \lfloor\frac{n-1}{2} \rfloor]$, then define $\hat G$ to be the graph

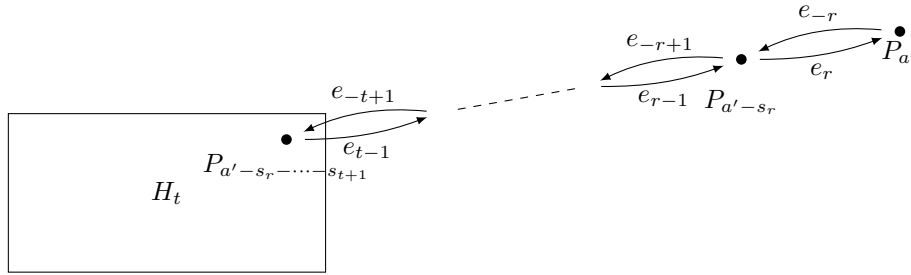
\renewcommand{\nodeDist}{2}
\[
\begin{tikzpicture}[vertex/.style={circle, fill,minimum size=4pt, inner sep=0pt, outer sep=0pt},auto,
                               edge/.style={->,>=latex, shorten > = 5pt,shorten < = 5pt},
position/.style args={#1:#2 from #3}{
    at=(#3.#1), anchor=#1+180, shift=(#1:#2)
}]

  \node[vertex] (v1) {};
  \node[vertex,position = 190:{\nodeDist} from v1] (v2) {};
  \coordinate[position = 190:{\nodeDist} from v2] (v3);
  \coordinate[position = 190:{\nodeDist} from v3] (v4) {};
  \node[vertex,position = 190:{\nodeDist} from v4] (v5) {};
  \node[draw, anchor=north east,rectangle, minimum height=60pt, minimum width=120pt,xshift=15pt, yshift=10pt] (Hm) at (v5) {$H_t$};

  \draw (v1) node[below] {$P_{a'}$};
  \draw (v2) node[below = 8] {$P_{a'-s_{r}}$};
  \draw (v5) node[below = 2] {$P_{a'-s_{r}-\dots-s_{t+1}}$};

  \draw[edge] (v1) to [bend right=15] node[above] {$e_{-r}$} (v2);
  \draw[edge] (v2) to [bend right=10] node[below] {$e_{r}$} (v1);

  \draw[edge] (v2) to [bend right=15] node[above] {$e_{-r+1}$} (v3);
  \draw[edge] (v3) to [bend right=10] node[below] {$e_{r-1}$} (v2);

  \draw[dashed,shorten > = 5pt, shorten < = 5pt] (v3) -- (v4);

  \draw[edge] (v4) to [bend right=15] node[above] {$e_{-t+1}$} (v5);
  \draw[edge] (v5) to [bend right=10] node[below] {$e_{t-1}$} (v4);
\end{tikzpicture}\enskip .
\]
In other words, $\hat G=H_t\ \CMamalg\ \hat G'$ where
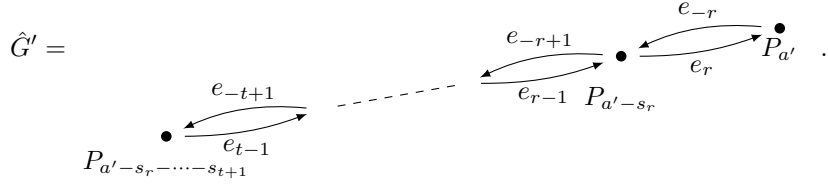
\renewcommand{\nodeDist}{2}
\[\hat G'=
\begin{tikzpicture}[baseline,vertex/.style={anchor=base, circle, fill,minimum size=4pt, inner sep=0pt, outer sep=0pt},auto,
                               edge/.style={->,>=latex, shorten > = 5pt,shorten < = 5pt},
position/.style args={#1:#2 from #3}{
    at=(#3.#1), anchor=#1+180, shift=(#1:#2)
}]

  \node[vertex] (v2) {};
  \node[vertex,position = 10:{\nodeDist} from v2] (v1) {};
  \coordinate[position = 190:{\nodeDist} from v2] (v3);
  \coordinate[position = 190:{\nodeDist} from v3] (v4) {};
  \node[vertex,position = 190:{\nodeDist} from v4] (v5) {};

  \draw (v1) node[below] {$P_{a'}$};
  \draw (v2) node[below = 8] {$P_{a'-s_{r}}$};
  \draw (v5) node[below = 2] {$P_{a'-s_{r}-\dots-s_{t+1}}$};

  \draw[edge] (v1) to [bend right=15] node[above] {$e_{-r}$} (v2);
  \draw[edge] (v2) to [bend right=10] node[below] {$e_{r}$} (v1);

  \draw[edge] (v2) to [bend right=15] node[above] {$e_{-r+1}$} (v3);
  \draw[edge] (v3) to [bend right=10] node[below] {$e_{r-1}$} (v2);

  \draw[dashed,shorten > = 5pt, shorten < = 5pt] (v3) -- (v4);

  \draw[edge] (v4) to [bend right=15] node[above] {$e_{-t+1}$} (v5);
  \draw[edge] (v5) to [bend right=10] node[below] {$e_{t-1}$} (v4);
\end{tikzpicture}\enskip .
\]
Then $\hat G$ is the maximal graph from $U(a,j)$.
\end{proposition}

This result is a key step in our proof of Proposition~\ref{Ic_L_j_ker} and thus of Theorem~\ref{main_thm}.
The remainder of this section will be devoted to proving a preparatory lemma, Lemma~\ref{max_connecting_path}. It
describes the maximal graphs in $U(a, j, I)$, where $I$ is a subset of $\{ 1, 2, \ldots, r \} \cup \{-1,-2,\dots,-r\}$, under certain conditions on $a$, $j$ and $I$.
In the subsequent application $\{ e_i \, | \, i \in I \}$ will be the set of edges that are not used in $H_t$.
We will use Lemma~\ref{max_connecting_path} to complete the proof of Proposition~\ref{max_graph_decomp}
in Section~\ref{sect.maximal_2}.

\begin{lemma}\label{max_connecting_path}
Let $I=\{i_1,i_2,\dots,i_\alpha,-i_1,-i_2,\dots,-i_\alpha\}\subseteq \{1,2,\dots,r\}\cup\{-1,-2,\dots,-r\}$ have corresponding edges labeled $e_{i_1},e_{i_2},\dots,e_{i_\alpha},e_{-i_1},e_{-i_2},\dots,e_{-{i_\alpha}}$ for some $i_\alpha > \dots > i_1>0$. Assume that $s_{i_1}+s_{i_2}+\cdots + s_{i_\alpha}\leqslant|a|,|a+j|$, $P_{a'}=\max(P_a,P_{a+j})$, and $G$ is the maximal graph in $U(a,j,I)$.

(a) If $a'\in [-\lceil\frac{n-1}{2} \rceil\interval 0]$, then $G$ is the graph
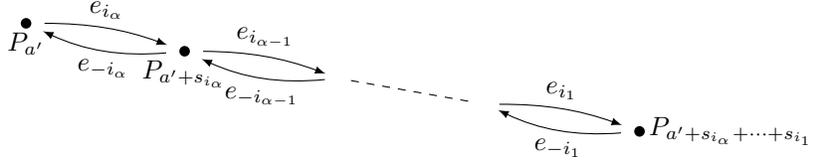
\[
\begin{tikzpicture}[vertex/.style={circle, fill,minimum size=4pt, inner sep=0pt, outer sep=0pt},auto,
                               edge/.style={->,>=latex, shorten > = 5pt,shorten < = 5pt},
position/.style args={#1:#2 from #3}{
    at=(#3.#1), anchor=#1+180, shift=(#1:#2)
}]

  \node[vertex] (v1) {};
  \node[vertex,position = -10:{\nodeDist} from v1] (v2) {};
  \coordinate[position = -10:{\nodeDist} from v2] (v3);
  \coordinate[position = -10:{\nodeDist} from v3] (v4) {};
  \node[vertex,position = -10:{\nodeDist} from v4] (v5) {};

  \draw (v1) node[below] {$P_{a'}$};
  \draw (v2) node[below] {$P_{a'+s_{i_{\alpha}}}$};
  \draw (v5) node[right] {$P_{a'+s_{i_{\alpha}}+\dots+s_{i_1}}$};

  \draw[edge] (v1) to [bend left=10] node[above] {$e_{i_\alpha}$} (v2);
  \draw[edge] (v2) to [bend left=15] node[below] {$e_{-i_\alpha}$} (v1);

  \draw[edge] (v2) to [bend left=10] node[above] {$e_{i_{\alpha-1}}$} (v3);
  \draw[edge] (v3) to [bend left=15] node[below] {$e_{-i_{\alpha-1}}$} (v2);

  \draw[dashed,shorten > = 5pt, shorten < = 5pt] (v3) -- (v4);

  \draw[edge] (v4) to [bend left=10] node[above] {$e_{i_1}$} (v5);
  \draw[edge] (v5) to [bend left=15] node[below] {$e_{-i_1}$} (v4);
\end{tikzpicture}\ .
\]

(b) If $a'\in[1\interval \lfloor\frac{n-1}2\rfloor]$, then $G$ is the graph
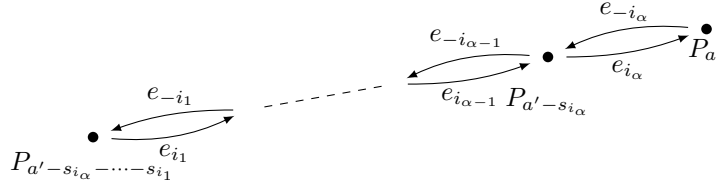
\[
\begin{tikzpicture}[vertex/.style={circle, fill,minimum size=4pt, inner sep=0pt, outer sep=0pt},auto,
                               edge/.style={->,>=latex, shorten > = 5pt,shorten < = 5pt},
position/.style args={#1:#2 from #3}{
    at=(#3.#1), anchor=#1+180, shift=(#1:#2)
}]

  \node[vertex] (v1) {};
  \node[vertex,position = 190:{\nodeDist} from v1] (v2) {};
  \coordinate[position = 190:{\nodeDist} from v2] (v3);
  \coordinate[position = 190:{\nodeDist} from v3] (v4) {};
  \node[vertex,position = 190:{\nodeDist} from v4] (v5) {};

  \draw (v1) node[below] {$P_{a'}$};
  \draw (v2) node[below = 8] {$P_{a'-s_{i_{\alpha}}}$};
  \draw (v5) node[below = 2] {$P_{a'-s_{i_{\alpha}}-\dots-s_{i_1}}$};

  \draw[edge] (v1) to [bend right=15] node[above] {$e_{-i_\alpha}$} (v2);
  \draw[edge] (v2) to [bend right=10] node[below] {$e_{i_\alpha}$} (v1);

  \draw[edge] (v2) to [bend right=15] node[above] {$e_{-i_{\alpha-1}}$} (v3);
  \draw[edge] (v3) to [bend right=10] node[below] {$e_{i_{\alpha-1}}$} (v2);

  \draw[dashed,shorten > = 5pt, shorten < = 5pt] (v3) -- (v4);

  \draw[edge] (v4) to [bend right=10] node[above] {$e_{-i_1}$} (v5);
  \draw[edge] (v5) to [bend right=15] node[below] {$e_{i_1}$} (v4);
\end{tikzpicture}\ .
\]
\end{lemma}

\begin{proof}
Let $P_v=\max(P_{a'+s},P_{a'-s})$, where $s=s_{i_1}+s_{i_2}+\cdots+s_{i_\alpha}$. We claim that the edges $e_{i_1}$ and $e_{-i_1}$ appear in $G$, incident on $P_v$, as in the picture above. That is, if $a'\in[-\lceil \frac{n-1}{2}\rceil\interval 0]$, then $v=a'+s$ and
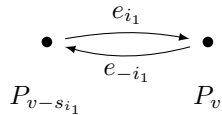
\[
\begin{tikzpicture}[baseline,vertex/.style={anchor=base, circle, fill,minimum size=4pt, inner sep=0pt, outer sep=0pt},auto,
                               edge/.style={->,>=latex, shorten > = 5pt,shorten < = 5pt},
position/.style args={#1:#2 from #3}{
    at=(#3.#1), anchor=#1+180, shift=(#1:#2)
}]
  \node[vertex] (v1) {};
  \node[vertex,position = 0:{\nodeDist} from v1] (v2) {};

  \draw (v1) node[below = 13] {$P_{v-s_{i_1}}$};
  \draw (v2) node[below = 13] {$P_{v}$};

  \draw[edge] (v1) to [bend left=10] node[above] {$e_{i_1}$} (v2);
  \draw[edge] (v2) to [bend left=15] node[below] {$e_{-i_1}$} (v1);
\end{tikzpicture}
\]
appears in $G$, and if $a'\in[1 \interval \lfloor \frac{n-1}2\rfloor]$, then $v=a'-s$ and
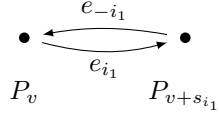
\[
\begin{tikzpicture}[baseline,vertex/.style={anchor=base, circle, fill,minimum size=4pt, inner sep=0pt, outer sep=0pt},auto,
                               edge/.style={->,>=latex, shorten > = 5pt,shorten < = 5pt},
position/.style args={#1:#2 from #3}{
    at=(#3.#1), anchor=#1+180, shift=(#1:#2)
}]
  \node[vertex] (v1) {};
  \node[vertex,position = 0:{\nodeDist} from v1] (v2) {};

  \draw (v1) node[below = 13] {$P_{v}$};
  \draw (v2) node[below = 13] {$P_{v+s_{i_1}}$};

  \draw[edge] (v1) to [bend right=15] node[below] {$e_{i_1}$} (v2);
  \draw[edge] (v2) to [bend right=10] node[above] {$e_{-i_1}$} (v1);
\end{tikzpicture}
\]
appears in $G$.

Let $b\in \Z/n\Z$ be such that $G_b=G\ \CMamalg\ \underset {P_b} \bullet \underset {e_{k+1}} \rightarrow \underset {P_{b+j}} \bullet$ has an Eulerian path from $P_a$. Such a $b$ exists, by the definition of $U(a,j,I)$.
Our order on edges is determined by our order on vertices. The largest edge which appears in some graph from $U(a,j,I)$ is incident on the largest vertex reachable from $P_a$ using edges of the form $e_i$ for $i\in I$ and $e_{k+1}$. We will show that $P_v$ is the largest such vertex.
The vertices reachable from $P_a$ in $G_b$ using edges of the form $e_i$ for $i\in I$ and $e_{k+1}$ are necessarily of the form $P_\lambda$ for some 
\[ \lambda\in[a-s_{i_1}-\dots-s_{i_\alpha}\interval a+s_{i_1}+\dots+s_{i_\alpha}]\cup[a+j-s_{i_1}-\dots-s_{i_\alpha}\interval a+j+s_{i_1}+\dots+s_{i_\alpha}]. \]

By our assumption of $s\leqslant |a|, |a+j|$, where $s=s_{i_1}+s_{i_2}+\cdots s_{i_\alpha}$, the intervals $[a-s\interval a+s]$ and $[a+j-s\interval a+j+s]$ are of the form

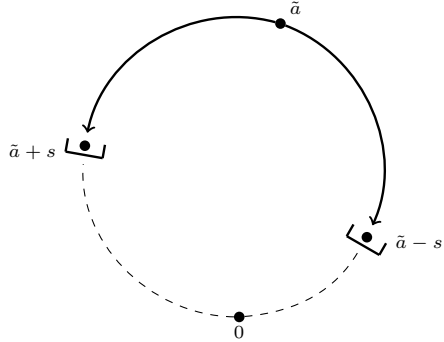
\newcommand{\bracksize}{0.1}
\[\begin{tikzpicture}[baseline,vertex/.style={anchor=base, circle, fill,minimum size=4pt, inner sep=0pt, outer sep=0pt},auto,
                               edge/.style={->,>=latex, shorten > = 5pt,shorten < = 5pt},
position/.style args={#1:#2 from #3}{
    at=(#3.#1), anchor=#1+180, shift=(#1-90:#2)
}]

  \coordinate (vc) {};
  \node[vertex,position = 0:{\nodeDist} from vc] (v0) {};
  \node[vertex,position = -100:{\nodeDist} from vc] (v1L) {};
  \node[vertex,position = 60:{\nodeDist} from vc] (v1R) {};
  \node[vertex,position = 160:{\nodeDist} from vc] (v2R) {};

  \node[coordinate,position=-40:{1*\bracksize} from v1L] (LbrackC) {};
  \node[coordinate,position=-100:{2.5*\bracksize} from LbrackC] (LbrackL1) {};
  \node[coordinate,position=80:{2.5*\bracksize} from LbrackC] (LbrackR1) {};
  \node[coordinate,position=170:{1.8*\bracksize} from LbrackL1] (LbrackL2) {};
  \node[coordinate,position=170:{1.8*\bracksize} from LbrackR1] (LbrackR2) {};

  \node[coordinate,position=-60:{1*\bracksize} from v1R] (RbrackC) {};
  \node[coordinate,position=-120:{2.5*\bracksize} from RbrackC] (RbrackL1) {};
  \node[coordinate,position=60:{2.5*\bracksize} from RbrackC] (RbrackR1) {};
  \node[coordinate,position=150:{1.8*\bracksize} from RbrackL1] (RbrackL2) {};
  \node[coordinate,position=150:{1.8*\bracksize} from RbrackR1] (RbrackR2) {};

  \draw (v0) node[below] {\scriptsize ${0}$};
  \draw (LbrackL1) node[left] {\scriptsize ${\tilde a+s}$};
  \draw (RbrackR2) node[right] {\scriptsize ${\tilde a-s}$};
  \draw (v2R) node[above right] {\scriptsize ${\tilde a}$};

  \draw [dashed, shorten < = 5] (v1R) to[bend left = 27] (v0);
  \draw [dashed, shorten > = 5] (v0) to[bend left = 47] (v1L);
  \draw [<-, shorten < = 3, line width = 0.9] (v1L) to[bend left = 47] (v2R);
  \draw [->, shorten > = 3, line width = 0.9] (v2R) to[bend left = 47] (v1R);

  \draw [line width = 0.9] (LbrackL2) to (LbrackL1);
  \draw [line width = 0.9] (LbrackL1) to (LbrackR1);
  \draw [line width = 0.9] (LbrackR1) to (LbrackR2);

  \draw [line width = 0.9] (RbrackL2) to (RbrackL1);
  \draw [line width = 0.9] (RbrackL1) to (RbrackR1);
  \draw [line width = 0.9] (RbrackR1) to (RbrackR2);
\end{tikzpicture}\]
for $\tilde a=a$ or $a+j$. In particular, $0$ does not lie in the interior of either interval. The maximal vertex of this form is $P_{v_0}$, where
\[|v_0|= \min (|a|-s_{i_1}-s_{i_2}-\dots-s_{i_\alpha},|a+j|-s_{i_1}-s_{i_2}-\dots-s_{i_\alpha})\ .\]
For $P_{v_0}$ to be reachable from $P_a$ in $G_b$, $v_0-a$ must be expressible as a partial sum of the integers $\pm s_{i_1},\pm s_{i_2},\dots, \pm s_{i_\alpha}$ and $j$. The only way this can happen is if
\[v_0=a\pm (s_{i_1}+s_{i_2}+\dots+s_{i_\alpha}) \text{, or } v_0=a+j\pm (s_{i_1}+s_{i_2}+\dots+s_{i_\alpha})\ .\]
If $a'\in [-\lceil\frac{n-1}2\rceil\interval 0]$, then $v_0=a'+s_{i_1}+s_{i_2}+\dots+s_{i_\alpha}$. Otherwise $v_0=a'-s_{i_1}-s_{i_2}-\dots -s_{i_\alpha}$. That is, we have $v=v_0$, and $P_v$ is the largest vertex which may appear in a graph from $U(a,j,I)$ with nonzero degree.

As $G$ is maximal, $v_0\neq a,a+j$, and there exists an Eulerian path from $P_a$ to $P_{a+j}$ on $G_b$, there must be at least $2$ edges incident on $P_{v_0}$ in $G_b$. The maximal edges which can be incident on $P_{v_0}$ are $e_{\pm i_1}$. Hence, in a maximal graph these edges must be incident on $P_{v_0}$.
This proves the claim.

We will now complete the proof of Lemma~\ref{max_connecting_path} by induction on $\alpha$.
Let us assume $a'\in[-\lceil\frac{n-1}2\rceil\interval 0]$; when $a'\in[1\interval \lfloor\frac{n-1}2\rfloor]$ the proof is symmetric. 
In this case, the maximal graph $G\in U(a,j,I)$ must include the edges
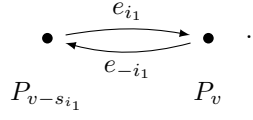
\[
\begin{tikzpicture}[baseline,vertex/.style={anchor=base, circle, fill,minimum size=4pt, inner sep=0pt, outer sep=0pt},auto,
                               edge/.style={->,>=latex, shorten > = 5pt,shorten < = 5pt},
position/.style args={#1:#2 from #3}{
    at=(#3.#1), anchor=#1+180, shift=(#1:#2)
}]
  \node[vertex] (v1) {};
  \node[vertex,position = 0:{\nodeDist} from v1] (v2) {};

  \draw (v1) node[below = 13] {$P_{v-s_{i_1}}$};
  \draw (v2) node[below = 13] {$P_{v}$};

  \draw[edge] (v1) to [bend left=10] node[above] {$e_{i_1}$} (v2);
  \draw[edge] (v2) to [bend left=15] node[below] {$e_{-i_1}$} (v1);
\end{tikzpicture}\enskip .
\]
By induction, the maximal graph in $U(a,j,I\setminus\{i_1,-i_1\})$ is
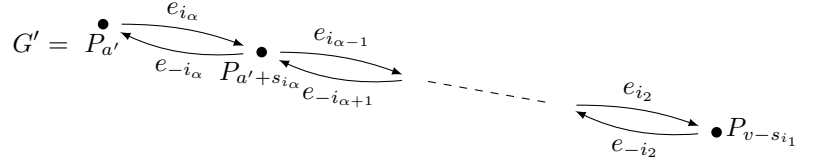
\[
G'=\begin{tikzpicture}[baseline,vertex/.style={anchor=base,circle, fill,minimum size=4pt, inner sep=0pt, outer sep=0pt},auto,
                               edge/.style={->,>=latex, shorten > = 5pt,shorten < = 5pt},
position/.style args={#1:#2 from #3}{
    at=(#3.#1), anchor=#1+180, shift=(#1:#2)
}]

  \node[vertex] (v2) {};
  \node[vertex,position = 170:{\nodeDist} from v1] (v1) {};
  \coordinate[position = -10:{\nodeDist} from v2] (v3);
  \coordinate[position = -10:{\nodeDist} from v3] (v4) {};
  \node[vertex,position = -10:{\nodeDist} from v4] (v5) {};

  \draw (v1) node[below] {$P_{a'}$};
  \draw (v2) node[below] {$P_{a'+s_{i_{\alpha}}}$};
  \draw (v5) node[right] {$P_{v-s_{i_1}}$};

  \draw[edge] (v1) to [bend left=10] node[above] {$e_{i_\alpha}$} (v2);
  \draw[edge] (v2) to [bend left=15] node[below] {$e_{-i_\alpha}$} (v1);

  \draw[edge] (v2) to [bend left=10] node[above] {$e_{i_{\alpha-1}}$} (v3);
  \draw[edge] (v3) to [bend left=15] node[below] {$e_{-i_{\alpha+1}}$} (v2);

  \draw[dashed,shorten > = 5pt, shorten < = 5pt] (v3) -- (v4);

  \draw[edge] (v4) to [bend left=10] node[above] {$e_{i_2}$} (v5);
  \draw[edge] (v5) to [bend left=15] node[below] {$e_{-i_2}$} (v4);
\end{tikzpicture}\ .
\]
We see that $G'\ \CMamalg\ \underset {P_{v-s_{i_1}}} \bullet \underset {e_{-i_1}} {\overset {e_{i_1}} \rightleftarrows} \underset {P_v} \bullet$ lies in $U(a,j,I)$, and as no graph from $U(a,j,I)$ can be larger than this graph, it must be maximal.

\end{proof}

\section{Conclusion of the proof of Proposition \ref{max_graph_decomp}}\label{sect.maximal_2}

Let $G=G'\ \CMamalg\ H_t$ be the maximal graph from $U(a,j)$. By Lemma~\ref{max_subgraphs} no $\tilde G\in U(a,j)$ can contain $H_{t'}$ for $t' > t$. 
As $G\in U(a,j)$, there exists $b\in \Z/n\Z$ such that 
\[ G_b=G\ \CMamalg\ \underset{P_b}\bullet\underset{e_{k+1}}\rightarrow\underset{P_{b+j}}\bullet\]
admits an Eulerian path from $P_a$ to $P_{a+j}$. For $t\in\{0,1,\dots,r\}$ define
\[I_t^+=\{t+1,t+2,\dots, r\}\ ,\]
\[I_t^-=\{-t-1,-t-2,\dots,-r\}\ ,\]
\[I_t=I_t^+\cup I_t^- .\]
Furthermore, let
\[R_t(a)=\{a+\overline p\ |\ \text{ $\overline p$ is a partial sum of $s_i,i\in I_t$ and $j$}\}\]
be the set of indices $v$ such that $P_v$ is reachable from $P_a$ using only edges from $G'$, i.e.,
a subset of $\{e_i\ |\ i\in I_t\} \cup \{e_{k+1}\}$. Recall that 
\begin{align*}
(s_1,s_2,s_3,\dots,s_k)&=(s_1,s_2,s_3,\dots,s_r,s_{-1},s_{-2},\dots,s_{-r})\\
&=\left(\lceil \frac {1} {2} \rceil,\lceil \frac {2} {2} \rceil,\lceil \frac {3} {2} \rceil,\dots,\lceil \frac {r} {2} \rceil,-\lceil \frac {1} {2} \rceil,-\lceil \frac {2} {2} \rceil,\dots,-\lceil \frac {r} {2} \rceil\right)\\
&=(1,1,2,\dots,\lceil\frac r 2 \rceil,-1,-1,\dots,-\lceil \frac r 2 \rceil).
\end{align*}
If we set $s_{k+1} = j$, then 
$R_t(a)=\{ a + \sum_{\lambda\in Q} s_\lambda \; | \; \text{$Q$ is a subset of $I_t\cup\{k+1\}$}\}$.

\begin{lemma}\label{reachable_H}
Assume $n>r$. There exists $G\in U(a,j)$ containing $H_t$ (for $t>0$) if and only if $R_t(a)\cap \supp(H_t)\neq\emptyset$.
\end{lemma}

Here
\[
\supp(H_t)=\begin{cases}
\{-\lceil \frac t 2 \rceil,-\lceil \frac t 2 \rceil+1,\dots,\lfloor \frac t 2 \rfloor\}\ =
[-\lceil \frac t 2 \rceil\interval \lfloor \frac t 2 \rfloor]\ , &\text{if $t>0$}\\
\emptyset\ , &\text{if $t=0$.}
\end{cases}
\]
denotes the support of $H_t$, i.e., the set of subscripts $i$ such that $P_i$ is adjacent to at least one edge from $H_t$. 
For example, $\supp(H_0)=\emptyset$,
$\supp(H_1)=\{-1,0\}$,
$\supp(H_2)=\{-1,0,1\}$,
$\supp(H_3)=\{-2,-1,0,1\}$, etc. Note also that $I_t = \{\pm 1,\pm 2,\dots,\pm r\} \setminus \supp(H_t)$.

\begin{proof}
Suppose that there exists $G\in U(a,j)$ containing $H_t$ with $t>0$. Then there is an Eulerian path $(e_{i_1},e_{i_2},\dots,e_{i_{k+1}})$ from $P_a$ in $G_b$ for some $b\in \Z/n\Z$. As this path traverses $G_b$, it must traverse $H_t$. Let $P_v$ be the first vertex on this path with $v\in\supp(H_t)$ and $m$ be the first integer with $\tar(e_{i_m})=P_v$ (with $m=0$ if $P_a=P_v$). Then, setting $s_{k+1}=j$, we have $v=a+\sum_{h=1}^m s_{i_h}$. Each $e_{i_h}$ with $1\leqslant h\leqslant m$ has $\src(e_{i_h})\notin H_t$ and so $i_h\in I_t\cup\{k+1\}$ by the definition of $H_t$. Thus $v\in R_t(a)$ as required.

Conversely, suppose that $R_t(a)\cap \supp(H_t)\neq \emptyset$ and let $I'$ be a minimal subset of $I_t\cup\{k+1\}$ such that $v=a+\sum_{i\in I'}s_i$ lies in $\supp(H_t)$. Note that at most one of $i$ and $-i$ can lie in $I'\setminus\{k+1\}$ for any $i$ by the minimality of $I'$ (otherwise we can remove both from $I'$). Now consider two cases.

\smallskip
Case 1. $k + 1 \not \in I'$. We claim that there exists a graph $G$ with edges $e_{\pm 1}, \ldots, e_{\pm r}$ containing $H_t$ such that
$G$ has no repeated edges and $G$ has an Eulerian path starting at $P_a$. If we can prove this claim, then appending $e_{k+1}$ to 
this Eulerian path at the end, we obtain an Eulerian path for $G_{a}$ starting at $P_a$. Thus $G \in U(a, j)$, and the proof in Case 1 will be complete.

To prove the claim, 
set $G'=H_t\ \CMamalg\ G''$ where $G''$ is a path from $P_a$ to $P_v$ with edge labels $e_{i}$, $i \in I'$ 
and back to $P_{a}$ with edge labels $e_{-i}$, $i \in I'$.
Since $H_t$ is an Euler circuit, $G'$ admits an Eulerian path. Note that
by our construction, edges in $G'$ come in pairs $e_{\ell}$ and $e_{-\ell}$ so that $\src(e_{\ell}) = \tar(e_{-\ell})$ for each $\ell$.
If $G'$ has edges for all the labels $e_{\pm 1},\dots,e_{\pm r}$, then we can set $G = G'$ and our proof is complete. 
If not, then we can construct a $G$ from $G'$ recursively, by attaching missing edges in pairs, $e_{\ell}$ and $e_{-\ell}$, as follows
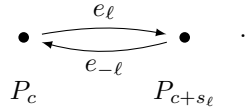
\[
\begin{tikzpicture}[baseline,vertex/.style={anchor=base, circle, fill,minimum size=4pt, inner sep=0pt, outer sep=0pt},auto,
                               edge/.style={->,>=latex, shorten > = 5pt,shorten < = 5pt},
position/.style args={#1:#2 from #3}{
    at=(#3.#1), anchor=#1+180, shift=(#1:#2)
}]
  \node[vertex] (v1) {};
  \node[vertex,position = 0:{\nodeDist} from v1] (v2) {};

  \draw (v1) node[below = 13] {$P_{c}$};
  \draw (v2) node[below = 13] {$P_{c + s_{\ell}}$};

  \draw[edge] (v1) to [bend left=10] node[above] {$e_{\ell}$} (v2);
  \draw[edge] (v2) to [bend left=15] node[below] {$e_{-\ell}$} (v1);
\end{tikzpicture} \; \; .
\]
Here $c \in \supp(G')$. Note that for a given $d\in \Z/n\Z$, there are at most two $\ell\in\{\pm 1,\dots,\pm r\}$ such that $s_\ell=d$. 
If we want to add $e_\ell$ and $e_{-\ell}$ to $G'$, and $e_{\ell'}$ and $e_{-\ell'}$ with $s_{\ell} = s_{\ell'}$ are not present 
in $G'$, then we can place $e_\ell$ and $e_{-\ell}$ at any $P_c$, where $c \in \supp(G')$, as above.
This way the extended graph will have no repeated edges and will retain an Eulerian circuit.
If $e_{\ell'}$ and $e_{-\ell'}$ with $s_{\ell'} = s_{\ell}$ are present in $G'$, then the same will be true if we choose
$c \in \supp(G')$ so that $P_c \neq \src(e_{\ell'})$. This completes the proof of the claim and thus of Lemma~\ref{reachable_H} in Case 1.

\smallskip
Case 2. $k + 1 \in I'$. Applying the claim in Case 1 with $a$ replaced by $a + j$, we see that
there exists a graph $G$ with edges $e_{\pm 1}, \ldots, e_{\pm r}$ containing $H_t$ such that
$G$ has no repeated edges and $G$ has an Eulerian path starting at $P_{a + j}$.  Appending $e_{k+1}$ to 
this Eulerian path at the beginning, we obtain an Eulerian path for $G_{a}$ starting at $P_a$. 
Thus $G \in U(a, j)$ and $H_t$ is contained in $G$. This completes the proof in Case 2. 
\end{proof}

\begin{lemma}\label{max_t}
Let $0\leqslant t\leqslant r$ be the largest integer such that there exists $G\in U(a,j)$ containing $H_t$. Set $s=\sum_{i\in I_t^+}s_i$. Then

\smallskip
(a) $s\leqslant |a|, |a+j|$.

\smallskip
(b) If $t>0$ and $v\in R_t(a)\cap \supp(H_t)$ then either every $s_i$ for $i \in I_t^+$ or every $s_i$ for $i \in I_t^-$  
must appear as a summand of $v$.
\end{lemma}


\begin{proof} Throughout the proof, $a'$ will denote either $a$ or $a + j$.

\smallskip
(a) We will assume that $0  < a' \leqslant \lfloor \frac{n-1}{2} \rfloor$; the case where  
$- \lceil \frac{n-1}{2} \rceil \leqslant a' \leqslant 0$ is symmetrical. We argue by contradiction. Suppose $s_r + \ldots + s_{t + 1} > a'$.
Let $t + 1 \leqslant \lambda \leqslant r$ be the smallest integer such that 
\begin{equation} \label{e.reachable}
s_r + s_{r-1} + \ldots + s_{\lambda+ 1} + s_{\lambda} > a'.
\end{equation}
We may assume without loss of generality that $\lambda \leqslant r-1$. 
Indeed, suppose $\lambda = r$. Then $0 \leqslant a' < s_r = \lceil \frac{r}{2} \rceil$ and consequently, $a' \in \supp(H_r)$. In this case
$t = r$, the sum $s_r + \ldots + s_{t + 1}$ is empty, and there is nothing to prove. 

   From now on we will assume that $\lambda \leqslant r -1$. By our choice of $\lambda$, $a' \geqslant s_r + \ldots + s_{\lambda + 1}$ and thus
\[ 0\leqslant a' - (s_r + s_{r-1} +  \ldots + s_{\lambda + 1}) = a' - (s_r + s_{r-1} + \ldots + s_{\lambda + 1} + s_{\lambda}) + s_{\lambda} < 
0 + s_{\lambda} = \lceil \frac{\lambda}{2} \rceil. \]
This shows that $a' - (s_r + \ldots + s_{\lambda + 1}) \in R_{\lambda}(a)  \cap \supp( H_{\lambda})$.  
Lemma~\ref{reachable_H} now tells us that there exists a $G \in U(a, j)$ containing $H_{\lambda}$.
Since $\lambda \geqslant t + 1$, this contradicts our choice of $t$.

\smallskip
(b) Suppose $v = a' + s'$ lies in $\supp(H_t)$ for some $s' = \epsilon_r s_r + \epsilon_{r-1} s_{r-1} + \ldots + \epsilon_{t + 1} s_{t + 1}$,
where $\epsilon_i  \in \{ - 1, 0, 1 \}$. We want to show that either all $\epsilon_i$ are $1$, or all $\epsilon_i$ are $-1$, for $i=t+1,\dots,r$.
First note that $\epsilon_{t + 1} \neq 0$. Indeed, otherwise we would have
\[ v \in R_{t + 1}(a) \cap \supp(H_t) \subset R_{t+1}(a) \cap \supp(H_{t + 1}). \]
By Lemma~\ref{reachable_H}, this contradicts the maximality of $t$.

It remains to show that if $\epsilon_{t + 1} \neq 0$ and $\epsilon_i\neq \epsilon_{t+1}$ for some $t+1<i\leqslant r$, then
\begin{equation} \label{e.ceil} |s'| < s - \lceil \dfrac{t}{2} \rceil. \end{equation}  
If we can prove this inequality, then $v$ cannot lie in $\supp(H_t)$ because
\[ |a' + s'| \geqslant |a'| - |s'| \geqslant |s| - |s'| > \lceil \dfrac{t}{2} \rceil, \]
and part (b) will follows.
We we will prove the inequality~\eqref{e.ceil} in two steps. 

Step 1. First we will show that~\eqref{e.ceil} holds if $\epsilon_{\lambda} = 0$ for any $\lambda = t+2, \ldots, r$.
(Recall that we know that $\epsilon_{t +1} \neq 0$.) Indeed,
\[  |s'| \leqslant \sum_{i \neq \lambda} |\epsilon_i| s_i   
 \,   \leqslant \, s - s_{\lambda}  \,  = \,  s -  \lceil \frac{\lambda}{2} \rceil < s - \lceil \frac{t}{2} \rceil, \]
where the last inequality follows from $\lambda \geqslant t + 2$.

Step 2. We are now ready to complete the proof of~\eqref{e.ceil}. By Step 1 we
may assume that $\epsilon_{\lambda} = \pm 1$ for every $\lambda = t + 1, t + 2 \ldots, r$. In this case
\[ |s'| = \sum_{\lambda \in A} s_{\lambda} - \sum_{\mu \not \in A} s_{\mu} \]
 for some proper subset $\emptyset \neq A \subsetneq \{ t+1, t+2, \ldots, r \}$. Here $A = \{ \lambda \, | \, s_{\lambda} = 1 \}$
or $A = \{ \lambda \, | \, s_{\lambda} = -1 \}$. Note that $A$ is a proper subset of $\{ t + 1, t + 2, \ldots, r \}$
because we are assuming that $\epsilon_i\neq \epsilon_{t+1}$ for some $t+1<i\leqslant r$. 
Thus $\displaystyle |s'| < \sum_{\lambda \in A} s_{\lambda}$. Now for any $\mu \in \{ t+1, t + 2, \ldots, r \} \setminus A$, we have
\[ |s'| < \sum_{\lambda \in A} s_{\lambda} \leqslant s - s_{\mu} = 
s - \lceil \frac{\mu}{2} \rceil \leqslant s - \lceil \frac{t}{2} \rceil . \]
This completes the proof of~\eqref{e.ceil} and thus of part (b).
\end{proof}

\begin{lemma}\label{subgraph_connected}
Let $0\leqslant t\leqslant r$ be the largest integer such that there exists a graph in $U(a,j)$ containing $H_t$.
Denote the maximal graph in $(U(a, j)$ by $G$. (Recall that by Lemma~\ref{max_subgraphs}, $G$ contains $H_t$.) 
Let $G'$ be the graph with vertices $P_v$, $v \in \mathbb Z / n \mathbb Z$ 
obtained from $G$ by removing the edges $e_{\pm 1}, \ldots, e_{\pm t}$, i.e., by removing the edges which occur in $H_t$.
Then $G'\in U(a,j,I_t)$.
\end{lemma}

\begin{proof}
If $t=0$ then $G=G'\in U(a,j)=U(a,j,I_0)$. So we may assume $t>0$.
As $G\in U(a,j)$ there exists $b\in \Z/n\Z$ such that 
\[ G_b={G\ \CMamalg\ \underset{P_b}\bullet\underset{e_{k+1}}\rightarrow\underset {P_{b+j}}\bullet} \]
admits an Eulerian path from $P_a$ to $P_{a+j}$. Let $G_b'$ be the graph obtained from $G'$ by adding the edge $e_{k+1}$ originating at $P_b$.
Equivalently, $G'_b$ is the graph obtained by removing $e_{\pm 1}, \ldots, e_{\pm t}$ from $G_b$.
It suffices to show that there exists an Eulerian path from $P_a$ to $P_{a+j}$ in $G'_b$.

As there exists an Eulerian path on $G_b$ from $P_a$ to $P_{a+j}$, by Theorem \ref{eulerian_unic}(a), the degree of the vertices in $G_b$ must satisfy
\begin{equation}\label{degrees}
\begin{split}
\outdeg(P_v)&=\indeg(P_v)\ , \text{ for all $v\neq a,a+j$}\ ,\\
\outdeg(P_a)&=\indeg(P_a)+1\ , \text{ if $j\neq 0$}\ ,\\
\outdeg(P_{a+j})&=\indeg(P_{a+j})-1\ , \text{ if $j\neq 0$}\ ,\\
\outdeg(P_a)&=\indeg(P_a)\ , \text{ if $j=0$}\ .
\end{split}
\end{equation}
As every vertex has its indegree equal to its outdegree in $H_t$, the equation (\ref{degrees}) also holds in $G'_b$. We claim that $G'_b$ 
is connected. If we can prove this claim, then Theorem \ref{eulerian_unic}(a) will tell us that $G_b'$ admits an Eulerian path from $P_a$ 
to $P_{a+j}$ and consequently, $G'\in U(a,j,I_t)$, as desired.

To prove the claim, we will argue by contradiction. Assume $G_b'$ is not connected. Then there exists a decomposition of $G_b'$ as $G_b'=\Gamma^1\ \CMamalg\ \Gamma^2$ for two non-empty graphs $\Gamma^1$ and $\Gamma^2$, such that
$P_a$ has nonzero degree in $\Gamma^1$, $\Gamma^1$ is connected, and there is no vertex having nonzero degree in both $\Gamma^1$ and $\Gamma^2$. Thus for any vertex $P_v$, either all edges incident on $P_v$ in $G_b'$ are present in $\Gamma^1$ or none of them are (and similarly for $\Gamma^2$).

The sum of the indegrees of the vertices in $\Gamma^1$ equals the sum of the outdegrees. By (\ref{degrees}) we must then have $P_{a+j}$ lies in $\Gamma^1$ and so every vertex having nonzero degree in $\Gamma^2$ has its indegree equal to its outdegree. There is a path from $P_{a}$ to some $P_v\in H_t$ in $G_b$. By Lemma \ref{max_t} (b), this path necessarily uses all edges of the form $e_\lambda$ for $\lambda\in I_t^+$, or all edges of the form $e_\lambda$ for $\lambda\in I_t^-$. Thus $\Gamma^1$ contains either every edge $e_\lambda$ for $\lambda\in I_t^+$, or every edge $e_\lambda$ for $\lambda\in I_t^-$. We assume that $\Gamma^1$ contains every edge $e_\lambda$ for $\lambda\in I_t^+$, the proof in the other case is symmetric.

The set of edges of $\Gamma^2$ is a subset of $\{e_{-t-1},e_{-t-2},\dots,e_{-r}\}\cup\{e_{k+1}\}$. Let $P_{v'}$ be a vertex having nonzero degree in $\Gamma^2$. As every vertex has its indegree equal to its outdegree in $\Gamma^2$, there exists a closed circuit in $\Gamma^2$ originating (and terminating) at
$P_{v'}$. This implies that a nonempty partial sum of $\{s_{-t-1},s_{-t-2},\dots,s_{-r}\}\cup \{j\}$ is zero modulo $n$. By Lemma~\ref{max_t}(a), $|s_{-t-1}+s_{-t-2}+\dots+s_{-r}|\leqslant |a|<n$. Thus our partial sum must be $0=j-(s_{\lambda_1}+s_{\lambda_2}+\dots+s_{\lambda_\alpha})$ for some $\lambda_1,\lambda_2,\dots,\lambda_{\alpha}\in\{t+1,t+2,\dots,r\}$.

As $j$ appears in this partial sum corresponding to a closed circuit in $\Gamma^2$ we know that $e_{k+1}$ is an edge in $\Gamma^2$ and therefore not in $\Gamma^1$. Thus the path from $P_a$ to $P_v$ in $\Gamma^1$ does not use the edge $e_{k+1}$ and hence by Lemma~\ref{max_t}(b) we have $v=a+(s_{t+1}+s_{t+2}+\dots+s_r)$. Using the equality $0=j-(s_{\lambda_1}+s_{\lambda_2}+\dots+s_{\lambda_\alpha})$ we can write $v$ as
\begin{align*}
v&=a+(s_{t+1}+s_{t+2}+\dots+s_r)+j-(s_{\lambda_1}+s_{\lambda_2}+\dots+s_{\lambda_\alpha})\\
&=a+j+\sum_{i\in I_t^+\setminus\{\lambda_1,\dots,\lambda_\alpha\}}s_i\enskip .
\end{align*}
This contradicts Lemma \ref{max_t} (b), as $v\in R_t(a)\cap\supp(H_t)$ and we have written $v$ as a sum that does not include $s_i$ for every $i\in I_t^+$. 
The proof of the claim (and thus of Lemma~\ref{subgraph_connected}) is now complete.
\end{proof}

We are now ready to finish the proof of Proposition \ref{max_graph_decomp}.

\begin{proof}[Conclusion of the proof of Proposition~\ref{max_graph_decomp}]
Let $t$ be the largest integer such that there exists a graph $G\in U(a,j)$ containing $H_t$, and $G$ be the maximal graph in $U(a,j)$. By Lemma \ref{max_subgraphs}, $G$ must contain $H_t$. Write $G=G'  \CMamalg\ H_t$. By Lemma \ref{subgraph_connected}, $G'\in U(a,j,I_t)$. 
By Lemma~\ref{max_connecting_path} (whose assumptions are met by Lemma \ref{max_t} (a)), the largest graph in $U(a,j,I_t)$ is 
the conjectured subgraph $\hat G'$ in the proposition. We want to show that $G'=\hat G'$. 
As $\hat G'$ is maximal in $U(a,j,I_t)$, it suffices to show that $\hat G=\hat G'\ \CMamalg\ H_t$ lies in $U(a,j)$.

Clearly $\hat G_a$ has an Eulerian circuit. Thus we only need to check that $\hat G$ has no repeated edges. 
As $s_{t+1}+s_{t+2}+\cdots +s_r\le |a'|\le \lceil \frac{n-1}{2} \rceil$, $n$ is sufficiently large so that the vertices of the graph $\hat G$ 
are all distinct. Thus the edges of $\hat G'$ are all distinct, and the edges of $H_t$ are all distinct. 
If a repeated edge were present, we would have more than one vertex having nonzero degree in both subgraphs $\hat G'$ and $H_t$. 
On the other hand, $\hat G'$ and $H_t$ have exactly one vertex in common, namely $P_{a' + s_r + \ldots + s_{t+1}}$ in part (a) and
$P_{a' - s_r + \ldots - s_{t+1}}$ in part (b). Any other vertex is reachable from $a'$ using the edges from $I_{t + 1}$ and hence,
cannot lie in $H_t$ (or even in $H_{t+1}$ by the minimality of $t$; see Lemma~\ref{reachable_H}). This shows that
the edges of $\hat G$ are distinct and hence, $\hat G\in U(a,j)$, completing the proof.
\end{proof}

\section{Proof of Proposition~\ref{Ic_L_j_ker}(a)}
\label{section_j_nonzero}

In this section we will prove the following.

\begin{proposition}\label{Ic_L_j_ker_j_neq_0} 
Let $1 \leqslant r < n$, $a, j \in \mathbb Z / n \mathbb Z$ and $j \neq 0$. Assume that $F$ is an infinite field of characteristic not dividing
$r!$. If both $a$ and $a+j$ lie in $\supp(H_r) = [-\lceil r/2 \rceil\interval \lfloor r/2 \rfloor]$, assume further that $a, a+j \neq 0$. Then
\[ \text{$(\Ic(L_j))_{a,b}= 0$ for any $b\neq a$ and
$(\Ic(L_j))_{a,a}\neq 0$} \]
in $F$.
\end{proposition}

Note that Proposition \ref{Ic_L_j_ker}(a) readily follows from Proposition~\ref{Ic_L_j_ker_j_neq_0}.
Indeed, fix $0 \neq j \in \Z/ n \Z$. By the definition of $\delta_j$ (see the statement of Proposition~\ref{Ic_L_j_ker}), 
there are exactly $\delta_j$ values of $a \in \Z/ n \Z$ such that $(a, j)$ does not satisfy the conditions of Proposition~\ref{Ic_L_j_ker_j_neq_0}.
(Note that this can only happen if $a \in \{ 0, -j \}$, so in each case $\delta_j = 0$, $1$ or $2$.) If we remove the $a^{\rm th}$ 
row and the $a^{\rm th}$ column from $\Ic(L_j)$ for every such $a$, we will be left with a diagonal matrix with non-zero 
diagonal entries. In other words, if we remove $\delta_j$ rows and $\delta_j$ columns from $\In(L_j)$, the remaining
$(n- \delta_j) \times (n - \delta_j)$ matrix is non-singular. This shows that $\nullity(\In(L_j)) \leqslant \delta_j$, as desired.

\smallskip
The remainder of this section will be devoted to proving Proposition~\ref{Ic_L_j_ker_j_neq_0}.

\begin{lemma}\label{only_diagonal}
Let $j\neq 0$ and $G$ be the maximal graph in $U(a,j)$. Suppose that there exists an Eulerian path from $P_a$ on $G_b$. Then $b=a$.
\end{lemma}

\begin{proof} The maximal graph $G$ in $U(a, j)$ is described in Proposition \ref{max_graph_decomp}. 
It is clear from this description that the indegree matches the outdegree for each vertex of $G$. Thus in the graph
\[G_b=G\ \CMcoprod\ \underset{P_b}\bullet \underset{e_{k+1}}\rightarrow \underset{P_{b+j}}\bullet\]
the outdegree of $P_b$ is $1$ larger than its indegree and the outdegree of $P_{b+j}$ is $1$ smaller than its indegree. 
Since $j \neq 0$ in $\mathbb Z/ n \mathbb Z$, we have $P_b\neq P_{b+j}$. Thus every Eulerian path on $G_b$ starts at $P_b$ and ends at $P_{b+j}$;
see Theorem~\ref{eulerian_unic}.
Since we assumed there exists an Eulerian path from $P_a$ on $G_b$ we conclude that $b=a$.
\end{proof}

\begin{lemma}\label{nonzero_diagonal} Let $a, j \in \mathbb Z / n \mathbb Z$ be as in the statement of Proposition~\ref{Ic_L_j_ker_j_neq_0}
and $G$ be the maximal graph in $U(a, j)$. Then the sum of the signatures of Eulerian paths on $G_a$ from $P_a$ to $P_{a+j}$ 
is $\pm\alpha!$ for some $\alpha\leqslant r$.
\end{lemma}

\begin{proof} Let $0 \leqslant t \leqslant r$ be the largest integer such that a graph from $U(a, j)$ 
contains $H_t$. Recall that $H_0 = \emptyset$, so $t$ is well defined. We will consider three cases.

\smallskip
Case 1. $P_{a+j}\succ P_{a}$ and $a$ does not lie in $\supp(H_t)$.  (Note that if $a$ lies in $\supp(H_t)$, then so does $a + j$;
hence $r = t$ and we are in Case 3 below.)
By Proposition \ref{max_graph_decomp}, $G_a=H_t\ \CMamalg\ G_a'$, for $G_a'$ of the form
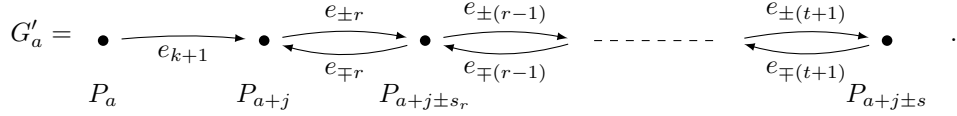
\renewcommand{\nodeDist}{2}
\[G_a'=
\begin{tikzpicture}[baseline,vertex/.style={anchor=base, circle, fill,minimum size=4pt, inner sep=0pt, outer sep=0pt},auto,
                               edge/.style={->,>=latex, shorten > = 5pt,shorten < = 5pt},
position/.style args={#1:#2 from #3}{
    at=(#3.#1), anchor=#1+180, shift=(#1:#2)
}]
  \node[vertex] (v1) {};
  \node[vertex,position = 0:{\nodeDist} from v1] (v2) {};
  \node[vertex,position = 0:{\nodeDist} from v2] (v3) {};
  \coordinate[position = 0:{\nodeDist} from v3] (v4);
  \coordinate[position = 0:{\nodeDist} from v4] (v5) {};
  \node[vertex,position = 0:{\nodeDist} from v5] (v6) {};

  \draw (v1) node[below = 13] {$P_a$};
  \draw (v2) node[below = 13] {$P_{a+j}$};
  \draw (v3) node[below = 13] {$P_{a+j\pm s_{r}}$};
  \draw (v6) node[below = 13] {$P_{a+j\pm s}$};

  \draw[edge] (v1) to [bend left=5] node[below] {$e_{k+1}$} (v2);

  \draw[edge] (v2) to [bend left=10] node[above] {$e_{\pm r}$} (v3);
  \draw[edge] (v3) to [bend left=15] node[below] {$e_{\mp r}$} (v2);

  \draw[edge] (v3) to [bend left=10] node[above] {$e_{\pm (r-1)}$} (v4);
  \draw[edge] (v4) to [bend left=15] node[below] {$e_{\mp (r-1)}$} (v3);

  \draw[dashed,shorten > = 5pt, shorten < = 5pt] (v4) -- (v5);

  \draw[edge] (v5) to [bend left=10] node[above] {$e_{\pm (t+1)}$} (v6);
  \draw[edge] (v6) to [bend left=15] node[below] {$e_{\mp (t+1)}$} (v5);
\end{tikzpicture}\enskip .
\]
where $s=s_{t+1}+\dots+s_r$.
Since $a$ does not lie in $\supp(H_t)$, any Eulerian path from $P_a$ on $G_a$ is of the form 
\[(e_{k+1},e_{\pm r},\dots,e_{\pm (t + 1)},w,e_{\mp (t+1)},\dots,e_{\mp r}), \] 
where $w$ is an Eulerian path on $H_t$. 
By Lemma \ref{diag_Eulerian_sum_1}, the sum of the signatures of these Eulerian paths is $\pm(t-1)!$ or $\pm t!$, depending on
whether $P_{a+j+s_{t+1}+\dots+s_{r}} = P_{0}$ or not.

\smallskip
Case 2. $P_{a}\succ P_{a+j}$ and $a+ j$ does not lie in $\supp(H_t)$. (Once again, if $a + j$ lies in $\supp(H_t)$, then we are in Case 3.)
By Proposition \ref{max_graph_decomp}, $G_a=H_t\ \CMamalg\ G_a'$, where $G_a'$ is of the form
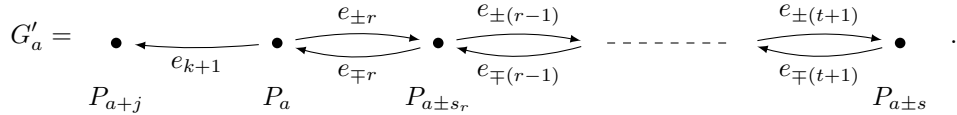
\[G_a'=
\begin{tikzpicture}[baseline,vertex/.style={anchor=base, circle, fill,minimum size=4pt, inner sep=0pt, outer sep=0pt},auto,
                               edge/.style={->,>=latex, shorten > = 5pt,shorten < = 5pt},
position/.style args={#1:#2 from #3}{
    at=(#3.#1), anchor=#1+180, shift=(#1:#2)
}]
  \node[vertex] (v1) {};
  \node[vertex,position = 0:{\nodeDist} from v1] (v2) {};
  \node[vertex,position = 0:{\nodeDist} from v2] (v3) {};
  \coordinate[position = 0:{\nodeDist} from v3] (v4);
  \coordinate[position = 0:{\nodeDist} from v4] (v5) {};
  \node[vertex,position = 0:{\nodeDist} from v5] (v6) {};

  \draw (v1) node[below = 13] {$P_{a+j}$};
  \draw (v2) node[below = 13] {$P_{a}$};
  \draw (v3) node[below = 13] {$P_{a\pm s_{r}}$};
  \draw (v6) node[below = 13] {$P_{a\pm s}$};

  \draw[edge] (v2) to [bend left=5] node[below] {$e_{k+1}$} (v1);

  \draw[edge] (v2) to [bend left=10] node[above] {$e_{\pm r}$} (v3);
  \draw[edge] (v3) to [bend left=15] node[below] {$e_{\mp r}$} (v2);

  \draw[edge] (v3) to [bend left=10] node[above] {$e_{\pm (r-1)}$} (v4);
  \draw[edge] (v4) to [bend left=15] node[below] {$e_{\mp (r-1)}$} (v3);

  \draw[dashed,shorten > = 5pt, shorten < = 5pt] (v4) -- (v5);

  \draw[edge] (v5) to [bend left=10] node[above] {$e_{\pm (t+1)}$} (v6);
  \draw[edge] (v6) to [bend left=15] node[below] {$e_{\mp (t + 1)}$} (v5);
\end{tikzpicture}\enskip .
\]
where $s=s_{t+1}+\dots+s_r$. Again, since $a+j$ does not lie in $\supp(H_t)$, the sum of the signatures of the Eulerian paths 
will be determined by the sum of the signatures of the Eulerian paths on $H_t$, being either $\pm (t-1)!$ or $\pm t!$.

\smallskip
Case 3. Both $a$ and $a+j$ lie in $\supp(H_t)$. By the definition of $t$, this forces $t=r$. By our assumptions on $a$ and $j$ 
we have $a \neq 0$, $j \neq 0$ and $a + j\neq 0$. Thus $G_a$ is of the same form as the graph considered in Lemma \ref{diag_Eulerian_sum_2}, 
where we showed there that $\sum_{w\in \uni_{P_a}(G_a)}\sgn(w)=\pm (r-1)!$.
\end{proof}

\begin{proof}[Proof of Proposition~\ref{Ic_L_j_ker_j_neq_0}] Let $G$ be the maximal graph in $U(a, j)$. The key point is that
since $\operatorname{char}(F)$ does not divide $r!$, Lemma~\ref{nonzero_diagonal} tells us 
under our assumptions on $a$ and $j$, $G \in U_{nz}(a, j)$. In other words, $G$ is the maximal graph in $U_{nz}(a, j)$. Thus
by Lemma~\ref{Ic_entries}, $(\Ic(L_j))_{a, b}$ is the sum of the signatures of the Eulerian paths on $G_b$ from $P_a$ to
$P_{a+j}$. If $b \neq a$, then Lemma~\ref{only_diagonal} tells us that there are no such paths, so $(\Ic(L_j))_{a, b} = 0$.
On the other hand, $(\Ic(L_j))_{a, a} \neq 0$ in $F$ by Lemma~\ref{nonzero_diagonal}.
\end{proof}

\begin{remark}
In those cases where the pair $(a, j)$ does not safisfy the conditions of Proposition~\ref{Ic_L_j_ker_j_neq_0},
the maximal graph $G$ of $U(a, j)$ does not lie in $U_{nz}(a,j)$. If it did, the nullity of $\In(L_j)$ would be lower than the value given by
Proposition~\ref{Ic_L_j_ker} for at least one $j$. In view of Lemma~\ref{Ic_det}, the nullity of $L = L_0 \oplus L_1 \oplus \ldots \oplus L_{n-1}$ 
would be lower than $k$, contradicting Lemma~\ref{lem.zariski}(b).

To illustrate this point more concretely, let us revisit Example~\ref{Ic_eg}. Here $k=2$ (so, $r = 1$), $n=3$, and $j = 1$.
The maximal graph from $U(a,1)$ does not lie in $U_{nz}(a,1)$ when $a=2$, though it does when $a=0$ or $1$. 
After removing the row and column corresponding to $a=2$ (i.e., the last row and the last column) 
from the $3 \times 3$ matrix $\Ic(L_1)$ we are
left with a nonsingular diagonal $2 \times 2$ matrix, showing that $\nullity(\Ic(L_1))=1$.
\end{remark}

\section{Proof of Proposition~\ref{Ic_L_j_ker}(b)}
\label{section_j_0}

\newcommand{\LT}{\mathcal{L}} 
\newcommand{\UT}{\mathcal{U}} 

As we showed in the previous section, for $j \neq 0$ in $\bbZ/ n \bbZ$ the matrix $\Ic(L_j)$ is close to being diagonal. In this section we will see that
$\Ic(L_0)$ has a more complicated structure. 
For ease of visualizing the matrix $\Ic(L_0)$, we will reorder the rows and columns. In the matrix $\Ic(L_0)$, our rows and columns corresponding to $a$ and $b$ respectively range from $0$ to $n-1$. We define the $n\times n$ matrix $\Ic(L_0)'$ to be the matrix $\Ic(L_0)$ but with the rows and columns corresponding to $a$ and $b$ ranging from $-\lceil \dfrac{n-1}{2}\rceil$ to $\lfloor \dfrac{n-1}{2} \rfloor$. Since permuting rows and columns does not effect the nullity of a matrix, 
we have 
\[ \nullity(\Ic(L_0)')=\nullity(\Ic(L_0)). \]
Thus our goal is to show that $\Ic(L_0)'$ is a non-singular matrix.
We will do this by proving that $\Ic(L_0)'$ is of the form
\begin{equation}\label{Ic_L_0}
\Ic(L_0)'=\left(
\begin{array}{ccc}
\UT & \ast &{\bf 0}\\
{\bf 0}& N &{\bf 0}\\
{\bf 0}& \ast & \LT
\end{array}
\right),
\end{equation}
where 

\begin{itemize}
\item
$\UT$, $N$ and $\LT$ are square matrices,

\item
$\UT$ is upper triangular with non-zero diagonal matrices, 

\item
$\LT$ is lower triangular with non-zero diagonal entries, and

\item
$N$ is the $(r+1)\times (r+1)$ submatrix corresponding to rows and columns 
labeled by elements of $\supp(H_r) = \{ -\lceil r/2 \rceil,-\lceil r/2 \rceil+1,\dots,\lfloor r/2 \rfloor \}$,

\item
$N$ is non-singular.
\end{itemize}

This will imply that $\Ic(L_0)'$ is a non-singular matrix and hence, so is $\Ic(L_0)$, thus completing the proof of 
Proposition~\ref{Ic_L_j_ker}(b). Proofs of these assertions will be carried out in 
Lemmas~\ref{upper_triang_part},~\ref{lower_triang_part},~\ref{N_lemma} and~\ref{N_nonsingular}.
The idea is to read off the entries of $\Ic(L_0)'$ from Proposition~\ref{max_graph_decomp} using 
their graph-theoretic interpretation given by Lemma~\ref{Ic_entries}.

\begin{lemma}\label{upper_triang_part}
The matrix consisting of the first $\lceil \dfrac{n-1}{2}\rceil-\lceil \dfrac r 2\rceil$ rows of $\Ic(L_0)'$ has the structure
\[\left(\begin{array}{ccc} \UT&\ast&{\bf 0}\end{array}\right)=
\left(\begin{array}{cccccccccccccc}
\ast     & \ast     & \ast     & \dots    & \ast    & \ast     & \ast      & \ast    & \dots   & \ast     & 0          & 0        & \dots & 0\\
0         & \ast     & \ast     & \dots    & \ast    & \ast     & \ast      & \ast    & \dots   & \ast     & 0          & 0        & \dots & 0\\
0         & 0         & \ast     & \dots    & \ast    & \ast     & \ast      & \ast    & \dots   & \ast     & 0          & 0         & \dots & 0\\
\vdots & \vdots & \vdots & \ddots & \vdots & \vdots & \vdots & \vdots &            & \vdots & \vdots & \vdots &          & \vdots\\
0         & 0         & 0         & \dots    & \ast    & \ast     & \ast      & \ast     & \dots   & \ast     & 0         & 0         & \dots & 0\\
0         & 0         & 0         & \dots    & 0        & \ast     & \ast      & \ast     & \dots   & \ast      & 0         & 0        & \dots & 0\\
\end{array}\right)\ .\]
The first $\lceil\frac{n-1}{2}\rceil-\lceil\frac r 2\rceil$ columns form an upper triangular block $\UT$, where the terms on the diagonal are $\pm 2\alpha!$ for some $\alpha\leqslant r$. The last $\lfloor\frac{n-1}2\rfloor-\lfloor\frac r 2\rfloor$ columns are zero.
\end{lemma}

\begin{proof}

By Lemma \ref{Ic_entries}, it suffices to prove the following.

\smallskip

Fix $-\lceil \dfrac{n-1}{2} \rceil \leqslant a\leqslant -\lceil \dfrac{r}{2} \rceil-1$. Let $G$ be the maximal graph in $U(a,0)$ and $b\in \Z/n\Z$ be such that there exists an Eulerian path from $P_a$ to $P_a$ on $G_b=G\ \CMamalg\ \underset {P_b} \bullet \underset {e_{k+1}} \rightarrow \underset {P_{b}} \bullet$. 
Then 

\smallskip
(a) $b\in\{a \, , \, a+1, \, \ldots, \, -1, 0, 1, \ldots, \, \lfloor \dfrac{r}{2} \rfloor - 1, \, \lfloor \dfrac{r}{2} \rfloor\}$. 

\smallskip
(b) the sum of the signatures of the Eulerian paths on $G_a$ from $P_a$ to $P_a$ is $\pm 2\alpha!$ for some $\alpha\leqslant r$.

\smallskip
Recall that by Proposition~\ref{max_graph_decomp}(a), $G=H_t\ \CMamalg\ G'$ 
for some $0 \leqslant t\leqslant r$, where $G'$ is of the form
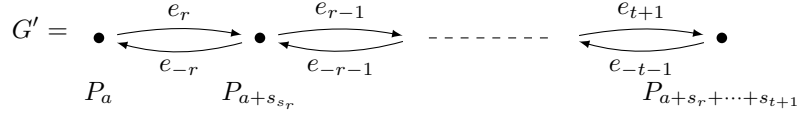
\[G'=
\begin{tikzpicture}[baseline,vertex/.style={anchor=base, circle, fill,minimum size=4pt, inner sep=0pt, outer sep=0pt},auto,
                               edge/.style={->,>=latex, shorten > = 5pt,shorten < = 5pt},
position/.style args={#1:#2 from #3}{
    at=(#3.#1), anchor=#1+180, shift=(#1:#2)
}]
  \node[vertex] (v1) {};
  \node[vertex,position = 0:{\nodeDist} from v1] (v2) {};
  \coordinate[position = 0:{\nodeDist} from v2] (v3);
  \coordinate[position = 0:{\nodeDist} from v3] (v4) {};
  \node[vertex,position = 0:{\nodeDist} from v4] (v5) {};

  \draw (v1) node[below = 13] {$P_a$};
  \draw (v2) node[below = 13] {$P_{a+s_{s_r}}$};
  \draw (v5) node[below = 13] {$P_{a+s_{r}+\dots+s_{t + 1}}$};

  \draw[edge] (v1) to [bend left=10] node[above] {$e_{r}$} (v2);
  \draw[edge] (v2) to [bend left=15] node[below] {$e_{-r}$} (v1);

  \draw[edge] (v2) to [bend left=10] node[above] {$e_{r - 1}$} (v3);
  \draw[edge] (v3) to [bend left=15] node[below] {$e_{-r - 1}$} (v2);

  \draw[dashed,shorten > = 5pt, shorten < = 5pt] (v3) -- (v4);

  \draw[edge] (v4) to [bend left=10] node[above] {$e_{t + 1}$} (v5);
  \draw[edge] (v5) to [bend left=15] node[below] {$e_{-t - 1}$} (v4);
\end{tikzpicture}
\]
and $s_{r}+s_{r-1}+\cdots s_{t + 1} \leqslant |a|$. The graph $G_b$ is obtained from $G$ by appending the extra edge $e_{k+1}$ at $P_b$.
Since $j = 0$, $e_{k+1}$ is a loop. If $G_b$ has an Eulerian path, it has to be connected. In other words, the loop $e_{k+1}$ has to be appended at
one of the vertices that has non-zero degree in $G$, i.e., at $P_a$, $P_{a + s_r}$, $\ldots$, $P_{a + s_r + \ldots + s_{t + 1}}$ or at
one the vertices from $H_t$. This shows that $b\in [a \interval \lfloor r/2 \rfloor]$, thus proving (a). 

To prove (b), note that that the graph $G_a$ is of the form
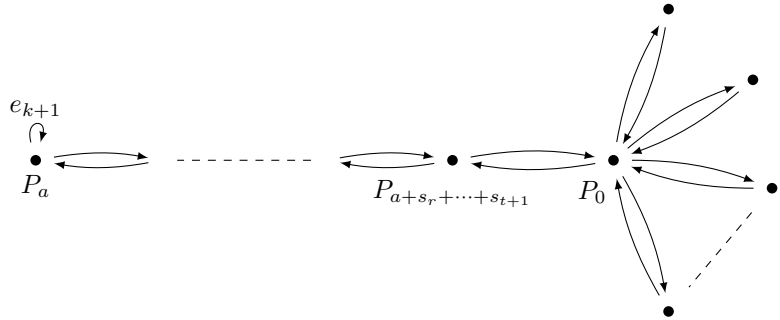
\[
\begin{tikzpicture}[baseline,vertex/.style={anchor=base, circle, fill,minimum size=4pt, inner sep=0pt, outer sep=0pt},auto,
                               edge/.style={->,>=latex, shorten > = 5pt,shorten < = 5pt},
position/.style args={#1:#2 from #3}{
    at=(#3.#1), anchor=#1+180, shift=(#1:#2)
}]

  \node[vertex] (v0) {};
  \node[vertex,position = 180:{\nodeDist} from v0] (va1) {};
  \coordinate[position = 180:{\nodeDist*0.8} from va1] (va2) {};
  \coordinate[position = 180:{\nodeDist*1.1} from va2] (va3) {};
  \node[vertex,position = 180:{\nodeDist*0.8} from va3] (va4) {};
  \node[vertex,position = 70:{\nodeDist} from v0] (v1L) {};
  \node[vertex,position = 30:{\nodeDist} from v0] (v2L) {};
  \node[vertex,position = -10:{\nodeDist} from v0] (v3L) {};
  \node[vertex,position = -70:{\nodeDist} from v0] (v4L) {};

  \draw (v0) node[below=12, left=-1] {$P_0$};
  \draw (va4) node[below=2] {$P_{a}$};
  \draw (va1) node[below=4] {$P_{a+s_{r}+\dots+s_{t + 1}}$};
  \draw (v1L) node[above] {};
  \draw (v2L) node[left] {};
  \draw (v3L) node[left] {};
  \draw (v4L) node[below] {};

  \draw[edge] (va4) to [loop above] node[above] {$e_{k+1}$} (va4);

  \draw[edge] (v0) to [bend left=10] (va1);
  \draw[edge] (va1) to [bend left=10] (v0);

  \draw[edge] (va1) to [bend left=10] (va2);
  \draw[edge] (va2) to [bend left=10] (va1);

  \draw[dashed,shorten > = 6pt, shorten < = 6pt, bend right=10] (va2) -- (va3);

  \draw[edge] (va3) to [bend left=10] (va4);
  \draw[edge] (va4) to [bend left=10] (va3);

  \draw[edge] (v0) to [bend left=10] node[left] {} (v1L);
  \draw[edge] (v1L) to [bend left=10] node[right] {} (v0);

  \draw[edge] (v0) to [bend left=10] (v2L);
  \draw[edge] (v2L) to [bend left=10] (v0);

  \draw[edge] (v0) to [bend left=10] (v3L);
  \draw[edge] (v3L) to [bend left=10] (v0);

  \draw[dashed,shorten > = 10pt, shorten < = 10pt, bend right=10] (v3L) -- (v4L);

  \draw[edge] (v0) to [bend left=10] (v4L);
  \draw[edge] (v4L) to [bend left=10] (v0);
\end{tikzpicture}
\]
where the (possibly empty) flower on the right corresponds to $H_t$. Here the vertices $P_{a+s_r+\cdots+s_{t+1}}$ and $P_0$ may be distinct
(as in the above diagram) or not (in which case the two edges connecting them should be removed),
depending on whether $G'$ intersects $H_t$ at $P_0$. In either case, $G_a$ will have the shape of the graph in Lemma~\ref{Eulerian_sum_2}. By
Lemma~\ref{Eulerian_sum_2},
\[\sum_{w\in \uni_{P_a}(G_a)}\sgn(w)=\pm 2\alpha!\ ,\]
where $\alpha$ is the number of vertices to the right of $P_0$. As $G_a$ has $2r+1$ edges, there are at most $r$ vertices to the right of $P_0$ in the picture and hence $\alpha\leqslant r$.
\end{proof}

\begin{lemma}\label{lower_triang_part}
The matrix consisting of the last $\lfloor\frac{n-1}2\rfloor-\lfloor\frac r 2\rfloor$ rows of $\Ic(L_0)'$ has the structure
\[\left(\begin{array}{ccc} {\bf 0}&\ast& \LT\end{array}\right)=
\left(\begin{array}{cccccccccccccc}
0         & 0         & \dots & 0         & \ast     & \ast     & \dots & \ast    & \ast    & 0          & 0         & \dots   & 0         & 0\\
0         & 0         & \dots & 0         & \ast     & \ast     & \dots & \ast    & \ast    & \ast      & 0         & \dots   & 0         & 0\\
0         & 0         & \dots & 0         & \ast     & \ast     & \dots & \ast    & \ast    & \ast      & \ast     & \dots   & 0         & 0\\
\vdots & \vdots &          & \vdots & \vdots & \vdots &         & \vdots & \vdots & \vdots & \vdots & \ddots & \vdots & \vdots\\
0         & 0         & \dots & 0          & \ast    & \ast     & \dots & \ast     & \ast    & \ast     & \ast     & \dots    & \ast    & 0\\
0         & 0         & \dots & 0          & \ast    & \ast     & \dots & \ast     & \ast    & \ast     & \ast     & \dots    & \ast    & \ast\\
\end{array}\right)\ .\]
The last $\lfloor\frac{n-1}2\rfloor-\lfloor\frac r 2\rfloor$ columns form a lower triangular block $\LT$ where the terms on the diagonal are $\pm 2\alpha!$ for some $\alpha\leqslant r$. The first $\lceil\frac{n-1}2\rceil-\lceil\frac r 2\rceil$ columns are zero.
\end{lemma}

\begin{proof}
The proof mirrors that of Lemma \ref{upper_triang_part}, except that here we appeal to Proposition~\ref{max_graph_decomp}(b) for the
the structure of the maximal graph $G$ instead of Proposition~\ref{max_graph_decomp}(a).
\end{proof}

Lemmas \ref{upper_triang_part} and \ref{lower_triang_part} show that the upper and lower portions of $\Ic(L_0)'$ have the structure as in (\ref{Ic_L_0}), and that $\UT$ and $\LT$ are nonsingular under our characteristic assumption. It remains to consider the rows corresponding to $a\in[-\lceil r/2\rceil\interval \lfloor r/2\rfloor]$. By Proposition \ref{max_graph_decomp} the maximal graph in $U(a,0)$ for $a$ in this range is $H_r$.

\begin{lemma}\label{N_lemma}
The middle rows of $\Ic(L_0)'$ corresponding to $a\in [-\lceil r/2\rceil\interval \lfloor r/2\rfloor]$ are
\[\left(\begin{array}{ccc}{\bf 0}&N&{\bf 0}\end{array}\right)\ ,\]
up to multiplication of each column by $\pm 1$, where $N$ is the $(r +1) \times (r + 1)$ matrix
\item 
\begin{equation} 
N = (r-1)!\cdot\left(
\begin{array}{ccccccccccc}
2&1&1&\cdots&1&r&1&\cdots&1&1&1\\
1&2&1&\cdots&1&r&1&\cdots&1&1&1\\
1&1&2&\cdots&1&r&1&\cdots&1&1&1\\
\vdots&\vdots&\vdots&\ddots&\vdots&\vdots&\vdots&&\vdots&\vdots&\vdots\\
1&1&1&\cdots&2&r&1&\cdots&1&1&1\\
r&r&r&\cdots&r&r(r+1)&r&\cdots&r&r&r\\
1&1&1&\cdots&1&r&2&\cdots&1&1&1\\
\vdots&\vdots&\vdots&&\vdots&\vdots&\vdots&\ddots&\vdots&\vdots&\vdots\\
1&1&1&\cdots&1&r&1&\cdots&2&1&1\\
1&1&1&\cdots&1&r&1&\cdots&1&2&1\\
1&1&1&\cdots&1&r&1&\cdots&1&1&2\\
\end{array}
\right).
\end{equation}
\end{lemma}

\begin{proof}
Let $a\in[-\lceil r/2\rceil\interval \lfloor r/2\rfloor]$. The $(a,b)^{\rm th}$ entry of $\Ic(L_0)'$ is 
$\sum_{w\in \uni_{P_a}(G_b)}\sgn(w)$, where $G_b$ is $H_r$ with an additional edge $e_{k+1}$, being a loop, 
placed at $P_b$. For the graph $G_b$ to be connected, we require $b\in [-\lceil r/2\rceil\interval \lfloor r/2\rfloor]$. This shows that the $(a,b)^{\rm th}$ entry of $\Ic(L_0)'$ is zero when $b\notin[-\lceil r/2\rceil\interval \lfloor r/2\rfloor]$. 

Assume $b\in [-\lceil r/2\rceil\interval \lfloor r/2\rfloor]$. Lemma \ref{Eulerian_sum_1} with $\alpha=r$ gives us

\[N_{(a,b)}=\sum_{w\in \uni_{P_a}(G_b)}\sgn(w)=\begin{cases}
\pm (r+1)!\ ,\ &\text{ if $a=0$ and $b=0$}\\
\pm 2(r-1)!\ ,\ &\text{ if $a=b$ and $a,b\neq 0$}\\
\pm r!\ ,\ &\text{ if $a=0$ and $b\neq 0$, or $a\neq 0$ and $b=0$}\\
\pm (r-1)!\ ,\ &\text{ otherwise,}
\end{cases}\]
where for fixed $b$, either all entries $N_{(a,b)}$ are positive, or all entries $N_{(a,b)}$ are negative. This shows that $N$ has the required form.
\end{proof}

\begin{lemma}\label{N_nonsingular}
Suppose the characteristic of $F$ does not divide $(2r+1)r!$. Then $\det(N) \neq 0$ in $F$.
\end{lemma}

\begin{proof}
We perform elementary row operations on the matrix $N$. After dividing $N$ by $(r-1)!$ and dividing
the row corresponding to $0$ by $r$, we obtain

\[\left(
\begin{array}{ccccccccccc}
2&1&1&\cdots&1&r&1&\cdots&1&1&1\\
1&2&1&\cdots&1&r&1&\cdots&1&1&1\\
1&1&2&\cdots&1&r&1&\cdots&1&1&1\\
\vdots&\vdots&\vdots&\ddots&\vdots&\vdots&\vdots&&\vdots&\vdots&\vdots\\
1&1&1&\cdots&2&r&1&\cdots&1&1&1\\
1&1&1&\cdots&1&r+1&1&\cdots&1&1&1\\
1&1&1&\cdots&1&r&2&\cdots&1&1&1\\
\vdots&\vdots&\vdots&&\vdots&\vdots&\vdots&\ddots&\vdots&\vdots&\vdots\\
1&1&1&\cdots&1&r&1&\cdots&2&1&1\\
1&1&1&\cdots&1&r&1&\cdots&1&2&1\\
1&1&1&\cdots&1&r&1&\cdots&1&1&2\\
\end{array}
\right)\ .\]
Next we subtract the row corresponding to $0$ from every other row. Our matrix is transformed to
\[\left(
\begin{array}{ccccccccccc}
1&0&0&\cdots&0&-1&0&\cdots&0&0&0\\
0&1&0&\cdots&0&-1&0&\cdots&0&0&0\\
0&0&1&\cdots&0&-1&0&\cdots&0&0&0\\
\vdots&\vdots&\vdots&\ddots&\vdots&\vdots&\vdots&&\vdots&\vdots&\vdots\\
0&0&0&\cdots&1&-1&0&\cdots&0&0&0\\
1&1&1&\cdots&1&r+1&1&\cdots&1&1&1\\
0&0&0&\cdots&0&-1&1&\cdots&0&0&0\\
\vdots&\vdots&\vdots&&\vdots&\vdots&\vdots&\ddots&\vdots&\vdots&\vdots\\
0&0&0&\cdots&0&-1&0&\cdots&1&0&0\\
0&0&0&\cdots&0&-1&0&\cdots&0&1&0\\
0&0&0&\cdots&0&-1&0&\cdots&0&0&1\\
\end{array}
\right)\ .\]
Finally, we subtract every other row from the row corresponding to $0$. Our matrix becomes
\[\left(
\begin{array}{ccccccccccc}
1&0&0&\cdots&0&-1&0&\cdots&0&0&0\\
0&1&0&\cdots&0&-1&0&\cdots&0&0&0\\
0&0&1&\cdots&0&-1&0&\cdots&0&0&0\\
\vdots&\vdots&\vdots&\ddots&\vdots&\vdots&\vdots&&\vdots&\vdots&\vdots\\
0&0&0&\cdots&1&-1&0&\cdots&0&0&0\\
0&0&0&\cdots&0&2r+1&0&\cdots&0&0&0\\
0&0&0&\cdots&0&-1&1&\cdots&0&0&0\\
\vdots&\vdots&\vdots&&\vdots&\vdots&\vdots&\ddots&\vdots&\vdots&\vdots\\
0&0&0&\cdots&0&-1&0&\cdots&1&0&0\\
0&0&0&\cdots&0&-1&0&\cdots&0&1&0\\
0&0&0&\cdots&0&-1&0&\cdots&0&0&1\\
\end{array}
\right)\ ,\]
which has determinant $2r+1\neq 0$. Thus $N$ is nonsingular under our assumption on the characteristic.
\end{proof}

	\section*{Acknowledgments} 
The authors are grateful to Omer Angel, John Dixon, Irwin Pressman and Bruce Shepherd for helpful comments.


\end{document}